\documentclass[reqno,10pt]{amsart}

\usepackage{style}

\renewcommand{\C}{{\mathcal{C}}}
\newcommand{\D}{\mathcal{D}}
\newcommand{\DS}{\mathcal{DS}}

\newcommand{\partialvis}{\partial_\text{vis}}
\DeclareMathOperator{\Bis}{Bis}

\DeclareMathOperator{\vv}{{\mathrm{v}}}
\DeclareMathOperator{\type}{type}
\DeclareMathOperator{\Hull}{Hull}
\DeclareMathOperator{\Diag}{Diag}

\DeclareMathOperator{\asym}{slope}
\DeclareMathOperator{\Th}{Th}
\DeclareMathOperator{\Ker}{ker}
\DeclareMathOperator{\Thomega}{{Th}^{\omega}_{\rm flag}} 
\DeclareMathOperator{\ThomegaH}{{Th}^{\omega}_{\rm horo}}

\newcommand{\Omegaom}{\Omega^{\omega}_{\rm flag}} 
\newcommand{\OmegaomH}{\Omega^{\omega}_{\rm horo}} 
\newcommand{\partialomX}{\partial^{\omega}_{\rm horo} \X}
\newcommand{\partialvisX}{\partial_{\rm vis} \X}

\newcommand{\Selberg}{\mathfrak{s}}

\title{Dirichlet domains for Anosov subgroups}

\author{Colin Davalo}
\address{Mathematisches Institut, Ruprecht-Karls Universität Heidelberg, Im Neuenheimer Feld 205, 69120 Heidelberg, Germany}
\email{cdavalo@mathi.uni-heidelberg.de}

\author{J.\ Maxwell Riestenberg}
\address{Max Planck Institute for Mathematics in the Sciences, Inselstraße 22, 04103 Leipzig, Germany}
\email{max.riestenberg@mis.mpg.de}

\begin{document}

\begin{abstract}
We introduce a sufficient condition for a finitely generated subgroup $\Gamma$ of a semisimple Lie group $G$ to admit finite-sided Dirichlet domains for polyhedral Finsler metrics on the symmetric space $G/K$.
The condition always implies the $\Theta$-Anosov condition for some $\Theta$, and can be arranged to be equivalent to the $\Theta$-Anosov condition when $G$ is simple and $\Theta$ is the set of long roots or the set of short roots.
The Dirichlet domain we obtain extends to a fundamental domain for the action of $\Gamma$ on a domain of discontinuity in a flag manifold. 

For instance, Borel Anosov subgroups of $\SL(d,\R)$ have finite-sided Dirichlet domains for the Hilbert metric on the symmetric space which extends to the space of line-hyperplane flags, and $n$-Anosov subgroups of $\Sp(2n,\R)$ have finite-sided Dirichlet-Selberg domains in $\SL(2n,\R)/\SO(2n)$ which extend to a domain in projective space bounded by quadrics.
\end{abstract}

\maketitle

\section{Introduction}

Let $\Gamma$ be a discrete group acting properly by isometries on a metric space $(X,d)$. 
Such an action can be understood by producing a \emph{fundamental domain}: a connected closed subset meeting every orbit whose interior points meet only one orbit.
The \emph{Dirichlet domain} based at $o\in X$ is the cell containing $o$ in the Voronoï diagram associated with the $\Gamma$-orbit of this point, or in other words the set of points $x\in X$ such that for all $\gamma\in \Gamma$, $d(o,x)\leq d(\gamma\cdot o,x)$.
When $X$ is a proper geodesic metric space, the Dirichlet domain is a fundamental domain \cite{Rat19}.
It is natural to ask when a Dirichlet domain has finitely many sides. We address that question in this paper for certain discrete subgroups acting on higher rank symmetric spaces equipped with various Finsler metrics.

\medskip

Before discussing our results for higher rank symmetric spaces, let us mention what is known in rank one.
For real hyperbolic space $\mathbb{H}^n$, the use of Dirichlet domains to study actions of discrete isometry groups goes back to Poincare \cite{poincare1883memoire}.
In that setting the construction is particularly well-behaved: the bisectors are totally geodesic submanifolds, and the Dirichlet domain is convex.
Moreover in the projective model, the resulting domain is a convex polyhedron: a projectively convex subset of the ball which has a locally finite collection of sides.
For lattices and convex cocompact subgroups in $\Isom(\HH^n)$, every Dirichlet domain has finitely many sides; these are special cases of \emph{geometrically finite} groups. 
A subgroup of $\Isom(\HH^n)$ for $n =2$ or $n=3$ is called geometrically finite if any (equivalently, every) Dirichlet domain has finitely many sides.
The characterization of geometrically finite groups in $\Isom(\HH^n)$ for $n\ge 4$ in terms of their Dirichlet domains is somewhat subtle: there exist subgroups admitting finite-sided Dirichlet domains for some basepoints and not others, see \cite{Rat19} for a detailed discussion. 

Characterizing convex cocompact subgroups in higher dimensions in terms of their Dirichlet domains is more straightforward. 
Let $\overline{\mathbb{\HH}^n} = \HH^n \cup \partial_{\rm vis} \HH^n$ be the \emph{visual compactification} of $\HH^n$. 
To each point of $\partial_{\rm vis} \HH^n$, one can associate a Busemann function defined up to an additive constant.
Let $\mathcal{D}_\Gamma(o)$ denote the closure in $\overline{\mathbb{\HH}^n}$ of the Dirichlet domain in $\HH^n$. 
The intersection $\mathcal{D}_\Gamma(o) \cap \partial_{\rm vis}\HH^n$ consists precisely of the points whose associated Busemann function $b$ restricted to $\Gamma\cdot o$ has $o$ as a minimum.
We say that a Dirichlet domain is \emph{properly finite-sided} if there exist a neighborhood $U$ of $\mathcal{D}_\Gamma(o)$ that is included for all except finitely many $\gamma$ in $\Gamma$ in the closure in $\overline{\mathbb{\HH}^n}$ of the half-space:
$$\mathcal{H}(o,\gamma \cdot o) \coloneqq \lbrace x\in \mathbb{H}^n \mid d(o,x)\leq d(\gamma\cdot o,x)\rbrace \cup \lbrace [b] \in \partial_{\rm vis} \mathbb{H}^n \mid b(o) \leq b(\gamma\cdot o)\rbrace .$$
A subgroup $\Gamma$ is convex cocompact if and only some (and hence any) of its Dirichlet domains is properly finite-sided. 
Moreover, $\mathcal{D}_\Gamma(o) \cap \partial_{\rm vis} \HH^n$ is a fundamental domain for the action of $\Gamma$ on its \emph{domain of discontinuity} in $\partial_{\rm vis} \mathbb{H}^n$.
Accordingly, the Busemann functions $[b] \in \partial_{\rm vis} \HH^n$ satisfy a dichotomy: either $[b]$ belongs to the limit set and then goes to $-\infty$ along some sequence in the orbit, or $[b]$ is proper and bounded below on the whole orbit $\Gamma \cdot o$. 
This whole discussion applies for any negatively curved symmetric space. 
Our results can be viewed as a generalization of this characterization and its proof to higher rank symmetric spaces. 

\medskip

In this paper we study the case when $\Gamma$ is an \emph{Anosov subgroup} of a semisimple Lie group $G$ acting on the associated symmetric space $\X=G/K$.  
Anosov subgroups are regarded as the natural generalization of convex cocompact subgroups of rank 1 Lie groups to higher rank.
An Anosov subgroup is the image of an Anosov representation; these were introduced by Labourie \cite{Lab06} for fundamental groups of closed negatively curved manifolds and Guichard-Wienhard \cite{GW12} for word hyperbolic groups in general.

\medskip

The symmetric space $\X$ has a natural $G$-invariant Riemannian metric, uniquely determined up to rescaling by a constant on each isometric factor. 
But in higher rank, Dirichlet domains for this Riemannian metric differ too drastically from those for rank $1$ symmetric spaces to have a chance at mimicking the characterization of convex cocompact subgroups. 
To see the issue, suppose that there is a loxodromic element $\gamma$ of $\Gamma$ translating along some axis $c \colon \R \to \X$. 
In higher rank, one can find a Busemann function which is constant along $c$: just take any perpendicular direction in any maximal flat containing $c$. 
Such a Busemann function violates the dichotomy phenomenon we saw in rank $1$. 

Instead, it is fruitful to consider $G$-invariant \emph{polyhedral Finsler metrics} on $\X$, i.e.\ $G$-invariant Finsler metrics whose restriction to a maximal flat is polyhedral.
This gives us the flexibility to control the shapes of horospheres in a manner compatible with the asymptotics of $\Gamma$ (i.e.\ the limit cone).
 Polyhedral Finsler metrics on symmetric spaces have been studied recently by various authors \cite{KL18a,HSWW,Lemmens2023symmetriccones}.

\subsection{Undistorted subgroups}

Given a linear form $\omega$ on the model restricted Cartan subalgebra $\mathfrak{a}$ of the Lie algebra $\mathfrak{g}$ of $G$, one can define a (possibly asymmetric) norm on $\mathfrak{a}$ invariant by the Weyl group $W$ by setting for $v\in \mathfrak{a}$:
$$|v|_\omega=\max_{w\in W} \omega(wv).$$
Such a $W$-invariant norm on $\mathfrak{a}$ defines a $G$-invariant Finsler metric on $G$, which we call a polyhedral metric since the unit ball of the norm $|\cdot|_\omega$ on $\mathfrak{a}$ is a polyhedron. 

Unlike the Riemannian metric, this metric $d_\omega$ has the remarkable property that typical Riemannian geodesics cannot intersect level sets of distance functions in a nontrivial interval. 
So we can consider a rich class of quasi-isometrically embedded subgroups whose orbits uniformly avoid such level sets.
The problematic directions in $\mathfrak{a}$ are those $v\in \mathfrak{a}$ satisfying $\omega (w\cdot v)=0$ for some $w\in W$.
This motivates the following definition:

\begin{definition}[$\omega$-undistorted subgroups]
	We say that a finitely generated subgroup
	$\Gamma < G$ is \emph{$\omega$-undistorted} if there exists a word metric on $\Gamma$ and constants $A,B >0$ such that for every $\gamma \in \Gamma$ and $w$ in the Weyl group:
	\begin{equation}
		\label{eq:omega-URU equation intro}
		\abs{\omega\left(w\cdot \vec{d}(o,\gamma\cdot o)\right)} \geq A \abs{\gamma}-B.
	\end{equation}
	Here $\vec{d}$ stands for the vector valued distance on the symmetric space $\X$.
	Note that $\vec{d}(o,g \cdot o)$ agrees with the Cartan projection of $g \in G$ for the basepoint $o \in \X$.
\end{definition}

In rank 1 symmetric spaces, every $G$-invariant Finsler metric is the unique (up to scaling) invariant Riemannian metric, so in that case $\omega$-undistorted subgroups are the same as convex cocompact subgroups.
In higher rank, $\omega$-undistorted subgroups are a strict subclass of those groups whose orbit maps are quasi-isometric embeddings with respect to $d_\omega$. 
Indeed, the metric $d_\omega$ is biLipschitz to the Riemannian metric, so has the same quasi-isometric embeddings. Rather, the condition tells us that diverging sequences in the orbit also diverge linearly away from a natural class of hypersurfaces through the basepoint. 

\medskip

The definition of $\omega$-undistorted subgroups bears a striking similarity to the definition of $\Theta$-Anosov subgroups, see Definition \ref{def: Anosov}.
If $\Gamma$ is $\omega$-undistorted and not virtually cyclic, then it is $\Theta$-Anosov for some $\Theta$, see Proposition \ref{prop:omega-undistorted implies Anosov}.
However, the precise set $\Theta$ is not determined only by $\omega$, see Example \ref{example: sp6r}. 
On the other hand, when $G$ is simple and $\Theta$ is the set of long simple roots (resp.\ the set of short simple roots), we may set $\omega_\Theta$ to be any long (resp.\ short) root, and in this case the $\Theta$-Anosov subgroups are precisely $\omega_\Theta$-undistorted subgroups, see Remark \ref{rmk: weyl group orbits}.

\medskip

In order to make sense of the notion of properly finite-sided Dirichlet domains, we need to work in the appropriate compactification of the symmetric space $\X$, which is given by the \emph{horofunction compactification} $\overline{\X}^\omega$ of $\X$ with respect to the Finsler metric $d_\omega$, see Section \ref{Sec:Finsler metrics and compactifications}. 
This is defined by embedding $\X$ into the space of functions $\X\to \R$ modulo constant functions via the map $x\in \X \mapsto [d_\omega(\cdot,x)]$, equipped with the quotient of the compact-open topology. 
With this topology the image of $\X$ is relatively compact so $\overline{\X}^\omega=\X\cup\partial_{\rm horo}^\omega\X$ is defined to be the closure of the image of $X$ with respect to this embedding. 
Elements of $\partial_{\rm horo}^\omega\X$ are called \emph{horofunctions}.

\medskip

We prove that {$\omega$-undistorted subgroups} have finite-sided Dirichlet domains $\mathcal{D}^\omega_\Gamma(o)$ in the same strong sense as in rank 1.

\begin{theorem}[{Theorem \ref{thm:omega-URU implies properly finite-sided}}]
	Let $\Gamma<G$ be an {$\omega$-undistorted} subgroup, and let $o \in \X$.
	The Dirichlet domain $\mathcal{D}_\Gamma^\omega(o)$ is \emph{properly finite-sided}, i.e.\ there exist a neighborhood of the Dirichlet domain in $\overline{\X}^\omega$ that meets only finitely many of the bisectors appearing in its definition.
\end{theorem}

As a special case, we note that when $G=\SL(d,\R)$, there is a choice $\omega= \omega_\Delta$ so that the Borel-Anosov condition in $\SL(d,\R)$ is equivalent to being $\omega_\Delta$-undistorted, and the Finsler metric $d_{\omega_\Delta}$ is proportional to the Hilbert metric, see Example \ref{subsubsec:omega URU in SLdR}.
In particular we obtain this corollary as a special case of Theorem \ref{thm:omega-URU implies properly finite-sided}:

\begin{corollary}
		If $\Gamma < \SL(d,\R)$ is Borel-Anosov, then any Dirichlet domain in $\mathcal{X}$ with respect to the Hilbert metric is properly finite-sided.
\end{corollary}

While the converse to Theorem \ref{thm:omega-URU implies properly finite-sided} holds in rank 1 symmetric spaces, it fails in higher rank: uniform lattices have properly finite-sided Dirichlet domains, but cannot be $\omega$-undistorted, since they are not Gromov hyperbolic when the real rank of $G$ is at least $2$.

We show that $\omega$-undistorted subgroups can be characterized as subgroups admitting sufficiently many disjoint bisectors in the symmetric space (Theorem \ref{thm:disjoint half-spaces implies URU}). This characterization leads to a local-to-global principle for $\omega$-undistorted subgroups, and as a corollary, this gives a new proof that the class of $\omega$-undistorted subgroups is stable (sufficiently small deformations of the inclusion representation have image which remains $\omega$-undistorted). The same result has been obtained by Kassel-Tholozan \cite{KasselTholozan}, and also follows from the more general local-to-global principle for Anosov subgroups due to Kapovich-Leeb-Porti \cite{KLP18a}. 

\subsection{Domains of discontinuity}

Guichard-Wienhard \cite{GW12} have shown that Anosov representations give rise to geometric structures via cocompact domains of discontinuity in flag manifolds. The construction has since been systematized by Kapovich-Leeb-Porti \cite{KLP18a} in terms of balanced thickenings of the Weyl group.

We show that for $\omega$-undistorted subgroups, the Dirichlet domains for the distance $d_\omega$ extend at infinity to a fundamental domain $\mathcal{D}_{\rm horo}$ for the action of $\Gamma$ on a cocompact domain of discontinuity $\Omega_{\rm horo}^\omega$ in the horoboundary $\partial_{\rm horo}^\omega \X$. The horoboundary $\partial_{\rm horo}^\omega \X$ compactification contains a certain flag manifold $\mathcal{F}_\omega$ as a closed orbit, and  $\mathcal{D}_{\rm horo}$ intersects $\mathcal{F}_\omega$ in a fundamental domain $\mathcal{D}_{\rm flag}$ for the action of $\Gamma$ on a cocompact domain of discontinuity $\Omega_{\rm flag} \subset \mathcal{F}_\omega$, which is one of the domains constructed by Kapovich-Leeb-Porti, associated to a metric thickening, see Remark \ref{rem:KLP thickenings}.
This reproves the existence of such domains for certain Anosov subgroups and establishes a new characterization of the domains.

\begin{proposition}[Proposition \ref{prop:characterization of domains}]
	An element $[h]\in\partial_{\rm horo}^\omega \X$ belongs to $\Omega_{\rm horo}^\omega$ if and only if $h$ restricted to the $\Gamma$-orbit of $o$ is bounded from below. 
	In this case, the horofunction is proper on any $\Gamma$-orbit.
\end{proposition}

Let us illustrate this result with the special case of $G=\SL(2n,\R)$ and $\omega = \omega_1$, the first fundamental weight, see Example \ref{subsubsec:omega URU in SLdR}.
In this case the flag manifold $\mathcal{F}_{\omega_1}$ is simply the projective space $\mathbb{RP}^{2n-1}$ and every non-elementary $\omega_1$-undistorted subgroup is $n$-Anosov. 
For $\omega_1$-undistorted subgroups we obtain a fundamental domain $\mathcal{D}_{\rm flag} \subset \Omega_{\rm flag}$ of the domain of discontinuity in $\mathbb{RP}^{2n-1}$ originally constructed by Guichard-Wienhard \cite[Section 10.2.5]{GW12}. 

In Figure \ref{fig:Octogon} we consider a surface subgroup of $\SL(2,\R)$ whose fundamental domain in $\mathbb{H}^2$ is a regular octagon, which we embed diagonally into $\SL(4,\R)$. The domain of discontinuity associated to this $\omega_1$-undistorted subgroup in the flag manifold $\mathbb{RP}^3$ is the complement of the black hyperboloid, and we illustrate the eight sides of the fundamental domain associated with the Dirichlet domain for the most natural basepoint in the totally geodesic surface in $\X$ preserved by $\Gamma$. This domain has 2 connected components.

\begin{figure}[h]
	\centering
	\includegraphics[width=380pt]{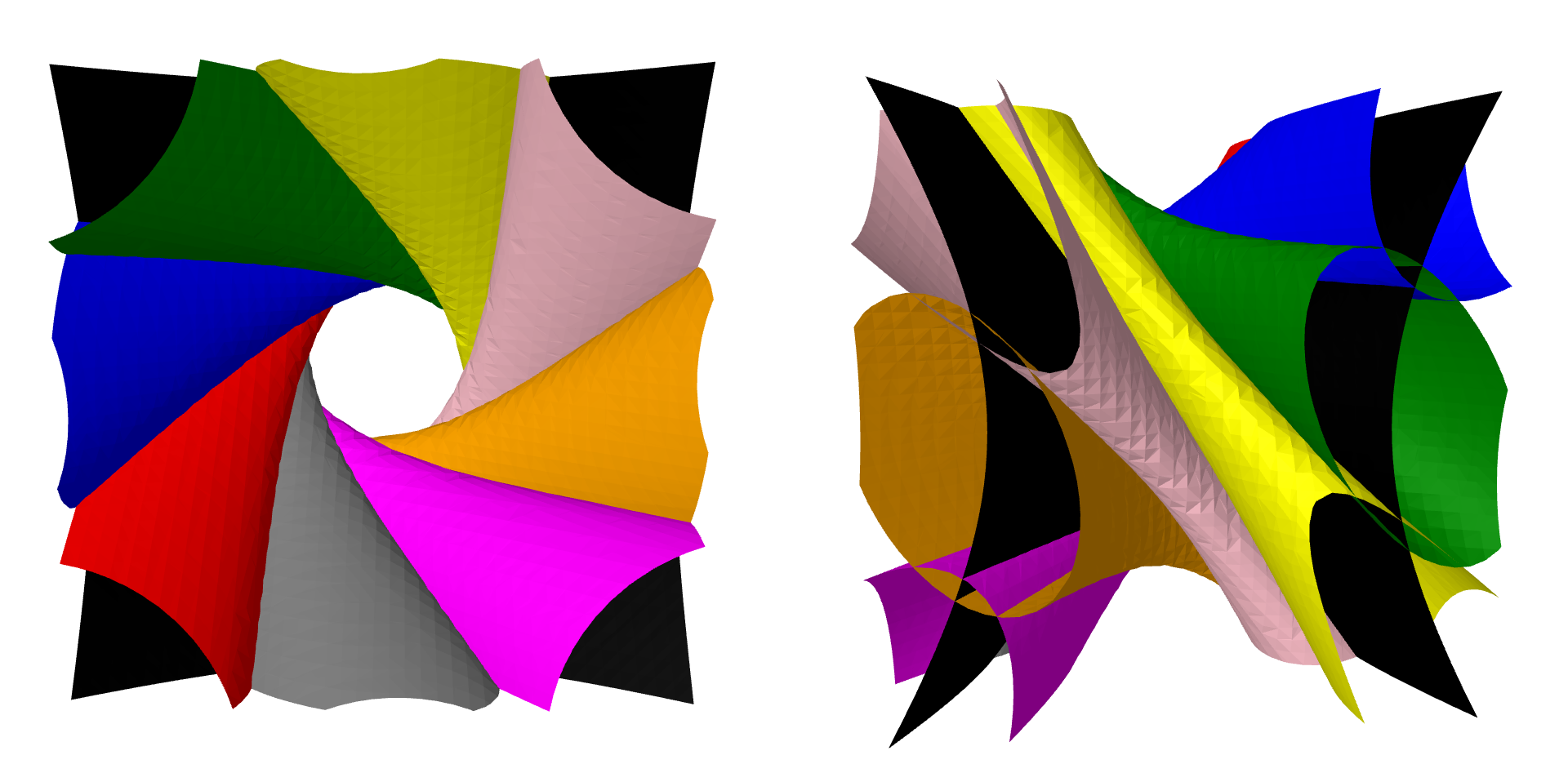}
	\caption{Extension of a Dirichlet domain to $\mathbb{RP}^3$ for a reducible surface subgroup in $\SL(4,\R)$.}
	\label{fig:Octogon}
\end{figure}

In Figure \ref{fig:Triangle} we consider a $(4,4,4)$ triangle group in $\SL(2,\R)$ irreducibly embedded into $\SL(4,\R)$. The domain of discontinuity associated to this $\omega_1$-undistorted subgroup in $\mathbb{RP}^3$ is the complement of the black hypersurface, and we illustrate the sides of the fundamental domain associated with the Dirichlet domain for the most natural basepoint in the totally geodesic surface in $\X$ preserved by $\Gamma$. This domain has 2 connected components, and we conjecture that it has $18$ sides. The green, yellow, orange and red sides are associated with elements $\gamma\in \Gamma$ of word length respectively $1,2,3,4$. 

\begin{figure}[h]
	\centering
	\includegraphics[width=380pt]{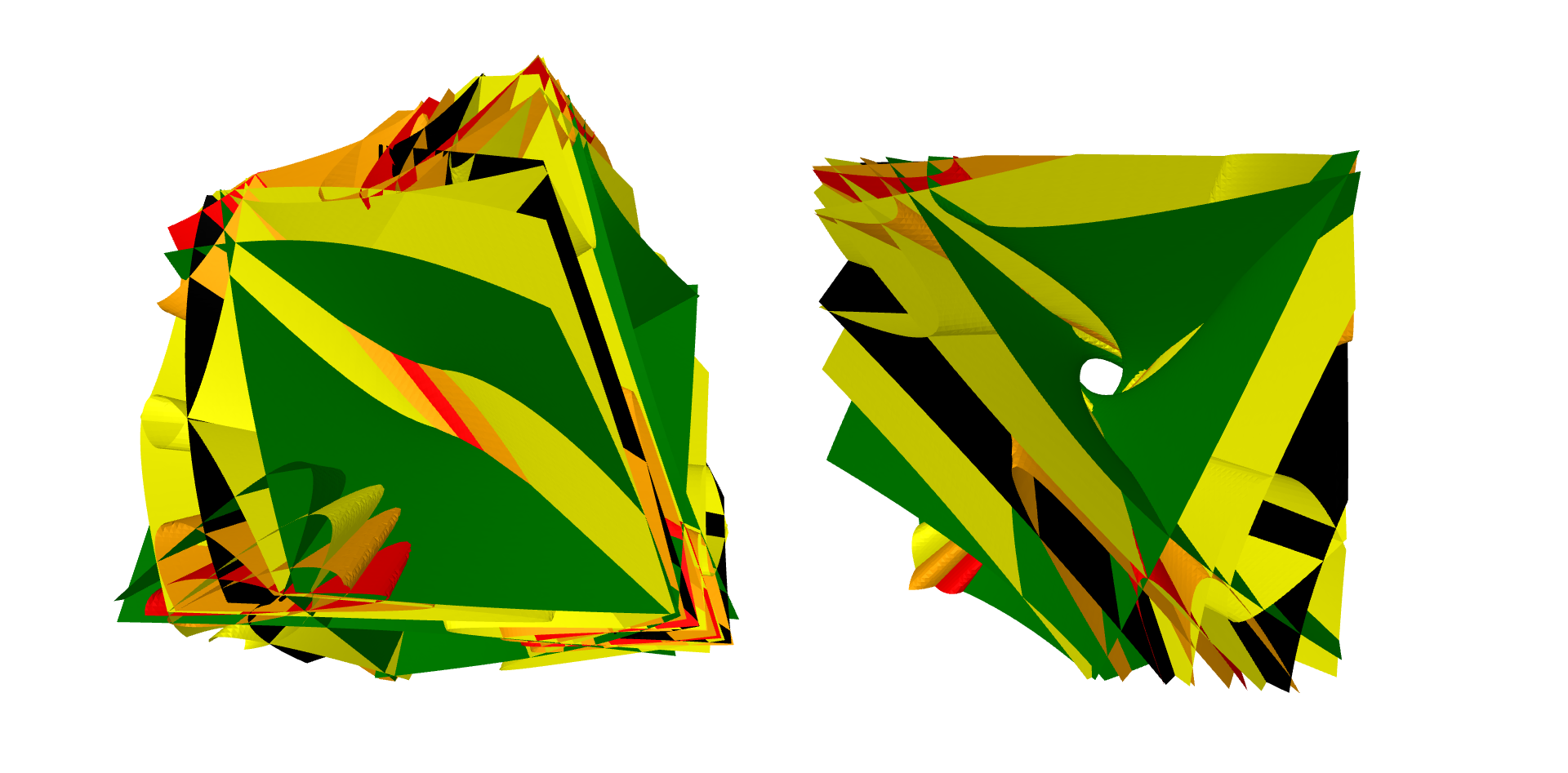}
	\caption{Extension of a Dirichlet domain to $\mathbb{RP}^3$ for a Hitchin $(4,4,4)$ triangle subgroup in $\SL(4,\R)$.}
	\label{fig:Triangle}
\end{figure}

\subsection{Dirichlet-Selberg domains}

Our results on Dirichlet domains for Finsler metrics also imply results for Dirichlet-Selberg domains, which are fundamental domains for subgroups of $\SL(d,\R)$ that have a particularly nice geometry.

Selberg introduced a construction of a fundamental domain for the action on the symmetric space $\SL(d,\R)/\SO(d)$ of discrete subgroups of $\SL(d,\R)$ \cite{Sel60}.
This symmetric space has a \emph{projective model}, denoted $\mathcal{X}_d$, given by the projectivization of positive definite symmetric real $d \times d$ matrices. 
Selberg modified the Dirichlet domain construction by replacing the distance function with:
$$ \Selberg([x],[y])=\log\left(\frac{1}{d}\Tr(x^{-1}y)\right) .$$
In this expression we choose positive definite representatives $x,y$ of the same determinant. 
This invariant is asymmetric and fails the triangle inequality, but nonetheless can be used in place of a metric for the sake of defining Dirichlet domains:
\begin{equation}\label{eqn:ds domain in X}
	\bigcap_{\gamma\in \Gamma}\left\{ x\in \mathcal{X}_d \mid \mathfrak{s}(o,x)\leq \mathfrak{s}(\gamma\cdot o,x) \right\}.
\end{equation}

The advantage of this approach is that the bisectors are linear hyperplanes in $\mathcal{X}_d$ and the resulting domain is convex with respect to projective line segments.
When $\Gamma$ is a uniform lattice in $\SL(d,\R)$, this construction gives rise to a finite-sided projective polyhedron which is a fundamental domain for the action of $\Gamma$ on $\mathcal{X}_d$.
Selberg used these domains to demonstrate local rigidity of uniform lattices in higher rank (an early precursor to Margulis' superrigity), and they are now called \emph{Dirichlet-Selberg domains}.
Recently, Dirichlet-Selberg domains have been studied by Kapovich \cite{Kap23} and Du \cite{du2024geometry,du2025busemannselbergfunctionscompletenessdirichletselberg}.

As before, we will be interested in a natural compactification of this domain. 
We consider the closure $\DS_\Gamma(o)$ of the domain described in (\ref{eqn:ds domain in X}) in the \emph{Satake compactification} $\overline{\mathcal{X}_d}$ of $\mathcal{X}_d$ i.e.\ the closure of the projective model in the projectivization of symmetric matrices. 

\medskip

Kapovich raised questions about which discrete subgroups of $\SL(d,\R)$ admit finite-sided Dirichlet-Selberg domains, and suggested investigating this question in the special case where the discrete subgroups are \emph{Anosov}.

\begin{question}[{Kapovich \cite[Question 11.3(3)]{Kap23}}]\label{q:Do Anosov subgroups admit finite-sided domains}
	Do Anosov subgroups admit finite-sided Dirichlet-Selberg domains?
\end{question}

While $s$ is not itself a metric, it is at bounded distance from the asymmetric Finsler metric $d_{\omega_1}$ associated to the first fundamental weight $\omega_1\in \mathfrak{a}^*$, see Lemma \ref{lem:comparing Finsler and Selberg}. 
Using Theorem \ref{thm:omega-URU implies properly finite-sided} we find many examples of Anosov subgroups admitting finite-sided Dirichlet-Selberg domains.

\begin{corollary}[{Corollary \ref{cor:omega1-URU implies properly finite sided}}]\label{thm:omega1-URU implies finitely-sided}
	Let $\Gamma < \SL(d,\R)$ be $\omega_1$-undistorted, and let $o \in \mathcal{X}_d$. 
	The Dirichlet-Selberg domain $\DS_\Gamma(o) \cap \mathcal{X}_d$ is properly finite-sided.
\end{corollary}

On the other hand, we find examples of Anosov subgroups giving rise to infinite-sided Dirichlet-Selberg domains.

\begin{theorem}[{Theorem \ref{thm:InfiniteSided}}]\label{thm:infinitely-sided example}
	Let $\Gamma$ be a lattice in $\SO(n,1)$, viewed as a subgroup of $\SL(n+1,\R)$. 
	There exists $o \in \mathcal{X}_{n+1}$ such that the Dirichlet-Selberg domain $\DS_\Gamma(o) \cap \mathcal{X}_{n+1}$ has infinitely many sides.
\end{theorem}
Theorem \ref{thm:infinitely-sided example} can be applied for instance to uniform lattices in $\SO(2,1)$ acting on $\mathcal{X}_3$.
Such an example is Borel Anosov but admits Dirichlet-Selberg domains with infinitely many sides, demonstrating that Corollary \ref{thm:omega1-URU implies finitely-sided} cannot hold for Anosov subgroups in general.
In the proof we take the basepoint $o$ to be on the totally geodesic copy of hyperbolic space preserved by $\SO(n,1)$. 
We consider the intersection of $\mathcal{DS}_\Gamma(o)$ with the space of rank one lines, which can be naturally identified with $\mathbb{RP}^{n}$.
Inside $\mathbb{RP}^{n}$ there is a projective hyperlane belonging to $\mathcal{DS}_\Gamma(o)$. For this special basepoint any bisector intersects the hyperplane in a codimension two subspace, see Figure \ref{fig:BissectorsSelberg}.

\medskip

This does not exactly answer Question \ref{q:Do Anosov subgroups admit finite-sided domains} for these subgroups and one can ask the following: 
\begin{question}
	Let $\Gamma<\SO(2,1)$ be convex cocompact subgroup which is not virtually cyclic, and consider its natural action on $\mathcal{X}_3$ from the inclusion $\SO(2,1) < \SL(3,\R)$. 
	Does there exist any $o \in \mathcal{X}_3$ such that $\DS_\Gamma(o)$ has finitely many sides?
\end{question}

\subsection{Locally symmetric spaces}

When $\omega$ is the highest restricted weight of a finite-dimensional irreducible representation $V$ of $G$ of dimension $d$, the Selberg invariant for points in $\mathcal{X}_d$ defines a notion of a \emph{restricted Selberg invariant} for points in the symmetric space $\X$ associated to $G$. 
When $\Gamma$ is an $\omega$-undistorted subgroup, the horofunction compactification of the locally symmetric space $\overline{\X/\Gamma}$ for this restricted Selberg invariant is related in a natural way to the horofunction compactification of $\X$ for the associated Finsler metric. 

\begin{theorem}[{Theorem \ref{thm:Comparison compactifications of locally sym space}}]
	Let $\omega$ be the highest restricted weight of a irreducible representation $V$ of a semisimple Lie group $G$, and let $\Gamma$ be a {$\omega$-undistorted} subgroup of $G$.
	Then the horofunction compactification $\overline{\X/\Gamma}$ of $\X/\Gamma$ for the restricted Selberg invariant is homeomorphic to $\left(\X\cup \Omega^\omega_\text{horo}\right)/\Gamma$ via the natural map $\phi:\left(\X\cup \Omega^\omega_\text{horo}\right)/\Gamma\to \overline{\X/\Gamma}$.
\end{theorem}

In other words the compactification of the quotient is homeomorphic to the quotient of an open domain in the compactification of $\X$ where $\Gamma$ acts properly discontinuously.
As a corollary, the locally symmetric spaces associated to {$\omega$-undistorted} subgroups are topologically tame, which was proven more generally for Anosov subgroups by Guichard-Kassel-Wienhard \cite{GKW} and Kapovich-Leeb \cite[Theorem 1.5.(ii)]{KL18a}.

\subsection{Restricted Dirichlet-Selberg domains}
We also obtain some partial results about Dirichlet-Selberg domains for some Anosov subgroups of $\SL(d,\R)$ which are not $\omega_1$-undistorted.

\medskip

Let $V$ be a finite dimensional real vector space and let $S^2V$ be the vector space of symmetric tensors in $V\otimes V$. 
Let $S^2V^{\geq0}$ be the closed cone of semi-positive symmetric tensors. 
The Satake compactification of the associated symmetric space is $\overline{\mathcal{X}(V)} = \mathbb{P}(S^2V^{\geq 0})$.
As seen in Theorem \ref{thm:InfiniteSided}, the Dirichlet-Selberg domain can have infinitely many sides for certain Anosov subgroups $\Gamma$. 
We consider certain projectively convex subsets of $\overline{\mathcal{X}(V)}$ so that the intersection with the Dirichlet-Selberg domain is once again properly finite-sided. 
This provides some information about where the infinitely many sides of such domains can accumulate. 

To be more precise, we let $\mathcal{F}\subset\mathbb{P}(V)$ denote a $\Gamma$-invariant compact subset and write $S^2\mathcal{F}$ for the set of rank one lines in $\mathbb{P}(S^2V)$ corresponding to $\mathcal{F}$.
We define $\mathcal{DS}^\mathcal{F}_\Gamma(o)$ as the intersection of $\mathcal{DS}_\Gamma(o)$ with the convex set $\Hull(S^2\mathcal{F})\subset \overline{\mathcal{X}(V)}$. 
In Section \ref{sec:restricting ds domains} we prove a sufficient criterion for this intersection to be properly finite-sided (Theorem \ref{thm:Properly finite sided domains in Hull(F)}).
We describe here a few interesting applications of this criterion.

\begin{theorem}[{Theorem \ref{thm:Projective Anosov finite sided}}]
Let $\Gamma < \SL(d,\R)$ be a projective Anosov subgroup, and let $\Lambda$ denote the projective limit set of $\Gamma$, i.e.\ $\Lambda =\lbrace \xi^1_\Gamma(x)|x\in \partial\Gamma\rbrace\subset \mathbb{RP}^{d-1}$.
The domain $\mathcal{DS}_\Gamma^{\Lambda}(o)$ is properly finite-sided in $\Hull(S^2\Lambda)\subset \mathbb{P}(S^2\mathbb{R}^d)$ for all $o\in \mathcal{X}(V)$.
\end{theorem}

If $\Gamma$ is moreover convex-cocompact in the sense of \cite{DGK19,Zim21}, one can take $\mathcal{F}$ larger.

\begin{theorem}[{Theorem \ref{thm:Projective Anosov convex cocompact finite sided}}]
Let $\Gamma < \SL(d,\R)$ be a projective Anosov subgroup that is convex cocompact, i.e. that preserves a properly convex domain $\Omega\subset \mathbb{RP}^{d-1}$ and acts cocompactly on a convex set $\mathcal{C}\subset \Omega$. 
Let $\Lambda=\lbrace \xi^1_\Gamma(x)|x\in \partial\Gamma\rbrace$. 
The domain $\mathcal{DS}_\Gamma^{\mathcal{C}\cup\Lambda}(o)$ is properly finite sided in $\Hull\left(S^2\left( \mathcal{C}\cup \Lambda\right)\right)\subset \mathbb{P}(S^2\mathbb{R}^d)$ for all $o\in \mathcal{X}(V)$.
\end{theorem}

We now consider Anosov subgroups of more general semisimple Lie groups $G$. 
These subgroups can be viewed as subgroups of $\SL(V)$ by choosing a linear representation $G$ on a finite dimensional vector space $V$.

Note that in the following results the set $\mathcal{F}$ does not depend on the discrete subgroup.

\begin{theorem}[{Theorem \ref{thm:restiction applied for URU representations in G}}]
\label{thm:restiction applied for URU representations in GINTRO}
Let $\omega$ be the highest restricted weight of a representation $V$ of a semisimple Lie group $G$.
There exists a non-empty compact $G$-invariant subset $\mathcal{F}\subset \mathbb{P}(V)$ such that for all {$\omega$-undistorted} subgroups $\Gamma$ of $G$ and all $o \in \mathcal{X}(V)$ the restricted Dirichlet-Selberg domain $\mathcal{DS}^\mathcal{F}_\Gamma(o)$ is properly finite-sided.
\end{theorem}

While Theorem \ref{thm:omega-URU implies properly finite-sided} already applies to $\omega$-undistorted subgroups, Theorem \ref{thm:restiction applied for URU representations in GINTRO} may be more useful in practice. 
Indeed, the Finsler metrics we consider are not smooth and their bisectors are difficult to understand. 
Instead, one can embed $\X$ as a totally geodesic submanifold of $\mathcal{X}(V)$ for a suitable representation $V$ of $G$ and then restrict the Dirichlet-Selberg construction. 
Since the Selberg invariant is smooth with hyperplane bisectors, the restricted Dirichlet-Selberg domain may be easier to understand than the Dirichlet domain for the Finsler metric $d_\omega$.

Let $\Delta$ denote the set of simple restricted roots of $G$.
We consider $\Delta$-Anosov subgroups and the adjoint representation $V = \mathfrak{g}$.
Let $\mathcal{N}\subset \mathbb{P}(\mathfrak{g})$ be the \emph{nilpotent cone}, i.e.\ the cone of nilpotent elements in the Lie algebra $\mathfrak{g}$ of $G$ and let $S^2\mathcal{N}\subset \mathbb{P}\left(S^2\mfg \right)$ be the corresponding space of rank one tensors.
The restricted Dirichlet-Selberg domain $\DS_\Gamma^\mathcal{N}(o)$ in $\mathbb{P}(S^2 \mathfrak{g})$ is obtained via the adjoint representation of $G$, intersected with the convex set $\Hull(S^2\mathcal{N})$.

\begin{theorem}[{Corollary \ref{cor:RestrictionBorel Anosov}}]
	Let $\Gamma < G$ be a $\Delta$-Anosov subgroup. 
	Every restricted Dirichlet-Selberg domain $\DS_\Gamma^\mathcal{N}(o)$ is properly finite-sided.
\end{theorem}

\subsection{Outline of the paper}
In Section \ref{Sec:background on symmetric spaces and Anosov subgroups} we recall the properties of Anosov subgroups that we need in the paper and give the definition of {$\omega$-undistorted} subgroups. 
In Section \ref{Sec:Finsler metrics and compactifications} we review the construction of Finsler metrics associated to linear forms $\omega \in \mathfrak{a}^\ast$ and properties of their horofunction compactifications.

Section \ref{sec:DomainProperHorofunctions} is the heart of the paper.
We show that any {$\omega$-undistorted} subgroup $\Gamma$ admits a cocompact domain of proper discontinuity which can be characterized as those horofunctions which are proper and bounded below on any/every $\Gamma$ orbit.
Such horofunctions satisfy an important stronger property: they are locally uniformly proper, see Lemma \ref{lem:LocalUnif}.

In Section \ref{sec:DirichletFinsler} we deduce that {$\omega$-undistorted} subgroups have properly finite-sided Dirichlet domains for the Finsler metric $d_\omega$. 
We also give an equivalent characterization of the {$\omega$-undistorted} property in terms of disjoint half-spaces, see Theorem \ref{thm:disjoint half-spaces implies URU}. 

In Section \ref{sec:Dirichlet-Selberg domains} we review the definition and basic properties of Dirichlet-Selberg domains and prove Theorem \ref{thm:infinitely-sided example}. 
In Section \ref{sec:locally symmetric spaces} we consider a restricted Selberg invariant, and show that locally symmetric spaces associated to {$\omega$-undistorted} subgroups are topologically tame. 
In Section \ref{sec:restricting ds domains} we consider restricted Dirichlet-Selberg domains for $\Theta$-Anosov subgroups $\Gamma < G$ and representations $V$ of $G$. 
We then discuss several applications. 

\subsection*{Acknowledgments.}
We thank Marit Bobb, Jeffrey Danciger, Fran\c{c}ois Gu\'{e}ritaud, Rylee Lyman, Beatrice Pozzetti, Gabriele Viaggi and Anna Wienhard for interesting discussions related to this project, and additionally thank Beatrice Pozzetti and Anna Wienhard for their remarks on a previous version of the paper.
We would also like to thank the anonymous reviewer for helpful suggestions which have improved the exposition of the paper.
C. Davalo was funded through the RTG 2229 “Asymptotic
Invariants and Limits of Groups and Spaces” and the DFG Emmy Noether project 427903332 of B. Pozzetti.
J.M.\ Riestenberg was supported by the RTG 2229 “Asymptotic Invariants and Limits of Groups and Spaces” and by the DFG under Project-ID 338644254 - SPP2026.

\tableofcontents

\section{Background on symmetric spaces and Anosov subgroups}\label{Sec:background on symmetric spaces and Anosov subgroups}

In this section we introduce the notion of an \emph{$\omega$-undistorted} subgroup of $G$, see Definition \ref{def: v-uru}. 
First we review some important properties of the visual boundary of a symmetric space of non-compact type, and fix some notation. 
We then recall the relevant properties of Anosov subgroups.

\subsection{The visual boundary of a symmetric space of non-compact type}

Let $G$ be a connected semisimple Lie group with finite center and let $\X$ be the associated symmetric space of non-compact type. 
The symmetric space $\X$ is a Hadamard manifold. 
Its \emph{visual boundary}, denoted $\partial_\text{vis}\X$, is the set of asymptote classes of geodesic rays. 
The visual boundary of a symmetric space has the structure of a thick spherical building, see \cite{KLP17,E96} for further discussion.

\medskip

Let $\mathfrak{a}$ be a maximal abelian subspace of $\mathfrak{p}$ where $\mathfrak{g} = \mathfrak{k}\oplus \mathfrak{p}$ is a Cartan decomposition of the Lie algebra $\mathfrak{g}$ of $G$. 
Let $\Sigma\subset \mathfrak{a}^*$ be the associated (restricted) root system, and $\Delta$ be a choice of simple roots. 
This choice defines a positive \emph{(Euclidean) Weyl chamber} $\mathfrak{a}^+ \coloneqq \lbrace \mathrm{v}\in \mathfrak{a} \mid \forall \alpha \in \Delta, \,\alpha(\mathrm{v})\geq 0\rbrace$. 

For $x,y \in \X$, there is an isometry $g \in G$ conjugating the transvection from $x$ to $y$ into $\exp(\mathfrak{a}^+)$. 
The corresponding element of $\mathfrak{a}^+$ is called the the vector-valued distance, and denoted by $\vec{d}(x,y)$.

The projecitivization $\mathbb{S}\mathfrak{a}^+$ is naturally identified with a subset of $\partial_{\rm vis} \X$ called a \emph{(spherical) Weyl chamber}. 
It is a fundamental domain for the natural action of $G$ on $\partial_{\rm vis} \X$. 
In particular, $\partial_\text{vis}\X$ is a union of Weyl chambers $\sigma$, and each is naturally identified with a \emph{model Weyl chamber} $\sigma_{mod}$. 
Every element of the visual boundary has a {\em type} in the model Weyl chamber:
\[ \type \colon \partialvisX \to \sigma_{mod}. \] 

Non-empty faces $\tau_{mod}$ of $\sigma_{mod}$ are in one-to one correspondence with non-empty subsets $\Theta$ of $\Delta$. 
To such a face/subset of simple roots one can associate a \emph{flag manifold} $\mathcal{F}_\Theta = \Flag(\tau_{mod})$ defined as the set of faces $\tau \subset \partial_{\rm vis} \X$ of type $\tau_{mod}$.
It can also be written as $\mathcal{F}_\Theta = G/P_\Theta$ where $P_\Theta$ is the standard parabolic subgroup associated to $\Theta$.
The {\em star} of a simplex $\tau$ in $\partial_\text{vis}\X$ is the union of chambers containing $\tau$, and denoted $\st(\tau)\subset \partialvis \X$.

\medskip

We often fix a subset $\C$ of $\sigma_{mod}$; which can be for instance the limit cone of some discrete subgroup, see below.\footnote{In \cite{KLP17}, $\C$ would be denoted $\Theta$, but we reserve that notation for a collection of simple roots.}
The {\em $\C$-star} of $\tau$ is the subset of the star of $\tau$ with types in $\C$:
\[ \st_\C(\tau) \coloneqq \st(\tau) \cap \type^{-1}(\C) .\]

We further consider certain subsets of $\X$ which appear as cones on subsets of $\partialvisX$. 
For $x \in \X$ and $A \subset \partialvisX$, we let $\mathcal{V}(x,A)$ denote the union of points on geodesic rays from $x$ to $A$. 
In particular, we will consider later the {\em Weyl cone} $\mathcal{V}(x,\st_\C(\tau))$ of a simplex $\tau$ in $\partialvisX$.

\medskip
 
To $\omega\in \mathfrak{a}^*$, one can associate its orthogonal vector $\omega^\perp\in \mathfrak{a}$ for the Killing form. 
Up to the action of the Weyl group $W$, we may assume that $\omega^\perp \in \mathfrak{a}^+$. 
We let $\mathcal{F}_\omega$ denote the flag manifold
\begin{equation}\label{eqn:flag manifold omega}
	\mathcal{F}_\omega \coloneqq \type^{-1}(\omega^\perp) = G \cdot [c_{\omega^\perp}] \subset \partialvisX
\end{equation}
where $c_{\omega^\perp}$ is a geodesic ray determined by $\omega^\perp$.
The flag manifold $\mathcal{F}_\omega$ is naturally identified with $\mathcal{F}_\Xi$ for the subset of simple roots $\Xi=\lbrace \alpha\in \Delta \mid \alpha(\omega^\perp)\neq 0\rbrace$.

\begin{remark}
	In the following sections, we will consider $\Theta$-Anosov subgroups which have limit maps with values in $\mathcal{F}_\Theta$ and domains of discontinuity in the flag manifold $\mathcal{F}_\omega = \mathcal{F}_\Xi$. 
	We emphasize that $\Xi$ is typically not equal to $\Theta$ in this setup. 
\end{remark}

\subsection{The \texorpdfstring{$\omega$}{omega}-undistorted condition}
Let $\Gamma$ be a discrete subgroup of $G$.
The \emph{limit cone} $\mathcal{C}_\Gamma$ of $\Gamma$, introduced by Benoist \cite{Benoist}, is given by:
\begin{equation}\label{def:limit cone}
	\mathcal{C}_\Gamma \coloneqq \bigcap _{n\in \mathbb{N}}\overline{\left\lbrace \frac{\vec{d}(o,\gamma \cdot o)}{d(o,\gamma  \cdot o)} |\gamma\in \Gamma,\, d(o,\gamma\cdot o)\geq n\right\rbrace}\subset \sigma_{mod} .
\end{equation}
This definition does not depend of the base-point $o\in \X$. 
The limit cone is non-empty when $\Gamma$ is unbounded and is compact in general.
Note that for a fixed $o \in \mathbb{X}$, the map $\mu \colon G \to \mathfrak{a}^+$ given by $\mu(g) \coloneqq \vec{d}(o,g\cdot o)$ is often called the \emph{Cartan projection}.

Before introducing the $\omega$-undistorted notion, we recall a similar condition that characterizes the Anosov property.
\begin{definition}[{\cite[Definition 5.17]{KLP17}}]\label{def: Anosov}
Let $\Theta\subset \Delta$ be a set of simple roots. 
A finitely generated subgroup $\Gamma$ is \emph{$\Theta$-Anosov} if and only if for one (and hence any) word metric $|\cdot|$ on $\Gamma$, there exist $\epsilon, C>0$ such that for all $\alpha\in \Theta$ and $\gamma\in \Gamma$:
$$\alpha\left(\vec{d}(o,\gamma\cdot o)\right)\geq \epsilon |\gamma|-C.$$

Equivalently, $\Gamma$ is \emph{$\Theta$-Anosov} if it is quasi-isometrically embedded and the limit cone $\mathcal{C}_\Gamma$ avoids $\Ker(\alpha)$ for all $\alpha\in \Theta$.
\end{definition}
 
 We introduce a similar notion.
 
\begin{definition}\label{def: v-uru}
	Let $\omega \in \mathfrak{a}^\ast$ be nonzero.
   	We say that a finitely generated subgroup $\Gamma <G$ is \emph{$\omega$-undistorted} if for one (and hence any) word metric $|\cdot|$ on $\Gamma$, there exist $\epsilon, C>0$ such that for all $w$ in the Weyl group and $\gamma\in \Gamma$: 
	   \begin{equation}
	   	\label{eq:DefURU}
	   	\abs{\omega\left(w\cdot \vec{d}(o,\gamma\cdot o)\right)} \geq \epsilon |\gamma|-C.
	   \end{equation}
 
Equivalently $\Gamma$ is $\omega$-undistorted if and only if it is quasi-isometrically embedded in $G$ and if its limit cone $\mathcal{C}_\Gamma$ avoids $w\cdot \Ker(\omega)$ for all $w$ in the Weyl group. 
\end{definition}
  
\begin{remark}\label{rmk: weyl group orbits}
	It turns out that for irreducible root systems, the set of simple roots in some Weyl group orbit $\lbrace \omega\circ w|w\in W\rbrace\cap \Delta$ is either the set of long or short roots in $\Delta$ when non-empty, see \cite[{Figure 4}]{Dav23}. 
	In this case, we set $\omega_\Theta$ to be such a root and then the set of $\omega_\Theta$-undistorted subgroups coincides with the set of $\Theta$-Anosov representations. 
	As a consequence, our results apply to many Anosov subgroups of interest; for instance those appearing in higher Teichm\"{u}ller theory: Hitchin and maximal subgroups, or the more general $\Theta$-positive subgroups introduced by Guichard-Wienhard \cite{GW18}, see \cite{GLW21}.
\end{remark}

\begin{proposition}\label{prop:omega-undistorted implies Anosov}
Let $\Gamma$ be a $\omega$-undistorted subgroup of $G$ that is not virtually cyclic.
There is a unique component $\sigma_\Gamma$ of $\sigma_\text{mod}\setminus \bigcup_{w\in W}w\cdot \Ker(\omega)$ containing $\mathcal{C}_\Gamma$. 
Let $\Theta(\sigma_\Gamma)\in \Delta$ be the set of simple roots whose associated walls do not intersect $\sigma_\Gamma$. 
The set $\Theta(\sigma_\Gamma)$ is nonempty and $\Gamma$ is $\Theta(\sigma_\Gamma)$-Anosov.
\end{proposition}

\begin{proof}
	Let $\Gamma$ be an $\omega$-undistorted subgroup of $G$ that is not virtually cyclic. 
	Since $\Gamma$ is discrete, it follows from the definition and \cite[Theorem 1.1]{Kassel}, see also \cite[Fact 4.10]{KasselTholozan}, that there is a unique component $\sigma_\Gamma$ of $\sigma_\text{mod}\setminus \bigcup_{w\in W}w\cdot \Ker(\omega)$ containing $\mathcal{C}_\Gamma$. 
	The fact that $\Theta(\sigma_\Gamma)\neq \emptyset$ is exactly \cite[Lemma 5.20]{Dav23}. 
	 Finally, since $\mathcal{C}_\Gamma\subset \sigma_\Gamma$, the limit cone $\mathcal{C}_\Gamma$ avoids the walls associated to the simple roots in $\Theta(\sigma_\Gamma)$, and $\Gamma$ is quasi-isometrically embedded, so $\Gamma$ is $\Theta(\sigma_\Gamma)$-Anosov.
\end{proof}

\begin{proposition}\label{prop:existence of omega-undistorted subgroups}
	There exists a Zariski dense $\omega$-undistorted subgroup $\Gamma$ of $G$ if and only if the fixed point set of the opposition involution $\iota \colon \mathfrak{a}^+ \to \mathfrak{a}^+$ is not contained in $w \ker \omega$ for some $w \in W$. 
	In this case, one can take $\Gamma$ to be a free group. 
\end{proposition}

\begin{proof}
	 If there is an $\omega$-undistorted subgroup $\Gamma$ of $G$ which is not virtually cyclic, then the component $\sigma_\Gamma$ of  $\sigma_\text{mod}\setminus \bigcup_{w\in W}w\cdot \Ker(\omega)$ containing $\mathcal{C}_\Gamma$ is invariant by the opposition involution and convex, so contains a nonzero vector fixed by the opposition involution. 
	 On the other hand, if there is a vector $v$ with $\iota(v)=v$ and not contained in any hyperplane $w \ker \omega$ for $w \in W$, then we may choose a convex neighborhood $U$ of $[v]$ in $\sigma_{mod}$ which is still in $\sigma_\text{mod}\setminus \bigcup_{w\in W}w\cdot \Ker(\omega)$. 
	 By \cite[Theorem 7.4]{Ben96} there is a Zariski dense free subgroup $\Gamma$ of $G$ with limit cone contained in $U$. Up to taking powers of its generators one can furthermore assume that $\Gamma$ is undistorted, by \cite[Proposition 6.4]{Ben96}.
\end{proof}

\subsection{Illustration of the \texorpdfstring{$\omega$}{omega}-undistorted condition}

In this subsection we consider a few examples to illustrate the $\omega$-undistorted condition, and its relation to the Anosov properties.

\subsubsection{Let $G=\PSL(d,\R)$.}\label{subsubsec:omega URU in SLdR}
Consider the element $\omega_1 \in \mathfrak{a}^\ast$ given by:
$$\omega_1:\Diag(\sigma_1,\cdots ,\sigma_d)\mapsto \sigma_1 ,$$
where $\mathfrak{a}$ is identified with traceless symmetric matrices.
Note that, for $g \in \PSL(d,\R)$, the set of logarithms of singular values $\{\log \sigma_i(g)\}_{i=1}^d$ coincides with the set of values $\omega_1\left(w\cdot \vec{d}(o,g \cdot o)\right)$ as $w$ varies over $W$.
So a subgroup $\Gamma < \PSL(d,\R)$ is $\omega_1$-undistorted exactly when 
	\[ \lvert \log \sigma_i(g) \rvert \ge A \abs{\gamma} - B \]
holds for some $A,B>0$ and for all $i=1,\dots,d$ for all $\gamma \in \Gamma$. 
Figure \ref{fig:SL4} illustrates the intersection of the model Weyl chamber $\sigma_{mod}$ with the hyperplanes $w \cdot \Ker(\omega_1)$ as dotted lines.
When $\Gamma$ is not virtually cyclic, the limit cone $\mathcal{C}_\Gamma$ is connected and invariant by the opposition involution, hence contained in the gray region.
More generally, we have the following. 

\begin{proposition}
	\label{prop:omega1ImpliesNAnosov}
	If $d=2n$, then any $\omega_1$-undistorted subgroup of $\SL(d,\R)$ that is not virtually cyclic is $n$-Anosov. 
	If $d$ is odd, there exist no $\omega_1$-undistorted subgroups of $\SL(d,\R)$ that is not virtually cyclic.
\end{proposition}

\begin{proof}
	For $1\leq k<d$, let $c_k$ denote the set of tuples $(\lambda_1,\lambda_2,\cdots, \lambda_d)$ such that $\lambda_1\geq \lambda_2\geq \cdots \geq \lambda_k>0>\lambda_{k+1}\geq \cdots \geq \lambda_d$.
	These are exactly the connected components of $\mathfrak{a}^+ \setminus \bigcup_{w\in W}w\cdot \Ker(\omega_1)$.
	The opposition involution maps $c_k$ to $c_{d-k}$.
	
	For $d$ odd there are no invariant connected components.
	For $d=2n$ there is only one invariant component, and this component avoids the wall $\lambda_n=\lambda_{n+1}$. 
	Hence $\Theta(\sigma_\Gamma)=\{\lambda_n-\lambda_{n+1}\}$.
\end{proof}

\begin{figure}[ht]
	\centering
	\begin{tikzpicture}	
		\node[circle,scale=0.5,fill=black,label=below:{$\tau_1$}] (t1) at (0,0) {};
		\coordinate (t2) at (4,3.46);
		\coordinate (t13) at (4,0);
		\node[circle,fill=black,scale=0.5,label=below:{$\tau_3$}] (t3) at (8,0) {};
		\draw (t1) -- coordinate[pos=0.5](t12) (t2) -- (t3) coordinate[pos=0.5](t23) -- (t1);
	\fill[gray!30] (t2) -- (t12) -- (t13)-- (t23) -- cycle;	
	\draw (t1) -- coordinate[pos=0.5](t12) (t2) -- (t3) coordinate[pos=0.5](t23) -- (t1);
		\draw[dashed] (t12) -- (t13) -- (t23);	
		\node[circle,fill=black,scale=0.5,label=above:{$\tau_2$}] at (t2) {};	
	\end{tikzpicture}
	\caption{An illustration of $\bigcup_{w} \ker w \omega_1$ for $\SL(4,\R)$.} 
	\label{fig:SL4}
\end{figure}

Now consider the following:
 $$\omega_\Delta:\Diag(\sigma_1,\cdots ,\sigma_d)\mapsto \sigma_1-\sigma_d.$$ 
An $\omega_\Delta$-undistorted subroup $\Gamma < \PSL(d,\R)$ is exactly a $\Delta$-Anosov subgroup of $\PSL(d,\R)$, also called \emph{Borel Anosov}. This linear form is interesting in particular as its associated metric in the symmetric space is the Hilbert metric, see Example \ref{rem:Hilbert metric}.

 \subsubsection{Let $G=\Sp(6,\R)$.} \label{example: sp6r}
In the following example we see that the set of roots $\Theta$ for which a $\omega$-undistorted subgroup is Anosov can vary. 

 We consider representations of convex cocompact subgroups $\Gamma\subset \SL(2,\R)$ that factor through a representation $f \colon \SL(2,\R)\to \Sp(6,\R)$. 
 Given a partition $\tau = \lbrace \tau_1,\tau_2,\cdots \tau_k\rbrace$ with repetition such that $\tau_1+\cdots+\tau_k=6$ and each odd integer appears an even number of times, we can define a representation $f_\tau \colon \SL(2,\R)\to \Sp(6,\R)$ as the direct sum of irreducible representations $\SL(2,\R)\to \SL(\tau_i,\R)$. 
 Let $j \colon \Gamma\to \SL(2,\R)$ be a Fuchsian representation of the fundamental group of a closed surface. 
 The representations $\rho_\tau = f_\tau \circ j$ are discrete and faithful for $\tau\neq \lbrace 1,1,1,1,1,1\rbrace$ and their limit cone is a single point $\{\mathrm{v}_\tau \}$.

 \medskip

 The positive Weyl chamber $\mathfrak{a}^+$ can be identified as the space of triples $(\lambda_1,\lambda_2,\lambda_3)$ with $\lambda_1\geq \lambda_2\geq \lambda_3\geq 0$. The three roots are $\alpha_1(\lambda_1,\lambda_2,\lambda_3)=\lambda_1-\lambda_2$, $\alpha_2(\lambda_1,\lambda_2,\lambda_3)=\lambda_2-\lambda_3$ and $\alpha_3(\lambda_1,\lambda_2,\lambda_3)=2\lambda_3$.
 
 \medskip

 Let $\omega_3=\lambda_1+\lambda_2+\lambda_3$ be the third fundamental weight, which is not in the Weyl group orbit of a simple root, even up to rescaling.

 \medskip

 We illustrate some of these points in Figure \ref{fig:Sp6}, in the projective chart defined by $\lambda_1=1$. In this picture, the blue line corresponds to the union of $\Ker(w\cdot\omega_3)$ for $w\in W$.

 \begin{figure}[ht]
 	\centering
 	\begin{tikzpicture}[scale=4]	
 	\coordinate (t1) at (1,1);
 	\coordinate (t2) at (1,0);
 	\coordinate (t3) at (0,0);
 	\node[scale=0.5,label=right:{$\lambda_1=\lambda_2$}] at (1,.5) {};
 	\node[scale=0.5,label=above left:{$\lambda_2=\lambda_3$}] at (.5,.5) {};
 	\node[scale=0.5,label=below:{$\lambda_3=0$}] at (0.5,0) {};
         \draw[thick] (t3) -- (t2) -- (t1) -- cycle;	
         \draw[very thick,blue] (.5,.5) -- (1,0);	
 	\node[circle,fill=black,scale=0.5,label=above:{$\mathrm{v}_{\lbrace 2,2,2\rbrace}$}] at (t1) {};
 	\node[circle,fill=black,scale=0.5,label=below right:{$\mathrm{v}_{\lbrace 2,2,1,1\rbrace}=\mathrm{v}_{\lbrace 3,3\rbrace}$}] at (t2) {};
 	\node[circle,fill=black,scale=0.5,label=below left:{$\mathrm{v}_{\lbrace 2,1,1,1,1\rbrace}$}] at (t3) {};

 \node[circle,fill=black,scale=0.5,label=above:{$\mathrm{v}_{\lbrace 6\rbrace}$}] at (.6,.2) {};
 \node[circle,fill=black,scale=0.5,label=left:{$\mathrm{v}_{\lbrace 4,2\rbrace}$}] at (.333,.333) {};
 \node[circle,fill=black,scale=0.5,label=above:{$\mathrm{v}_{\lbrace 4,1,1\rbrace}$}] at (.333,0) {};

 	\end{tikzpicture}
 	\caption{An illustration of the positive Weyl chamber $\mathbb{S}\mathfrak{a}^+=\sigma_{mod}$ for $\Sp(6,\R)$.} 
 	\label{fig:Sp6}
 \end{figure}

 In this picture, we see that $\rho_{\lbrace 6\rbrace},\rho_{\lbrace 4,2\rbrace},\rho_{\lbrace 2,1,1,1,1\rbrace},\rho_{\lbrace 4,1,1\rbrace},\rho_{\lbrace 2,2,2\rbrace}$ are $\omega_3$-undistorted. There are two connected components of $\mathfrak{a}^+$ minus the gray line. The first one contains $\mathrm{v}_{\lbrace 2,2,2\rbrace}$, and having a limit cone in this component implies being $\lbrace \alpha_3\rbrace$-Anosov. Undistorted representations whose limit cone lies in the other component are all $\lbrace \alpha_1\rbrace$-Anosov. 

\subsection{Boundary maps and the Morse property}

An important feature of Anosov subgroups is the existence of a boundary map, which can be characterized in the following way.
Let $\Theta$ be a non-empty set of simple roots. 

\begin{theorem}[{\cite{KLP17,BPS19}}]\label{thm:characterizing boundary maps}
Let $\Gamma$ be $\Theta$-Anosov subgroup of $G$. 
The group $\Gamma$ is hyperbolic, with Gromov boundary $\partial \Gamma$. 
There exists a unique continuous $\Gamma$-equivariant map $\xi_\Theta:\partial \Gamma\to \mathcal{F}_\Theta$ such that for all $o\in \X$ and every geodesic ray $(\gamma_n)$ in $\Gamma$ converging to $\zeta\in \partial \Gamma$, every limit point of $(\gamma_n\cdot o)$ belongs to  a Weyl chamber that contains $\xi_\Theta(\zeta)$, i.e. every limit point belongs to $\st\left(\xi_\Theta(\zeta)\right)\subset\partialvisX$.
\end{theorem} 

The map $\xi_\Theta \colon \partial \Gamma \to \mathcal{F}_\Theta$ is called the \emph{boundary map} of $\Gamma$. 

\medskip

Symmetric spaces of rank one satisfy the Morse Lemma: quasi-geodesics stay close to geodesics. 
That property fails in higher rank, but a suitable generalization holds: uniformly regular quasigeodesics stay close to Weyl cones.

\begin{theorem}[{\cite{KLP17,KLP18b,BPS19}}]
\label{thm:MorseLemmaKLP}
Let $\Gamma$ be $\Theta$-Anosov subgroup of $G$. Let $o\in \X$, and let us a fix a word metric $|\cdot|$ on $\Gamma$. 
There exist $D \ge 0$ such that if $\gamma \in \Gamma$ lies on a geodesic ray from $e \in \Gamma$ to $\zeta \in \partial \Gamma$ then the distance from $\gamma \cdot o$ to the Weyl cone $\mathcal{V}\left(o,\st(\xi_\Theta(\zeta))\right)$ is at most $D$.

\end{theorem}

We observe that the orbit also stays close to Weyl cones on $\mathcal{C}$-stars.
This has important consequences throughout the paper.

\begin{lemma}\label{lem:Morse Lemma}
	Let $\Gamma$ be a $\Theta$-Anosov subgroup of $G$. 
	Let $\C$ be any compact neighborhood of the limit cone $\C_\Gamma$, let $o \in \X$ and fix some word metric $|\cdot |$ on $\Gamma$.
	There exists $D\ge 0$ such that:
	\begin{enumerate}
		\item If $(\gamma_n)_{n \in \mathbb{N}}$ is a geodesic ray in $\Gamma$ converging to $\zeta \in \partial \Gamma$ with $\gamma_0=e$, then for all $n \in \mathbb{N}$, the distance from $\gamma_n \cdot o$ to the Weyl cone $\mathcal{V}\left(o,\st_\mathcal{C}\left(\xi_\Theta(\zeta)\right)\right)$ is at most $D$,
		\item For all $\gamma \in \Gamma$, there exists $\zeta\in \partial \Gamma$ such that the distance from $\gamma \cdot o$ to the Weyl cone $\mathcal{V}\left(o,\st_\mathcal{C}\left(\xi_\Theta(\zeta)\right)\right)$ is at most $D$.
	\end{enumerate}
\end{lemma}

\begin{proof}

We first remark that since $\Gamma$ is hyperbolic, there exist $D' \ge 0$ such that every $\gamma \in \Gamma$ is at distance at most $D'$ from some element $\gamma'$ that belongs to an infinite geodesic ray starting from $e\in \Gamma$. 
Hence, since Anosov subgroups are quasi-isometrically embedded, (2) follows from (1).

\medskip

Now let $\gamma$ lie on a geodesic ray $(\gamma_n)$ in $\Gamma$ with $\gamma_0=e$.
By Theorem \ref{thm:MorseLemmaKLP} we know that there exist a point $x\in \mathcal{V}\left(o,\st \left(\xi_\Theta(\zeta)\right)\right)$ for some $\zeta\in \partial \Gamma$ that is at distance at most $D$ from $\gamma\cdot o$. 
Since $\mathcal{C}$ contains a neighborhood of $\mathcal{C}_\Gamma$, it contains $\vec{d}(o,\gamma\cdot o)$ for all $\gamma$ large enough. 
Moreover since $d(o, \gamma\cdot o)$ goes to $+\infty$, $\mathcal{C}$ must contain $\vec{d}(o,x)$ if $\gamma$ is large enough. 
Hence there exist $D''$ such that either $d(o, \gamma\cdot o)\leq D''$ or $x\in \mathcal{V}\left(o,\st_\mathcal{C}\left(\xi_\Theta(\zeta)\right)\right)$. 
This concludes the proof.
\end{proof}

\begin{example}[Taking a neighborhood of $\mathcal{C}_\Gamma$ is necessary]
	We discuss an example where $\Gamma$ orbits fail to stay at uniform distance from Weyl cones on $\mathcal{C}_\Gamma$.	
	Indeed, let $\gamma$ be an isometry of $\HH^2 \times \HH^2$ which decomposes as a hyperbolic isometry on the first factor and a parabolic\footnote{i.e. a unipotent element of $\PSL(2,\R)$.} isometry on the second. Then $\Gamma = \langle \gamma \rangle$ is Anosov with respect to the first factor. Let $o=(o_1,o_2)$ be a basepoint in $\HH^2 \times \HH^2$. The fixed points of $\gamma$ on $\partial \HH^2 \times \partial \HH^2$ are $(\xi_1^{\pm},\xi_2)$; let $\tau$ be the simplex in $\partial_{\rm vis} (\HH^2 \times \HH^2)$ corresponding to $\xi_1^+$. While the (forward) orbits of $\gamma$ uniformly fellow travel the Weyl cone 
	$$ \mathcal{V}(o,\st(\tau)) = \{(p,q) \in \HH^2 \times \HH^2 \mid o_1p(+\infty)=\xi_1^+ \},$$
	they drift logarithmically away from
	$$ \mathcal{V}(o,\st_{\mathcal{C}_\Gamma}(\tau)) = \mathcal{V}(o,\tau) = \{(p,o_2) \in \HH^2 \times \HH^2 \mid o_1p(+\infty)=\xi_1^+ \}.$$
	On the other hand, for any neighborhood $\mathcal{C}$ of $\mathcal{C}_\Gamma$, the orbits of $\gamma$ uniformly fellow travel 
	$$ \mathcal{V}(o,\st_{\mathcal{C}}(\tau)) = \{(p,q) \in \HH^2 \times \HH^2 \mid o_1p(+\infty)=\xi_1^+, d(o_2,q) \le C d(o_1,p) \} $$
	where $C$ is a constant depending on $\mathcal{C}$. 
\end{example}

\section{Finsler metrics and horofunction compactifications}\label{Sec:Finsler metrics and compactifications}
In this section we review a class of $G$-invariant polyhedral Finsler metrics on the symmetric space $\mathbb{X}$ as well as their horofunction compactifications, which were previously studied by Kapovich-Leeb \cite{KL18a}.
In the sequel we will study Dirichlet domains for these Finsler metrics by considering their closure in the horofunction compactification, which are closely related to Satake compactifications, see \cite{HSWW}.  

\subsection{A family of Finsler metrics on the symmetric space}\label{sec:finsler metrics}

To a non-zero element $\omega\in \mathfrak{a}^*$ one can associate a seminorm $\lVert\cdot\rVert_\omega$ on the model Cartan subalgebra for $\vv\in \mathfrak{a}$ by:
$$\lVert \vv\rVert_\omega=\max_{w\in W} \omega(w\cdot\vv).$$
The seminorm only depends on the Weyl group orbit $W \cdot \omega$, and is symmetric if and only if $W\cdot \omega = W\cdot (-\omega)$.

When $W\cdot \omega$ spans $\mfa^\ast$, $\lVert \cdot \rVert_\omega$ is moreover definite. 
This is always the case when $G$ is simple. 
We will assume from now on that $\omega$ is chosen in such a way.

The seminorm on $\mathfrak{a}$ defines a $G$-invariant Finsler metric $\lVert\cdot\rVert_\omega$ on $\X$, that can be characterized for $\vv\in T_o\X$ by:
$$\lVert \vv\rVert_\omega=\max_{a\in \mathcal{F}_\omega} -\mathrm{d}b_{a,o}(\vv).$$
Here for $o,x,y\in \X$ and $a\in \mathcal{F}_\omega$, $b_{a,o}:\X\to \R$ is the Busemann function associated to $a$ and based at $o$ and $\mathcal{F}_\omega$ denotes the $G$ orbit of an ideal point dual to $\omega$, see \eqref{eqn:flag manifold omega}. 

This defines a \emph{Finsler distance} on $\X$, characterized for $x,y\in \X$ by:
$$d_\omega(x,y)=\lVert\vec{d}(x,y)\rVert_\omega=\max_{a\in \mathcal{F}_\omega} b_{a,y}(x).$$

\begin{nonexample}
	Let $G=\PSL(2,\R)^n$ with $n \ge2$ and let $\omega \colon \mathfrak{a}\to \R$ be the linear form coming from the projection onto the first factor. 
	For this functional  $W\cdot \omega$ does not span $\mfa^\ast$ and the seminorm $\lVert \cdot \rVert_\omega$ is not definite. 
	Indeed, the degenerate pseudo-metric $d_\omega$ on $\left(\mathbb{H}^2\right)^n$ is the composition of projection onto the first coordinate with the distance in $\mathbb{H}^2$.
	In general, we only consider those functionals $\omega$ which lead to nondegenerate metrics, so this example is ruled out.
\end{nonexample}

\begin{example}\label{rem:Hilbert metric}
	Let $G=\PSL(d,\R)$ and $\omega=\omega_\Delta$ from Section \ref{subsubsec:omega URU in SLdR}. Then $d_{\omega_\Delta}$ is the \emph{Hilbert metric} on the projective model $\mathcal{X}$ of $\X$.
\end{example}

\subsection{Horofunction compactification}

We review the construction of a horofunction compactification for an asymmetric metric, see \cite{Wal14,KL18a,HSWW} for further details. 
Let $\mathcal{Y}$ be the space of $1$-Lipschitz functions $f:\X\to \R$ for a $G$-invariant Riemannian metric on $\X$, modding out the line of constant functions. 
This space is endowed with the compact open topology: a basis of neighborhoods of $[f]\in \mathcal{Y}$ is defined by the open sets of the form: 
$$U_{K,\epsilon}=\lbrace [g] \mid \forall x\in K, \,(g-f)(x)<\epsilon\rbrace,$$ 
for $K\subset \X$ compact and $\epsilon>0$.

\medskip
One can define a topological embedding $\iota:\X \to\mathcal{Y}$ by setting $\iota(x_0): x\in \X\mapsto d_\omega(x,x_0)$ for $x_0\in \X$.
Since $\mathcal{Y}$ is compact Hausdorff, the closure of $\iota(\X)$ in $\mathcal{Y}$ is compact Hausdorff and we denote it by $\partial_{\omega}\X = \overline{\iota(\X)} \setminus \iota(\X)$.
The functions representing points in $\partial_{\omega}\X$ are called \emph{horofunctions}.

\subsection{Satake compactification}
\label{sec:representation}
 
Let $G$ be a semisimple real Lie group and $V$ an irreducible real representation of $G$ with finite kernel, and let $\rho:\mathfrak{g}\to \mathfrak{gl}(V)$ be the induced Lie algebra representation. 

We define as previously the space $S^2V$ of symmetric bilinear tensors $Q:V^*\to V$, and the subset $\mathcal{X} = \mathcal{X}(V) \subset \mathbb{P}(S^2 V)$ of projectivizations of positive definite elements, see Section \ref{sec:Dirichlet-Selberg domains}. 
The space $\mathcal{X}$ is the projective model for the symmetric space of $\SL(V)$. 
The symmetric space $\X$ associated to $G$ can be identified with a unique totally geodesic submanifold of the symmetric space of $\SL(V)$, hence it can be seen as a subset $\X\subset \mathcal{X}$ \cite{Kar53,Mos55}.  
Note that this subspace is not in general a linear subspace.

\begin{definition}
	The \emph{Satake compactification} of $\X$ associated to $\rho$ is the closure of $\X\subset \mathcal{X}$ inside the compact space $\mathbb{P}(S^2V)$.
\end{definition}

Let $\mathfrak{a}\subset \mathfrak{p}$ be a maximal abelian subspace of $\mathfrak{p}$ where $\mathfrak{g} = \mathfrak{k} \oplus \mathfrak{p}$ is a Cartan decomposition. 
Given $\lambda \in \mathfrak{a}^\ast$, let $V_\lambda=\lbrace x \mid \forall h\in \mathfrak{a},\, \rho(h)\cdot x=\lambda(h)x\rbrace$. 
The \emph{restricted weight system} associated to $\rho$ is the set $\Phi_\rho\subset \mathfrak{a}^*$ of elements $\lambda$ such that $V_\lambda \neq \lbrace 0\rbrace$. We have the following \emph{weight space decomposition}:
$$V=\bigoplus_{\lambda\in \Phi_\rho} V_\lambda.$$
The \emph{highest weight} $\omega\in \Phi_\rho$ of the representation $\rho$ is the unique element such that for any $\lambda\in \Phi_\rho$, $\omega-\lambda\geq 0$ on $\mathfrak{a}^+$. 

\medskip

Since $\rho \colon \mfg \to \mathfrak{sl}(V)$ is injective, the kernel of $\omega$ cannot contain a simple factor of $\mfg$, so the Finsler metric $d_\omega$ on $\X$ is nondegenerate.
In this case the Satake compactification associated to $\rho$ coincides with the horofunction compactification of $\X$ with respect to $d_\omega$.

\begin{theorem}[{\cite[Theorem 5.5]{HSWW}}]
\label{thm:Satake=Horofunction}
The Satake compactification of $\X$ associated to the representation $\rho$ is $G$-equivariantly homeomorphic to the horofunction compactification $\X\cup \partial_{\omega}\X$.
\end{theorem}

See also \cite[Remark 5.6]{HSWW}. 
Note that \cite[Theorem 5.5]{HSWW} also applies in the more general case of \emph{generalized Satake compactifications} of reducible representations.
In the present paper we restrict attention to the irreducible case, since the Finsler metrics we want to consider are defined by a single weight. 

\subsection{Description of horofunctions}
Horofunctions for the polygonal Finsler metric $d_\omega$ can be constructed from Busemann functions associated to elements in $\mathcal{F}_\omega$.

\medskip

Recall that a flag of $\nu$ of any type corresponds to a simplex in the visual boundary. 

\begin{definition}[Incidence]\label{def:incidence}
	We say that $\zeta \in \partial_{\rm vis} \X$ and a simplex $\nu \subset \partial_{\rm vis} \X$ are \emph{incident}, denoted $\zeta \sim \nu$, if there exists a chamber $\sigma \subset \partial_{\rm vis} \X$ such that $\zeta \in \sigma$ and $\nu \subset \sigma$. 
	We let $I_\nu^\omega\subset \mathcal{F}_\omega$ denote the set of $\zeta \in \mathcal{F}_\omega$ incident to $\nu$. 
\end{definition} 

Note that $I_\nu^\omega$ can also be written as $\st(\nu) \cap \mathcal{F}_\omega$.

\begin{proposition}[{\cite[Section 5]{KL18a}}]
\label{prop:CharacHorofunctions}
Every horofunction of $(\X,d_\omega)$ is of the form $b^\omega_{\nu,o}$ for some flag $\nu$ of any type and some point $o\in \X$:
$$b^\omega_{\nu,o} \coloneqq \max \left\{ b_{\zeta,o} \mid \zeta\in I_\nu^\omega \right\}. $$
\end{proposition}

\begin{example}
Let $G=\PSL(d,\R)$ and $\omega=\omega_1$, so that $\mathcal{F}_{\omega_1}\simeq \mathbb{RP}^{d-1}$. 
A flag $\tau$ of any type corresponds to a tuple $(E^{i_1},E^{i_2},\cdots, E^{i_k})$ of subspaces of $\R^d$ such that $E^{i_1}\subsetneq E^{i_2} \subsetneq\cdots \subsetneq E^{i_k}$ and $\dim(E^{i_\ell})=i_\ell$ for all $1\leq \ell\leq k$. 
A line $\ell\in \mathcal{F}_{\omega_1}\subset \partialvis \X$ satisfies $\ell\sim \tau$ if and only if $\ell \subset E^{i_1}$. 
Indeed this is equivalent to the existence of a full flag $(F^1,F^2,\cdots, F^{d-1})$ in $\R^d$ such that $F^1=\ell$ and $F^{i_\ell}=E^{i_\ell}$ for all $1\leq \ell\leq k$.

\medskip

Let $G=\PSL(d,\R)$ and $\omega=\omega_\Delta$, so that $\mathcal{F}_\omega\simeq \mathcal{F}_{1,d-1}$. 
A flag $\tau$ of any type as before and a pair $(\ell,H)\in \mathcal{F}_{\omega_\Delta}\subset \partialvis \X$ satisfies $(\ell,H)\sim \tau$ if and only if $\ell \subset E^{i_1}$ and $E^{i_k}\subset H$. 
Indeed this is equivalent to the existence of a full flag $(F^1,F^2,\cdots, F^{d-1})$ in $\R^d$ such that $F^1=\ell$, $F^{d-1}=H$ and $F^{i_\ell}=E^{i_\ell}$ for all $1\leq \ell\leq k$.
\end{example}

\section{Domain of proper horofunctions}
\label{sec:DomainProperHorofunctions}
In this section we consider an $\omega$-undistorted subgroup $\Gamma$ of $G$ and study two related domains of discontinuity: $\Omegaom$ in the flag manifold $\mathcal{F}_{\omega}$ and $\OmegaomH$ in the horoboundary $\partial_\omega \X$. 
We show that when $\Gamma$ is $\omega$-undistorted, these are cocompact domains of proper discontinuity for $\Gamma$.
In fact, the domains are constructed from a balanced metric thickening naturally associated to $\omega$ and the limit cone $\C_\Gamma$.
The proper discontinuity and cocompactness of the domains can be deduced from \cite{KLP18a,KL18a} and in some cases \cite{GW12}, but we give a simpler proof in the present case, see Remark \ref{rmk: comparing dods to KL,KLP,GW} for a precise comparison.
For $\omega$-undistorted subgroups, the domains can be characterized as a space of horofunctions which are proper and bounded below on $\Gamma$-orbits, see Proposition \ref{prop:characterization of domains}.

\subsection{Thickenings in flag manifolds and horoboundaries}
\label{sec:Slopes and thickenings}

Recall that we have fixed an $\omega \in \mfa^\ast$, which defines a flag manifold $\mathcal{F}_\omega \subset \partial_{\rm vis}\X$ and a definite Finsler metric $d_\omega$ on $\X$. 

\begin{definition}[{\cite{KLM09,KL06}}]
	The \textit{asymptotic slope} of a convex Lipschitz function $f \colon \X \to \R$ is 
	$$ \asym_f \colon \partial_{\rm vis} \X \to \R, \quad \asym_f(\eta) \coloneqq \lim_{t \to \infty} \frac{f\circ c_\eta(t)}{t} .$$
	This limit always exists for convex Lipschitz functions and is independent of the basepoint of the geodesic ray $c_\eta$.
\end{definition}

For a Riemannian Busemann function associated to $\xi \in \partial_{\rm vis}\X$, the asymptotic slope is given by the Tits angle:
$$\asym_{b_{o,\xi}}(\eta) =-\cos \angle_{\rm Tits}(\xi,\eta).$$

\begin{lemma}
	The slope of a mixed Busemann function $b^{\omega}_{\nu,o}$ is given by
	$$ \asym_{b^{\omega}_{\nu,o}}(\eta) = \max \left\{  -\cos \angle_{\rm Tits}(\xi,\eta) \mid \xi\in I_\nu^\omega \right\}. $$
\end{lemma}

Here $I_\nu^\omega\subset \mathcal{F}_\omega$ refers to the set of $\xi\in \mathcal{F}_\omega$ incident to $\nu$, see Definition \ref{def:incidence}.

\begin{proof}
	Since $b^\omega_{\nu,o}$ and $b_{\xi,o}$ for $\xi\in \mathcal{F}_\omega$ are convex, we can replace the limits by a supremum in the definition of the slope. 
	Therefore:
	$$ \asym_{b^\omega_{\nu,o}}(\eta) = \sup_{t \ge 0} \frac{b^\omega_{\nu,o}\circ c_{o,\eta}}{t} = \sup_{t \ge 0} \max_{\xi \in I^\omega_\nu} \frac{b_{o, \xi}\circ c_{o,\eta}}{t} = \max_{\xi \in I^\omega_\nu} \sup_{t \ge 0} \frac{b_{o,\xi}\circ c_{o,\eta}}{t} ,$$
	$$\sup_{t \ge 0} \frac{b_{o,\xi}\circ c_{o,\eta}}{t}=\asym_{b^\omega_{\nu,o}}(\eta)=-\cos \angle_{\rm Tits}(\xi,\eta) .$$
\end{proof}

\medskip

To a point $\xi\in \partial_\text{vis}\X$ such that $\omega\left(w\cdot\type(\xi)\right)\neq 0$ for all $w\in W$, one can associate its thickening in $\mathcal{F}_\omega$:
\[  \Thomega(\xi) \coloneqq \left\{ a \mid \angle_{\rm Tits}(a,\xi) < \pi/2 \right\} \subset \mathcal{F}_\omega. \]

Similarly we can define a thickening in the horoboundary:
\[ \ThomegaH(\xi) \coloneqq \left\{ [h] \mid \asym_{h}(\xi)<0) \right\} \subset \partial_{horo}^\omega \X .\]

Note that the intersection of $\ThomegaH(\xi)$ with $\mathcal{F}_\omega$ in $\partial_\omega\X$ coincides with $\Thomega(\xi)$.

\begin{lemma}
\label{lem:ThickeningVariation}
Let $\xi_1,\xi_2 \in \sigma$ be two points in an ideal Weyl chamber whose types belong to the same connected component of $\sigma_{mod}\setminus \bigcup_{w\in W}w\cdot \Ker(\omega)$. The associated thickenings coincide: 
$$ \Thomega(\xi_1)=\Thomega(\xi_2),$$ 
$$\ThomegaH(\xi_1)=\ThomegaH(\xi_2).$$
\end{lemma}

\begin{proof}
The fact that $\type(\xi_1),\type(\xi_2)$ lie in the same connected component of $\sigma_{mod}\setminus \bigcup_{w\in W}w\cdot \Ker(\omega)$ implies that the segment $c$ between $\xi_1$ and $\xi_2$ in $\sigma$ contains only elements whose types do not belong to $w\cdot \Ker(\omega)$ for any $w\in W$.

\medskip

The Tits angle with a point $a\in \partial_\text{vis}\X$ is a function on $\partial_\text{vis}\X$ that is continuous on any ideal Weyl chamber: indeed given a Weyl chamber of $\partial_\text{vis}\X$, there exist a flat containing it as well as $a$ in its boundary. On this flat the Tits angle is just the standard Euclidean angle. 

\medskip

Now let $a\in \Thomega(\xi_1)$: on the segment $c$ the Tits angle $\angle_{\rm Tits}(a,\xi)$ is never equal to $\frac{\pi}{2}$, and it varies continuously, so $a\in  \Thomega(\xi_2)$, and vice versa. 
Hence $\Thomega(\xi_1)=\Thomega(\xi_2)$. 
The set $\ThomegaH(\xi_i)$ can be characterized as the set of mixed Busemann functions $b^\omega_{\nu,o}$ such that $I_\nu^\omega\subset \Thomega(\xi_i)$ for $i=1,2$. 
Hence also $\ThomegaH(\xi_1)=\ThomegaH(\xi_2)$.
\end{proof}

Let $\mathcal{C}$ be a subset of a connected component of $\sigma_{mod}\setminus \bigcup_{w\in W}w\cdot \Ker(\omega)$, and let $\Theta\subset \Delta$ be the set of roots which do not vanish on $\mathcal{C}$. Given a flag $\tau \in \mathcal{F}_\Theta$ we define therefore its thickenings:
\[  \Thomega(\tau,\mathcal{C}) \coloneqq \left\{ \eta \mid \angle_{\rm Tits}(\eta,\zeta) < \pi/2 \right\} \subset \mathcal{F}_\omega. \]
\[ \ThomegaH(\tau,\mathcal{C}) \coloneqq \left\{ [h] \mid \asym_{h}(\zeta)<0) \right\} \subset \partial_{\rm horo}^\omega \X .\]
For this definition we chose some $\zeta\in \st_\mathcal{C}(\tau) \in \partial_\text{vis}\X$. The definition does not depend of this choice because of Lemma \ref{lem:ThickeningVariation}. In particular one has the following:
\[ \Thomega(\tau,\C) = \left\{ \eta \mid \forall \xi \in \st_{\C}(\tau), \, \angle_{Tits}(\xi,\eta) < \pi/2 \right\} \subset \mathcal{F}_\omega \]
\[ \ThomegaH(\tau,\C) = \left\{ [b^{\omega}_{\nu,x}] \mid I_\nu^\omega \subset \Thomega(\tau,\C) \right\} \subset \partial_{\rm horo}^\omega \X. \]
These thickenings are closely related to the \emph{metric thickenings} considered by Kapovich-Leeb-Porti, see \cite[Section 8.3]{KL18a} and Remark \ref{rem:KLP thickenings}.

Note that the thickenings $\Thomega(\tau,\mathcal{C})$ and $\ThomegaH(\tau,\mathcal{C})$ depend only on $\tau$ and on the connected component of $\sigma_{mod}\setminus \bigcup_{w\in W}w\cdot \Ker(\omega)$ in which $\mathcal{C}$ lies. 
In practice we will apply this to the case when $\Gamma$ is $\omega$-undistorted and $\C$ is an auxiliary neighborhood of the limit cone $\C_\Gamma$ in the same connected component. 

\begin{remark}\label{rem:KLP thickenings}
	In the present paper we directly define thickenings as subsets of flag manifolds and horoboundaries.
	Kapovich-Leeb-Porti \cite{KLP18a} define thickenings to be subsets of the Weyl group $W$ and use such subsets to construct thickenings in flag manifolds and horoboundaries.
	When $\omega = \omega_1$ and $G=\SL(2n,\R)$, we have $W=S_{2n}$ and the thickening is the subset of $S_{2n}$ taking $1 \in \{1,\dots,2n\}$ into $\{1,\dots n\}$.
	This thickening is balanced, left-invariant for the subgroup of $W$ stabilizing the subset $\{1,\dots,n\}$ (equivalently, the vertex of $\sigma_{mod}$ corresponding to $\Gr(n,2n)$), and right-invariant for the subgroup of $W$ stabilizing $\{1\}$ (equivalently, the vertex of $\sigma_{mod}$ corresponding to $\Gr(1,2n) = \RP^{2n-1}$).
	In general, the thickening is the metric thickening $\Th_{\zeta,\mathrm{w},\pi/2}$ of \cite{KLP18a}
	where $\mathrm{w}$ is the $W$-translate of the dual to $\omega$ in $\sigma_{mod}$ and $\zeta$ is a point in the simplex $\tau_{mod}$ corresponding to $\Theta(\sigma_\Gamma)$, see Proposition \ref{prop:omega-undistorted implies Anosov}.
\end{remark}

\subsection{Behaviour of horofunctions along geodesic rays}

For a semi-simple Lie groups of real rank at least $2$, for every geodesic in the symmetric space there exist a Busemann function that is constant on this geodesic. 
However if we restrict to certain types of geodesics and if we consider only Busemann functions associated to points in $\mathcal{F}_\omega$ we can rule out this phenomenon. 
In this section we prove that the behavior of a Busemann function associated to a point $a\in \mathcal{F}_\omega$, and more generally the behavior of a horofunction in $\partial_\omega\X$ along the orbits of $\Gamma$ are subject to a dichotomy. 

\medskip

Let $ \mathcal{C} \subset \mathfrak{a}^+$ be a closed subset that avoids $w\cdot \Ker(\omega)$ for all $w\in W$. 
Let 
$$ C_{\mathcal{C},\omega}=\inf \left\{ \frac{\abs{\omega(w\cdot\vv)}}{\lVert \vv\rVert \lVert \omega\rVert} \mid \vv\in \mathcal{C}, w\in W \right\},$$
which is a positive constant. 

For $o \in \X$ and $\eta \in \partial_{\rm vis}\X$ we let $c_{o,\eta} \colon [0,\infty) \to \X$ denote the geodesic ray emanating from $o$ asymptotic to $\eta$. 

	\begin{lemma}
	\label{lem:Dichotomy}
	Let $[h]\in \partial_\omega \X$ be a horofunction, let $o\in \X$ be a basepoint and let $f\in \mathcal{F}_\Theta$. 
	Exactly one of the following holds:
	\begin{itemize}
		\item[(i)] $[h]\in \ThomegaH(f,\mathcal{C})$ and for every $\eta \in \st_{\mathcal{C}}(f)$ the geodesic ray $c_{o,\eta}$ satisfies 
		$$ h(c_{o,\eta}(t))-h(o)\leq -C_{{\mathcal{C}},\omega}t .$$ 
		\item[(ii)] $[h]\notin \ThomegaH(f,\mathcal{C})$ and for all $\epsilon>0$ there exist $A>0$ such that for every $\eta \in \st_{\mathcal{C}}(f)$ the geodesic ray $c_{o,\eta}$ satisfies 
		$$ h(c_{o,\eta}(t))-h(o)\geq (C_{\mathcal{C},\omega}-\epsilon)t-A .$$
	\end{itemize}
\end{lemma}

We emphasize that $A$ depends on $o,[h]$, and $f$.

\begin{proof}
By Proposition \ref{prop:CharacHorofunctions} we may write $h=b^{\omega}_{\nu,x}$ for some simplex $\nu$, i.e.\ some flag of any type, and some $x\in \X$. 
Let $c=c_{o,\eta} \colon \R_{\geq 0}\to \X$ be a geodesic ray based at $o$, corresponding to a point $\eta\in \partial_\text{vis}\X$. 

Suppose that $[h]\in \ThomegaH(f,\mathcal{C})$ and $\eta \in \st_{\mathcal{C}}(f)$, i.e.\ $I_\nu^\omega \subset \Thomega(f,\mathcal{C})$ by the discussion after Lemma \ref{lem:ThickeningVariation}.
This means that every $\xi \in \mathcal{F}_\omega$ incident to $\nu$ satisfies $\angle_{\rm Tits}(\xi,\eta)<\pi/2$.
The slope of $b_{x,\xi}$ along $\eta$ is then $-\cos \angle_{\rm Tits}(\xi,\eta)$, which is at most $-C_{\mathcal{C},\omega}$. 
So each $b_{x,\xi} \circ c_\eta$ is bounded above by $-C_{\mathcal{C},\omega}t$, and the same applies to their maximum, $h=b^{\omega}_{\nu,x}$. 

\medskip
Suppose now that $[h]\notin \ThomegaH(f,\mathcal{C})$ , i.e.\ there exists some $\xi\in \mathcal{F}_\omega$ incident to $\nu$ and $\eta \in \st_{\mathcal{C}}(f)$ such that $\angle_\text{Tits}(\xi, \eta)> \frac{\pi}{2}$.
The asymptotic slope of the convex function $h\circ c_{o,\eta}$ is greater than or equal to $C_{\mathcal{C},\omega}$; in particular there exist $E>0$ such that $h(c(E))-h(c(0))\geq (C_{\mathcal{C},\omega}-\epsilon)E$. 
The constant $E$ can be chosen uniformly for all geodesic rays $c_{o,\eta}$ since $h$ is continuous and the set of geodesic rays based at $o$ with $\eta \in \st_{\mathcal{C}}(f)$ is compact.
Since $h\circ c$ is convex, $h(c(t))-h(o)\geq (C_{\mathcal{C},\omega}-\epsilon)t$ for $t>E$. 
Since $h$ is $1$-Lipshitz one has in particular $h(c(t))-h(o)\geq (C_{\mathcal{C},\omega}-\epsilon)t-E$ for all $t\geq 0$, which concludes the proof.
\end{proof}

\begin{example}\label{example:Busemann functions for SL(n,R) and RP^n}
	For $G=\SL(d,\R)$, the Busemann functions $b_{o,\zeta}$ with $\zeta = \R w \in \mathcal{F}_{\omega_1} = \mathbb{RP}^{d-1}$ are given by 
	\[ b_{o,\zeta}(q)= \log \left( \frac{ \abs{w}_q}{\abs{w}_o} \right) .\]
	A maximal flat containing $o$ corresponds to a line decomposition which is orthogonal with respect to $o$. 
	$\R w$ is in the boundary of this flat if and only if it is one of these lines. 
	If it is, then the Busemann function along a ray in this flat is linear.
\end{example}	

\subsection{Characterization of the domain of discontinuity}
\label{sec: Subsection characterizing domains of disontinuity.}

In this section we consider an $\omega$-undistorted subgroup $\Gamma$ of $G$ that is not virtually cyclic.
In particular $\Gamma$ is Anosov for a set of roots $\Theta=\Theta(\sigma_\Gamma)$ by Proposition \ref{prop:omega-undistorted implies Anosov}.
We show that analogous results still hold for virtually cyclic subgroups in Section \ref{sec:CyclicCase}.

\medskip

Recall that $\mathcal{C} \subset \mathfrak{a}^+$ denotes a compact neighborhood of $\mathcal{C}_\Gamma$ that avoids $w\cdot \Ker(\omega)$ for all $w\in W$.
The limit map $\xi \colon \partial \Gamma \to \mathcal{F}_\Theta$ and the thickenings from the previous section provide the data to define domains in $\mathcal{F}_\omega$ and $\partial_\omega\X$, following \cite{KLP18a,KL18a}.

\begin{definition}
Let us define the following domains:
$$ \Omega^\omega_{\rm flag} \coloneqq \mathcal{F}_\omega\setminus\bigcup_{x\in \partial\Gamma} \Thomega(\xi_\Theta(x),\mathcal{C}), $$
$$ \Omega_\text{horo}^\omega \coloneqq \partial_\text{horo}^\omega\X\setminus\bigcup_{x\in \partial\Gamma} \ThomegaH(\xi_\Theta(x),\mathcal{C}). $$
\end{definition}

They can be characterized as domains of Busemann (resp.\ horofunctions) that are proper bounded from below on any/every $\Gamma$-orbit. 

\begin{proposition}
\label{prop:characterization of domains}
A point $\eta \in \mathcal{F}_\omega$ belongs to $\Omega^\omega_{\rm flag}$ if and only if the associated Busemann function $b_{\eta,o}$ restricted to the $\Gamma$-orbit of $o$ is bounded from below.

\medskip

An element $[h]\in\partial_{\rm horo}^\omega \X$ belongs to $\Omega_{\rm horo}^\omega$ if and only if $h$ restricted to the $\Gamma$-orbit of $o$ is bounded from below. In this case, the horofunction is proper on any $\Gamma$-orbit.
\end{proposition}

\begin{proof} 
Since $\mathcal{F}_\omega$ includes into $\partialomX$, the first statement follows from the second.

Let $[h]\in \partial_\text{horo}^\omega \X$ be a horofunction. 
If $[h]\in \partial_\text{horo}^\omega \X$ is not in $\Omega_\text{horo}^\omega$, then it belongs to $\ThomegaH(\xi_\Theta(z),\mathcal{C})$ for some $z\in \partial \Gamma$. 
If $(\gamma_n)$ is a geodesic ray in $\Gamma$ converging to $z$, Lemma \ref{lem:Morse Lemma} implies that there exist a constant $D\ge 0$ such that for all $n\geq 0$, $\gamma_n\cdot o$ is at distance at most $D$ from a point $x_n\in \mathcal{V}(o, \st_{\mathcal{C}}(\xi_\Theta(z)))$.

Lemma \ref{lem:Dichotomy} implies that $h(x_n)-h(o)\leq -C_{\mathcal{C},\omega} d(o, x_n)$. Since $h$ is $1$-Lipshitz, this implies that $h(\gamma_n\cdot o)$ goes to $-\infty$, so $h$ is unbounded from below.

To conclude the proof, we need to show that if $h$ is unbounded from below or fails to be proper on a $\Gamma$-orbit, then $h$ does not belong to $\Omega_{horo}^\omega$.
In either case, there is a diverging sequence of elements $(\gamma_n)_{n\in \mathbb{N}}$ of $\Gamma$ such that $h(\gamma_n\cdot o)$ is bounded from above by a constant $D \ge h(o)$.
The sequence of geodesic segments $([o, \gamma_n\cdot o])$ converges up to subsequence to a geodesic ray $[o,\eta)$ with $\eta \in \partial_\text{vis}\X$. 
Since $\Gamma$ is $\Theta$-Anosov with limit cone inside $\mathcal{C}$ one has $\eta\in \st_{ \mathcal{C}}(\xi_\Theta(z))$ for some $z \in \partial \Gamma$ by Lemma \ref{lem:Morse Lemma}.

Note that since $h$ is convex, it is bounded from above by $D$ on all the geodesic segments  $([o, \gamma_n\cdot o])$ and hence also on the geodesic ray $[o,\eta)$.
Lemma \ref{lem:Dichotomy} therefore implies that $h\in \ThomegaH(\xi_\Theta(\zeta),\mathcal{C})$, so $[h]$ does not belong to $\Omega_\text{horo}^\omega$.
\end{proof}

We show that the horofunctions belonging to $\X\cup\Omega_\text{horo}^\omega$ are not only proper and bounded from below on $\Gamma$-orbits, but are moreover locally uniformly proper. 
More precisely, the constant $A$ from Lemma \ref{lem:Dichotomy} can be chosen to be uniform on a neighborhood of $[h]$. 

\begin{lemma}
\label{lem:LocalUnif}
Let $[h_0]\in \X \cup \Omega_{\rm horo}^\omega$, and let $o\in \X$. 
There exists a neighborhood $U\subset \X\cup \partial_{\rm horo}^\omega\X$ of $[h_0]$ and a constant $A>0$ such that for $[h]\in U$ and $\gamma\in \Gamma$:
$$h(\gamma\cdot o)-h(o)\geq C_{\mathcal{C},\omega}d(o, \gamma\cdot o)-A.$$  
\end{lemma}

\begin{proof}	
Let $\mathcal{C}'$ be a compact neighborhood of $\mathcal{C}_\Gamma$ that lies in the interior of $\mathcal{C}$. 
Note that $C_{\mathcal{C}',\omega}> C_{\mathcal{C},\omega}$. 
We consider the following subset of the visual boundary:
$$E=\bigcup_{\zeta\in \partial \Gamma} \st_{\mathcal{C}'}(\xi_\Theta(\zeta)).$$

Applying Lemma \ref{lem:Dichotomy} for the subset $\mathcal{C}'$ implies that for every $\xi\in E$ there exist $t_0 > 0$ large enough that any point $x$ on the geodesic ray $[o,\xi)$ at distance $t \ge t_0$ from $o$ satisfies:
$$\frac{h_0(x)-h_0(o)}{t}> C_{\mathcal{C},\omega} .$$
Moreover the same property holds for every $\xi'$ close enough to $\xi$ and for $[h]$ close enough to $[h_0]$. Since $E$ is compact, there exist a real number $t_1$ and a neighborhood $U$ of $[h_0]$ such that for all $[h]\in U$ and $\xi\in E$ the point $x$ on the geodesic ray $[o,\xi)$ at distance $t_1$ from $0$ satisfies:
$$\frac{h(x)-h(o)}{t_1}> C_{\mathcal{C},\omega}.$$

Lemma \ref{lem:Morse Lemma} implies that for some $D>0$, for every $\gamma\in \Gamma$, $\gamma\cdot o$ is at distance at most $D$ from some point $y\in [o, \xi)$ for some $\xi\in E$. Let $[h]\in U$; the function $h-h(o)$ is convex on the geodesic ray $[o, \xi)$ and greater than $C_{\mathcal{C},\omega}t_1$ at the point $x\in [o,\xi)$ such that $d(o,x)=t_1$. 
Moreover it is $1$-Lipshitz, which implies that for all $y\in [o,\xi)$, $h(y)-h(o)\geq d(o,y)C_{\mathcal{C},\omega}-t_1C_{\mathcal{C},\omega} -t_1 $. 
Using again the fact that $h$ is $1$-Lipshitz we get:
$$h(\gamma\cdot o)-h(o)\geq C_{\mathcal{C},\omega}d(o, \gamma\cdot o) -(t_1C_{\mathcal{C},\omega}+t_1+D) .$$ 
\end{proof}

\subsection{Elementary subgroups}
\label{sec:CyclicCase}

In the previous subsection we required that the subgroup $\Gamma$ is not virtually cyclic.
The fact that such $\omega$-undistorted groups are Anosov, Proposition \ref{prop:omega-undistorted implies Anosov}, relies on the fact that their limit cone $\mathcal{C}_\Gamma$ is connected, which often fails for virtually cyclic groups.
For instance the following subgroup of $\SL(3,\R)$ has disconnected limit cone and is not Anosov, but is $\omega_1$-undistorted:
$$ \left\{ \begin{pmatrix}
	4^n&&\\
	&2^{-n}&\\
	&&2^{-n}
\end{pmatrix} \mid n \in \mathbb{Z} \right\} .$$

\medskip

In this subsection we adapt the previous results for elementary subgroups. 
Note that the only non-trivial case is the case of infinite virtually cyclic groups.

\medskip

Let $\Gamma\subset G$ be an infinite virtually cyclic group that is quasi-isometrically embedded, and $\langle \gamma\rangle\subset \Gamma$ be a finite index infinite cyclic subgroup.
The element $\gamma\in G$ admits a Jordan decomposition into $\gamma_t\gamma_e\gamma_u$ where $\gamma_t$ is a transvection, $\gamma_e$ is elliptic, and $\gamma_u$ is unipotent, and the factors commute \cite{E96}. 

\medskip

The element $\gamma_t$ is nontrivial since $\Gamma$ is quasi-isometrically embedded, so $\gamma_t$ is the transvection corresponding to an oriented geodesic axis $c \colon \R\to \X$ parametrized with speed one with endpoints $\eta_\pm \in \partial_{\rm vis}\X$. 
Let $m_g$ denote the translation length of $g \in G$. 

\begin{lemma}
	\label{lem:Logarithmic distance Cyclic case}
	Let $o \in \X$ and $\gamma \in G$ with $\gamma_t$ non-trivial. 
	There exists a constant $C$ depending on $o,\gamma$ such that:
	For all $n\in \mathbb{Z}$
	$$ d(\gamma^n \cdot o, c(n m_\gamma)) \le 2 \log(n) + C .$$
\end{lemma}

\begin{proof}
	Let $\gamma=\gamma_t\gamma_u\gamma_e$ be the Jordan decomposition. 
	Let $p$ be a fixed point of $\gamma_e$.
	Then 
	$$ d(\gamma^n \cdot o, c(n m_\gamma)) \le
	d(o, p) + d(\gamma_u^n \cdot p, p)+ d(p,c(0)) .$$
	By the proof of \cite[Claim 2.28]{GGKW17}, there exists $C'$ depending only on $\gamma_u$ and $p$ such that $d(\gamma_u^n\cdot p,p) \le 2 \log(n) + C'$. 
	We may set $C= C' + d(o,p) + d(p,c(0))$ to see the desired result.
\end{proof}

Lemma \ref{lem:Logarithmic distance Cyclic case} is weaker than the Morse Lemma, but it will be sufficient to generalize our results to cyclic groups.

\medskip

Now  let $\omega\in \mathfrak{a}^*$, and let us assume that for all $w\in W$, $\type(\eta_+)$ and $\type(\eta_-)$ do not belong to $\Ker(w\cdot \omega)$. 
This property is the equivalent to the $\omega$-undistorted condition, since the limit cone $\mathcal{C}_{\Gamma}$ consist of the two points $\type(\eta_+)$ and $\type(\eta_-)$.

In the virtually cyclic case we adapt our definitions as follows:
$$\Omega^\omega_{\rm flag}=\mathcal{F}_\omega\setminus \left\lbrace \xi\in \mathcal{F}_\omega \mid \angle_\text{Tits}(\xi, \eta_+)<\frac{\pi}{2}\;\; \text{or}\;\; \angle_\text{Tits}(\xi, \eta_-)<\frac{\pi}{2}\right\rbrace ,$$
$$\Omega_\text{horo}^\omega=\partial_\text{horo}^\omega \X \setminus \left\lbrace [h]\in \mathcal{F}_\omega \mid \;\text{slope}_h(\eta_+)<0\;\; \text{or}\;\;\text{slope}_h(\eta_-)<0 \right\rbrace .$$

We now adapt the following results.

\begin{proposition}[{Analog of Proposition \ref{prop:characterization of domains}}]
\label{prop:uniformlyProper CyclicCase}
	An element $[h]\in\partial_{\rm horo}^\omega \X$ belongs to $\Omega_{\rm horo}^\omega$ if and only if $h$ restricted to the $\Gamma$-orbit of $o$ is bounded from below. In this case, the horofunction is proper on any $\Gamma$-orbit.
\end{proposition}

\begin{proof}
Let $[h]\in \partial^\omega_{\rm horo} \X$ be a horofunction.
Let $(\gamma_n)_{n\in \mathbb{N}}$ be a diverging sequence of elements of $\Gamma$ such that $h(\gamma_n\cdot o)$ is bounded from above. 
In our case one can assume that $\gamma_n=\gamma^n$, or $\gamma^{-n}$. 
We consider the first case, as the other one is identical. 
By Lemma \ref{lem:Logarithmic distance Cyclic case}, $\gamma^n\cdot o$ is at logarithmic distance from $c(nm_g)$. 
Hence $h$ grows t most logarithmically on $c$, but then Lemma \ref{lem:Dichotomy} implies that $h$ has non-positive slope, and hence negative slope on $c$, hence $h$ does not belong to $\Omega_{\rm horo}^\omega$. 

The other part of the proof works as in Proposition \ref{prop:characterization of domains}.
\end{proof}

\begin{lemma}[{Analog of Lemma \ref{lem:LocalUnif}}]
\label{lem:LocalUnif CyclicCase}
	Let $[h_0]\in \X \cup \Omega_{\rm horo}^\omega$, and let $o\in \X$. 
	There exists a neighborhood $U\subset \X\cup \partial_{\rm horo}^\omega\X$ of $[h_0]$ and a constant $A>0$ such that for $[h]\in U$ and $n\in \mathbb{Z}$:
	$$h(\gamma^n\cdot o)-h(o)\geq C_{\mathcal{C}_\Gamma,\omega}d(o, \gamma^n\cdot o)-A.$$  
\end{lemma}

\begin{proof}
Here again we replace the Morse property by Lemma \ref{lem:Logarithmic distance Cyclic case}.
Let $\mathcal{C}$ be a neighborhood of $\mathcal{C}_\Gamma$ avoiding $w \cdot \ker(\omega)$ for all $w \in W$.  
We obtain for some $D,E,t_0>0$:
$$h(\gamma^n\cdot o)-h(o)\geq C_{\mathcal{C},\omega}d(o, \gamma^n\cdot o) -(t_0 C_{\mathcal{C},\omega}+D \log(n)+E) .$$
This is still enough to get the desired result for $A$ large enough since $C_{\mathcal{C},\omega}>C_{\mathcal{C}_\Gamma,\omega}$ and $d(o, \gamma^n\cdot o)$ grows linearly in $n$.
\end{proof}

\section{Dirichlet domains for Finsler metrics}

\label{sec:DirichletFinsler}
In this section, we consider Dirichlet domains associated to Finsler metrics.
Using the results of the previous section, we show in Theorem \ref{thm:finiteSidedURU} that such Dirichlet domains with respect to $d_\omega$ are properly finite-sided for $\omega$-undistorted subgroups.
Moreover we deduce Theorem \ref{thm:omega1-URU implies finitely-sided} in Corollary \ref{cor:omega1-URU implies properly finite sided}.
In the rest of the section we demonstrate some partial converse results. 
In Theorem \ref{thm:properly finite sided implies quasi-isometrically embedded} we show that any discrete group admitting a properly finite-sided Dirichlet domain is quasi-isometrically embedded in $\mathbb{X}$. 
In Section \ref{sec:disjoint half-spaces and URU} we show that the $\omega$-undistorted condition is equivalent to the \emph{disjoint half-space property}, see Definition \ref{def:disjoint half-space property}. 

Recall that we assume throughout the paper that $\omega \in \mfa^\ast$ defines a definite Finsler metric $d_\omega$ on $\X$. 

\subsection{\texorpdfstring{$\omega$}{omega}-undistorted implies properly finite-sided}
For $x,y \in \X$, the \emph{Finsler half-space} is 
$$ \mathcal{H}^\omega(x,y) \coloneqq \{[h] \in \X \cup \partialomX \mid h(x) \le h(y) \} .$$
It is the closure in $\X \cup \partialomX$ of the set of points $z \in \X$ satisfying $d_\omega(x,z) \le d_\omega(y,z)$.

\medskip

Let $\Gamma$ be a discrete subgroup of $G$.
\begin{definition}
	The \emph{Dirichlet domain} associated to $\Gamma$ based at $o$ with respect to the Finsler distance $d_\omega$ is given by:
	$$ \mathcal{D}^\omega_\Gamma(o) \coloneqq \left\{ [h]  \mid \forall \gamma \in \Gamma, \, h(o)\leq h(\gamma\cdot o)\right\} = \bigcap_{\gamma\in \Gamma\setminus \Gamma_o}\mathcal{H}^\omega(o, \gamma\cdot o)\subset \X\cup \partialomX. $$ 
\end{definition}

We call a Dirichlet domain $\D^\omega_\Gamma(o)$ \emph{properly finite-sided} if there exists a neighborhood $U$ of $\mathcal{D}^\omega_\Gamma(o)$ in $\X \cup \partialomX$ and a finite set $F\subset \Gamma$ such that for all $\gamma \in \Gamma\setminus F$, $U \subset \mathcal{H}^\omega(o,\gamma \cdot o)$.

\begin{theorem}
\label{thm:omega-URU implies properly finite-sided}
\label{thm:finiteSidedURU}

If $\Gamma$ is $\omega$-undistorted then for all $o\in \X$, the Dirichlet domain $\mathcal{D}^\omega_\Gamma(o)$ is properly finite-sided. 

\medskip

Moreover for any $A>0$ one can find a finite set $S\subset \Gamma$ and a neighborhood $U$ of $\mathcal{D}_\Gamma^\omega(o)$ such that for all $[h]\in U$ and $\gamma\in \Gamma\setminus S$, $h(\gamma\cdot o)> h(o)+A$. 
\end{theorem}

\begin{proof}
Every horofunction in $\mathcal{D}^\omega_\Gamma(o)$ is bounded from below on the $\Gamma$-orbit of $o$. 
Hence if $\Gamma$ is $\omega$-undistorted one has $\mathcal{D}^\omega_\Gamma(o)\subset \X\cup \Omega_\text{horo}^\omega$ by Proposition \ref{prop:characterization of domains} (or Proposition \ref{prop:uniformlyProper CyclicCase} in the elementary case).
Let $K \subset \X \cup \Omega^\omega _\text{horo}$ be a compact neighborhood of $\mathcal{D}_\Gamma^\omega(o)$. By considering an open cover of $K$, Lemma \ref{lem:LocalUnif} (or Lemma \ref{lem:LocalUnif CyclicCase}) implies the existence of constants $B,C>0$ such that for $[h]\in K$ and $\gamma\in \Gamma$:
$$h(\gamma\cdot o)-h(o)\geq C d(o, \gamma\cdot o)-B.$$

Hence for all $\gamma\in \Gamma$ such that $d(o, \gamma\cdot o)> \frac{A+B}{C}$, one has $h(\gamma\cdot o)-h(o)>A$. In particular the half-space $\mathcal{H}^\omega(o, \gamma\cdot o)$ contains $K$ for all but finitely many $\gamma\in \Gamma$, so $\mathcal{D}_\Gamma^\omega(o)$ is properly finite-sided.
\end{proof}

For $\omega$-undistorted subgroups, we can also give a direct proof that the domains $\Omega_{\rm flag}^\omega$ and $\Omega_{\rm horo}^\omega$ are properly discontinuous and cocompact.

\begin{proposition}
\label{prop:CoarseFibration}
The action of an $\omega$-undistorted subgroup $\Gamma$ on $\Omega^\omega_{\rm flag}$ and $\Omega_{\rm horo}^\omega$ is properly discontinuous and cocompact.
\end{proposition}

The idea is that these domain \emph{coarsely fiber} over $\Gamma$, via the map that associates to a horofunction its minimum on the $\Gamma$-orbit of $o\in \X$.

\begin{proof}
To a horofunction $[h]\in \Omega^{\omega}_{\rm horo}$ we associate the non-empty finite \emph{set of minima} $M_{[h]}\subset \Gamma$ of elements $\gamma_0$ such that $h(\gamma_0\cdot o)=\min_{\gamma\in \Gamma}h(\gamma\cdot o)$. 
This set is well defined and finite since $h$ is proper and bounded from below by Proposition \ref{prop:characterization of domains}. 
Moreover this association is equivariant, i.e. $M_{\gamma\cdot [h]}=\gamma M_{[h]}$ for all $\gamma\in \Gamma$.

Let $K\subset \Omega^{\omega}_{\rm horo}$ be a compact set. 
For each $[h_0]\in K$, by Lemma \ref{lem:LocalUnif} there is a neighborhood $U$ of $[h_0]$ and a finite set $M\subset \Gamma$  such that for all $[h]\in U$,  $M_{[h]}\subset M$. 
Since $K$ is compact, one can therefore find a finite set $M_K$ such that for all $[h]\in K$,  $M_{[h]}\subset M_K$. 
All but finitely many $\gamma\in \Gamma$ satisfy $\gamma\cdot M_K\cap M_K=\emptyset$. 
Therefore for all such $\gamma\in \Gamma$, $\gamma\cdot K\cap K=\emptyset$. 
Hence $\Gamma$ acts properly on $\Omega^{\omega}_{\rm horo}$.

\medskip

A horofunction $[h]\in \Omega^{\omega}_{\rm horo}$ belongs to the Dirichlet domain $\mathcal{D}^\omega_\Gamma(o)$ if (and only if) the neutral element $e\in \Gamma$ belongs to $M_{[h]}$. 
For every $[h']\in \Omega^{\omega}_{\rm horo}$, there exists some $\gamma'\in M_{[h']}$, and one has $(\gamma')^{-1}\cdot [h']\in\mathcal{D}_\Gamma^{\omega}(o)$. 
Therefore $\mathcal{D}_\Gamma^{\omega}(o) \cap \Omega^{\omega}_{\rm horo} = \mathcal{D}_\Gamma^{\omega}(o) \cap \partial_{\rm horo}^{\omega} \X$ is a fundamental domain for the action of $\Gamma$. 
Moreover it is closed in the metrizable compact $\partial_{\rm horo}^{\omega} \X$, hence it is compact. 
So this is a compact fundamental domain for the action of $\Gamma$ on $\Omega^{\omega}_{\rm horo}$, and the action is cocompact.
\end{proof}

\begin{remark}\label{rmk: comparing dods to KL,KLP,GW}
	The fact that $\Omegaom$ is a cocompact domain of proper discontinuity is a special case of a result due to Kapovich-Leeb-Porti \cite{KLP18a}, and Guichard-Wienhard \cite{GW12} for certain cases, such as $\omega_1$-undistorted subgroups of $\PSL(2n,\R)$.
	The construction of $\OmegaomH$ from a thickening is similar to a construction due to Kapovich-Leeb \cite{KL18a}, where they consider the case when $\omega$ is dual to a regular point of $\sigma_{mod}$. 
	They prove proper discontinuity and cocompactness when $\Gamma$ is an arbitrary Anosov subgroup. 
	Since their Finsler compactification is the maximal Satake compactification, it dominates the compactification we consider here, and the proper discontinuity and cocompactness follows.
	In fact, it follows from Proposition \ref{prop:characterization of domains} and their Theorem that the Dirichlet domains we consider are properly finite-sided for $\omega$-undistorted subgroups, without the use of Lemma \ref{lem:LocalUnif}.	
\end{remark}

\subsection{Properly finite-sided implies undistorted}
We now prove the following necessary condition for a group to admit a properly finite-sided Dirichlet domain. 

\begin{theorem}
\label{thm:properly finite sided implies quasi-isometrically embedded}
Let $\Gamma$ be a discrete subgroup of $G$ and suppose that a Dirichlet domain $\D^\omega_\Gamma(o)$ is properly finite-sided. 
The orbit map $\Gamma \to \X$ is a quasi-isometric embedding.
\end{theorem}

This result is an adaptation of the Milnor-Schwarz Lemma, replacing the cocompactness of the action by the fact that the Dirichlet domain is \emph{tame at infinity}.

\begin{lemma}
\label{lem:SchwarzMilnor}
Let $\Gamma$ be a group acting by isometries on a geodesic metric space $X$. 
Suppose that there exists a subset $D\subset X$ such that:
$$X=\bigcup_{\gamma\in \Gamma}\gamma\cdot D.$$
Suppose moreover that there exist a finite subset $F\subset \Gamma$ and $\epsilon>0$ such that the $\epsilon$-neighborhood $D_\epsilon$ of $D$ satisfies for all $\Gamma\in \Gamma\setminus F$:
$$ \gamma \cdot D \cap D_\epsilon=\emptyset.$$
For any $o\in X$ the orbit map $\gamma\in \Gamma \mapsto \gamma\cdot o\in X$ is a quasi-isometric embedding. 
\end{lemma}

\begin{proof}
We consider the word metric on $\Gamma$ defined for $\gamma\in \Gamma$ by:
$$|\gamma|=\min\{n|\gamma=s_1s_2\cdots s_n,\, s_i\in F\} .$$
Any other word metric with respect to a finite generating set is quasi-isometric to this one. 
Let $o\in X$. 
Let $A=\max_{s\in F}d(o,s\cdot o)$. 

Let $\gamma\in \Gamma$ be any element. 
First note that $d(o, \gamma\cdot o)\leq A|\gamma|$. 
Now let $n=\lceil\frac{d(o, \gamma\cdot o)}{\epsilon} \rceil$. We consider a sequence $x_0, x_1, \cdots, x_n$ of points on a geodesic in $X$ between $o$ and $\gamma\cdot o$ with $x_0=o, x_n=\gamma\cdot o$ and such that $d(x_i,x_{i+1})\leq \epsilon$ for $0\leq i<n$. 
Since $X=\bigcup_{\gamma\in \Gamma}\gamma\cdot D$, there exist for all $1\leq i<n$ an element $\gamma_i\in \Gamma$ such that $x_i\in \gamma_i\cdot D$. 
We set $\gamma_0=e$ and $\gamma_n=\gamma$. 

By the definition of $F$ we know that for all $1\leq i<n$, $\gamma_i^{-1}\gamma_{i+1}\in F$. 
Indeed $x_{i+1}\in \gamma_{i+1}\cdot D\cap \gamma_i\cdot D_\epsilon$ and hence $\gamma_{i}^{-1}\cdot x_{i+1}\in \gamma_i^{-1}\gamma_{i+1}\cdot D\cap D_\epsilon$. 
Therefore $\gamma$ can be written as the product of $n$ elements of $F$, so $|\gamma|\leq n$. 
Hence:
$$\epsilon |\gamma|-\epsilon\leq d(o, \gamma\cdot o)\leq A|\gamma|. $$
This concludes the proof.
\end{proof}

\begin{proof}[Proof of Theorem \ref{thm:properly finite sided implies quasi-isometrically embedded}]
Let $U$ be an open neighborhood of $\mathcal{D}^\omega_\Gamma(o)\subset \X\cup \partial_\omega \X$ such that there exist only finitely many $\gamma\in \Gamma$ such that $\gamma\cdot \mathcal{D}^\omega_\Gamma(o) \cap U\neq \emptyset$. 
In order to apply Lemma \ref{lem:SchwarzMilnor}, it suffices to prove that for some $\epsilon>0$, the intersection $U\cap \X$ contains the $\epsilon$-neighborhood of $\mathcal{D}^\omega_\Gamma(o) \cap \X$. Let $U^c$ be the complement of $U$ in $\X\cup \partial_\omega \X$, which is a compact set.

\medskip

Suppose the contrary; then there exists a sequence $(x_n)$ of points in $\mathcal{D}^\omega_\Gamma(o)\cap \X$ and a sequence $(y_n)$ of points in $U^c\cap \X$ such that $d(x_n,y_n)$ converges to zero (note that we consider here the Riemannian metric). 
Up to taking a subsequence, one can assume that the sequences converge to $x_\infty$ in $\mathcal{D}^\omega_\Gamma(o)$ and to $y_\infty$ in $U^c$ respectively. 
But since $d(x_n,y_n)$ converges to zero, the function $x\in \X\mapsto d_\omega(x,x_n)-d_\omega(x,y_n)$ also converges to zero, uniformly on $\X$. Therefore $x_\infty=y_\infty$, which is not possible since $\mathcal{D}^\omega_\Gamma(o)\subset U$.
\end{proof}

\subsection{Disjoint half-spaces and the \texorpdfstring{$\omega$}{omega}-undistorted condition}
\label{sec:disjoint half-spaces and URU}
In this section, we show that $\omega$-undistorted subgroups can be characterized by having sufficiently disjoint half-spaces for the Finsler distance $d_\omega$.

\begin{definition}\label{def:disjoint half-space property}
We say that a finitely generated subgroup $\Gamma$ satisfies the \emph{$\omega$-disjoint half-space property} if for some $o \in \X$, some word metric on $\Gamma$ and some integer $D>0$, for all triples $(x,y,z)$ in $\Gamma$ that lie in this order on a geodesic such that $d_\Gamma(x,y)=d_\Gamma(y,z)=D$, the half-space $\mathcal{H}^\omega(x\cdot o,y\cdot o)$ is disjoint from  $\mathcal{H}^\omega(z\cdot o,y\cdot o)$.

We say that $\Gamma$ satisfies the \emph{$\omega$-flag disjoint half-space property} if for some $o \in \X$, some word metric on $\Gamma$ and some integer $D>0$, for all triples $(x,y,z)$ in $\Gamma$ that lie in this order on a geodesic such that $d_\Gamma(x,y)=d_\Gamma(y,z)=D$, the intersection of $\mathcal{F}_\omega$ with the half-space $\mathcal{H}^\omega(x\cdot o,y\cdot o)$ is disjoint from $\mathcal{H}^\omega(z\cdot o,y\cdot o)$. 
\end{definition}

\medskip

\newcommand{\hgline}[2]{
\pgfmathsetmacro{\thetaone}{#1}
\pgfmathsetmacro{\thetatwo}{#2}
\pgfmathsetmacro{\theta}{(\thetaone+\thetatwo)/2}
\pgfmathsetmacro{\phi}{abs(\thetaone-\thetatwo)/2}
\pgfmathsetmacro{\close}{less(abs(\phi-90),0.0001)}
\ifdim \close pt = 1pt
    \draw[blue] (\theta+180:1) -- (\theta:1);
\else
    \pgfmathsetmacro{\R}{tan(\phi)}
    \pgfmathsetmacro{\distance}{sqrt(1+\R^2)}
    \draw[blue] (\theta:\distance) circle (\R);
\fi
}

\begin{figure}
\begin{center}
\begin{tikzpicture}[scale=2]
\draw (0,0) circle (1);
\clip (0,0) circle (1);
\hgline{30}{-30}
\hgline{180}{270}
\fill (0,0) circle (.02);
\node[above] () at (0,0) {$y\cdot o$};
\fill (0.8,0) circle (.02);
\node[above] () at (0.8,0) {$z\cdot o$};
\fill (-.5,-.5) circle (.02);
\node[above] () at (-.5,-.5) {$x\cdot o$};
\end{tikzpicture}
\end{center}
\caption{Illustration of the disjoint half-space property.}
\label{fig:DisjointBissector}

\end{figure}

Recall that the half-spaces are closed subsets of $\X\cup \partial_\omega\X$. 
A priori the flag property is weaker, but we see later that when $\omega$ is symmetric, the two are equivalent.

\begin{remark}
If the (flag) disjoint half-space property holds for all triples $(x,id,z)$ in $\Gamma$ with $x,z$ of word length $D$ then it holds for all triples in $\Gamma$.
In particular the (flag) disjoint half-space property can be verified on a finite subset of triples in $\Gamma$. 

When $G=\SL(d,\R)$, and $\omega=\omega_1$, the flag disjoint half-space property can be rephrased in terms of nestedness of quadric hypersurfaces. Indeed the intersection of the corresponding bisectors with $\mathcal{F}_{\omega_1}=\mathbb{RP}^{d-1}$ is the zero set of a quadratic form, as in Figure \ref{fig:Octogon}. This intersection also coïncide with the intersection of the Selberg bisectors with $\mathbb{RP}^{d-1}$, introduced Section \ref{sec:Dirichlet-Selberg domains}.
\end{remark}

\begin{theorem}\label{thm:URU implies disjoint half-spaces}
	Let $\Gamma$ be an $\omega$-undistorted subgroup. 
	It satisfies the $\omega$-disjoint half-space property and the $(-\omega)$-disjoint half-space property.
\end{theorem}

\begin{proof}
Let $\mathcal{G}$ be the space of bi-infinite geodesic $\eta:\Z\to \Gamma$ such that $\eta(0)=e$.
This is a compact set for the standard compact-open topology.

Let $\eta\in \mathcal{G}$.
We denote by $\eta^+$ and $\eta^-$ respectively the endpoints in $\partial \Gamma$ of the geodesic ray when $n$ goes respectively to $+\infty$ and $-\infty$. 
The thickenings $\text{Th}^\omega_\text{horo}(\xi_\Theta(\eta^+),\mathcal{C}_\Gamma)$ and $\text{Th}^\omega_\text{horo}(\xi_\Theta(\eta^-),\mathcal{C}_\Gamma)$ are disjoint since the flags $\xi_\Theta(\eta^+)$ and $\xi_\Theta(\eta^-)$ are transverse. 

Let $U\subset \X\cup \partialomX$ be an open set containing $\text{Th}^\omega_\text{horo}(\xi_\Theta(\eta^+),\mathcal{C}_\Gamma)$ whose closure is disjoint from $\text{Th}^\omega_\text{horo}(\xi_\Theta(\eta^-),\mathcal{C}_\Gamma)$.
Lemma \ref{lem:LocalUnif} implies that all $[h]$ in the complement of $U$ go to $+\infty$ locally uniformly along $\eta(n)$, and all $[h]$ in $\overline{U}$ go to $+\infty$ locally uniformly along $\eta(-n)$. 
If $\Gamma$ is elementary we can apply Lemma \ref{lem:LocalUnif CyclicCase} instead.
Since $\overline{U}$ and $U^c$ are compact, the local uniform behavior is global.
Hence there exists $n_0\in \mathbb{N}$ such that for all $n\geq n_0$, $\mathcal{H}^\omega(\eta(n)\cdot o,o)\subset U$ and $\overline{U}$ is contained in the complement of $\mathcal{H}^\omega(\eta(-n)\cdot o,o)$.

We write $V_{\eta,m}\subset \mathcal{G}$ for the open and closed set of geodesics $\eta':\mathbb{Z}\to \Gamma$ such that $\eta'_{|[-m,m]}=\eta_{|[-m,m]}$, for the compact open topology. 
The collection of neighborhoods $\{V_{\eta,n_0}\mid \eta \in \mathcal{G}\}$ covers $\mathcal{G}$, so it admits a finite subcover.
Hence there exists $m_0 \in \mathbb{N}$ such that for all $\eta\in \mathcal{G}$, the half-spaces $\mathcal{H}^\omega(\eta(m_0)\cdot o,o)$ and $\mathcal{H}^\omega(\eta(-m_0)\cdot o,o)$ are disjoint.
Therefore $\Gamma$ satisfies the $\omega$-disjoint half-space property. 
Since $\Gamma$ is also $(-\omega)$-undistorted, $\Gamma$ also satisfies the $(-\omega)$-disjoint half-space property.
\end{proof}

The following result can be seen as a coarse analogue of \cite[{Proposition 5.19}]{Dav23}, and has a similar proof.

\begin{theorem}
\label{thm:disjoint half-spaces implies URU}
Let $\Gamma$ be a finitely generated subgroup of $G$.
If $\Gamma$ satisfies the $\omega$-flag disjoint half-space property and the $(-\omega)$-flag disjoint half-space property, it is $\omega$-undistorted.
\end{theorem}

To prove this theorem, we first relate the disjoint half-space property to the convexity of Busemann functions.
Let $\epsilon>0$. 
We say that a sequence $(s_n)_{n\in \Z}$ is \emph{$\epsilon$-convex at critical points} if for all $n\in \Z$ such that $s_{n+1}-s_n\geq -\epsilon$ one has $s_{n+2}-s_{n+1}\geq \epsilon$. 
If for such a sequence one has $s_{-1}\leq s_0$, then for all $n\in \mathbb{N}$: 
$$s_n\geq \epsilon n +s_0 .$$

\begin{lemma}
\label{lem:epsilonConvex}
Suppose that for some $o\in \X$, $\Gamma$ satisfies the flag disjoint half-space property. Then there exist $D\in \mathbb{N}$ and $\epsilon>0$ such that for every geodesic sequence $(\gamma_n)_{n\in \Z}$ in $\Gamma$ and every $[b_\xi]\in \mathcal{F}_\omega\subset\partial^\omega_\text{horo}\X$, the sequence $\left(b_\xi(\gamma_{Dn}\cdot o)\right)$ is $\epsilon$-convex at critical points.
\end{lemma}

\begin{proof}
We assume that $\mathcal{D}_\Gamma^\omega(o)$ satisfies the flag disjoint half-space property and let $D$ be the constant from the definition.
Suppose to the contrary that there exist sequences $(x_n)$, $(y_n)$ and $(z_n)$ in $\Gamma$ and a sequence $[b_n]\in \mathcal{F}_\omega\subset\partial^\omega_\text{horo}\X$ such that for all $n\in \mathbb{N}$, $d(x_n,y_n)=d(y_n,z_n)=D$ while also $\liminf b_n(y_n\cdot o)-b_n(x_n\cdot o)\geq 0$ and $\limsup b_n(z_n\cdot o)-b_n(y_n\cdot o)\leq 0$.

Up to acting by $\Gamma$, one can assume that the sequence $(y_n)$ is constant and equal to the identity element of $\Gamma$.
Up to taking a subsequence one can assume that the sequences $(x_n)$ and $(z_n)$ are constant and equal to $x$ and $z$ respectively. 

By the flag disjoint half-space property, $\mathcal{F}_\omega \cap \mathcal{H}^\omega(x\cdot o,o)$ is disjoint from 
$$\mathcal{H}^\omega(z\cdot o,o) = \{[h] \mid h(z\cdot o)-h(o) \le 0 \}.$$
In particular, the function $[b] \mapsto b(o)-b(z\cdot o)$ is continuous and positive on the compact set $\mathcal{F}_\omega \cap \mathcal{H}^\omega(x\cdot o,o)$, so has a positive minimum $\eta$. 
Hence if any $[b]$ in $\mathcal{F}_\omega$ satisfies $b(o)-b(x\cdot o) \le 0$, then $b(z \cdot o)-b(o) \le -\eta$.  
This contradicts the assumptions on $(x_n)$, $(y_n)$ and $(z_n)$.
\end{proof}

In order to handle the disjoint half-space property at two different basepoints, we need to improve the constant $\epsilon$ by coarsifying the sequence. 

\begin{lemma}\label{lem:coarsifying a convex sequence}
	Suppose the sequence $(x_n)$ is $\epsilon$-convex at critical points.
	For any positive integer $N$, the sequence $(x_{nN})$ is $(N\epsilon)$-convex at critical points.
\end{lemma}

\begin{proof}
	If $x_{(n+1)N} - x_{nN} \ge -N\epsilon$ then there exists an integer $nN \le k < (n+1)N$ such that $x_{k+1}-x_k \ge -\epsilon$.
	For an $\epsilon$-convex sequence, if there exists $k$ such that $x_{k+1}-x_k \ge -\epsilon$, then for all $k''>k'>k$, it holds that $x_{k''}-x_{k'} \ge (k''-k')\epsilon$.
	In particular, $x_{(n+2)N}-x_{(n+1)N} \ge N \epsilon$. 
\end{proof}

\begin{proof}[{Proof of Theorem \ref{thm:disjoint half-spaces implies URU}}] 	
Let $\Gamma$ be a finitely generated subgroup of $G$ satisfying the $\omega$-flag disjoint half-space property for the basepoint $o$ and the $(-\omega)$-flag disjoint half-space property for another basepoint $o'$.

We first show that one can assume $o'=o$, up to replacing $D$ by some $D'>0$.
By Lemma \ref{lem:epsilonConvex}, there exists $D$ such that for every $[b] \in \mathcal{F}_{-\omega}$, and geodesic $(\gamma_n)$ the sequence $(b(\gamma_{nD}\cdot o'))$ is $\epsilon$-convex at critical points. 
By Lemma \ref{lem:coarsifying a convex sequence}, for any positive integer $N$, the sequence $(b(\gamma_{nDN}\cdot o'))$ is $\epsilon N$-convex at critical points. 
If $\epsilon N - 2d(o,o') \ge \epsilon$, then the sequence $(b(\gamma_{nDN}\cdot o))$ is also $\epsilon$-convex at critical points, since Busemann functions are $1$-Lipschitz. We set $D'=DN$.
	
\medskip
	
	Now let $\gamma\in \Gamma$ and let $e=\gamma_0, \gamma_1,\cdots, \gamma_{N}=\gamma$ be a geodesic sequence. 
	We consider a maximal flat passing through $o$ and $\gamma\cdot o$.
	In the visual boundary of this flat there exist $\xi_1\in \mathcal{F}_\omega$ and $\xi_2\in \mathcal{F}_{-\omega}$ such that the associated Busemann functions $[b_{\xi_1}]$ and $[b_{\xi_2}]$ satisfy:
	$$\omega(\vec{d}(o, \gamma\cdot o)=b_{\xi_1}(\gamma\cdot o)-b_{\xi_1}(o),$$
	$$-\omega(\vec{d}(o, \gamma\cdot o)=b_{\xi_2}(\gamma\cdot o)-b_{\xi_2}(o).$$
	Note also that $x\mapsto b_{\xi_1}(x)-b_{\xi_1}(o)+b_{\xi_2}(x)-b_{\xi_2}(o)$ is nonnegative on $\X$, since it is a convex function that vanishes at $o$ and whose gradient also vanishes at $o$. 
	Hence up to exchanging $\omega$ and $-\omega$, one can assume that:
	$$b_{\xi_1}(\gamma_{D'}\cdot o)-b_{\xi_1}(o)\geq 0.$$
	Letting $n$ be the integer part of $\frac{N}{D'}$, we have that 
	$$b_{\xi_1}(\gamma_{D'n}\cdot o)-b_{\xi_1}(o)\geq \epsilon (n-1).$$
	Let $E>0$ be the maximum distance between $o$ and $\gamma_0\cdot o$ for $\gamma_0\in \Gamma$ at distance at most $D'$ from the identity.
	Then:
	$$\omega\left(\vec{d}(o, \gamma\cdot o)\right)=b_{\xi_1}(\gamma\cdot o)-b_{\xi_1}(o)\geq b_{\xi_1}(\gamma_{nD'}\cdot o)-b_{\xi_1}(o)-E\geq \epsilon n-E-\epsilon.$$
	
	Hence $\Gamma$ is $\omega$-undistorted.
\end{proof}

\section{Dirichlet-Selberg domains}

\label{sec:Dirichlet-Selberg domains}
\subsection{Selberg's construction}

We recall a construction that goes back to Selberg \cite{Sel60} of a projectively convex fundamental domain for discrete subgroups $\Gamma < \PSL(d,\R)$ acting on the projective model of the associated symmetric space. These are called Dirichlet-Selberg domains and have been recently studied by Kapovich \cite{Kap23} and Du \cite{du2024geometry,du2025busemannselbergfunctionscompletenessdirichletselberg}. In the present section, we give some examples of discrete, even Anosov, subgroups admitting infinitely-sided Dirichlet-Selberg domains, and also show that $\omega_1$-undistorted subgroups have properly finite-sided Dirichlet-Selberg domains.

\medskip

Let $V$ be a $d$-dimensional real vector space. 
Let $S^2V$ be the space of symmetric bilinear tensors $Q:V^*\to V$ with $Q^\ast=Q$. 
We consider the subspace $\mathcal{X}=\mathcal{X}(V)=\mathbb{P}(S^2V^{>0})$ of $\mathbb{P}(S^2V)$ consisting of positive symmetric $2$-tensors, i.e.\ positive definite symmetric bilinear forms on $V^\ast$.
The Lie group $\SL(V)\simeq \SL(d,\R)$ acts naturally on $S^2V$ so that for $Q\in S^2V$ and $g\in \SL(V)$:
$$g\cdot Q=g \circ Q\circ g^\ast.$$
Hence $\PSL(V)$ acts on $\mathbb{P}(S^2V)$, and preserves $\mathcal{X}$. 
This action is moreover transitive on $\mathcal{X}$, and the stabilizer of an element $[Q]\in \mathcal{X}$ is equal to the subgroup $\PSO(Q)\simeq \PSO(d,\R)$ of $\PSL(V)\simeq\PSL(d,\R)$. 
Hence $\mathcal{X}$ can be identified with the symmetric space $\X=\PSL(d,\R)/\PSO(d,\R)$ associated to $\PSL(d,\R)$. 
The space $\mathcal{X}$ is called the \emph{projective model} for this symmetric space.

\medskip

Given $x_1,x_2\in \mathcal{X}$, we choose any representatives $Q_1,Q_2\in S^2V$ of the corresponding lines so that $\det(Q_1^{-1}Q_2)=1$. 
The \emph{Selberg invariant} is given by:
$$\mathfrak{s}(x_1,x_2) \coloneqq \log\left(\frac1{d} \Tr(Q_1^{-1}Q_2)\right).$$

It is asymmetric and fails the triangle inequality, but has other good properties in common with metrics.

\begin{proposition}
	\label{prop:SelbergInv}
	Let $x_1,x_2\in \mathcal{X}$ and $g \in \PSL(V)$:
	\begin{itemize}
		\item[-] $\mathfrak{s}(x_1,x_2)=0$ if and only if $x_1=x_2$.
		\item[-] $\mathfrak{s}(x_1,x_2)\geq 0$.
		\item[-] $\mathfrak{s}(gx_1,gx_2)=\mathfrak{s}(x_1,x_2)$.
		\item[-] $\mathfrak{s}(x_1,x_2)=\log\left(\frac1{d}\sum_{i=1}^d e^{\lambda_i}\right)$, where $Q_1^{-1}Q_2:V\to V$ is conjugate to:
		$$\begin{pmatrix}
			e^{\lambda_1}&&\\
			&\ddots&\\
			&&e^{\lambda_d}
		\end{pmatrix}.$$
	\end{itemize}
\end{proposition}

The closure of the properly convex domain $\mathcal{X}$ in $\mathbb{P}(S^2V)$ is called the \emph{Satake compactification} of $\mathcal{X}$, denoted by $\overline{\mathcal{X}}$ with boundary $\partial \mathcal{X}$. 
For $h\in \overline{\mathcal{X}}=\partial \mathcal{X}\cup \mathcal{X}$ and $o,x\in \mathcal{X}$ one can not always define $\mathfrak{s}(x,h)$ but one can make sense of the difference $\mathfrak{s}(x,h)-\mathfrak{s}(o,h)$. Indeed let $S, Q_0,Q$ be representatives of $h,o,x$ respectively such that $\det(Q_o^{-1}Q)=1$:
$$\mathfrak{s}_{o}(x,h):=\log\left(\frac1{d}\Tr(Q^{-1}S)\right)- \log\left(\frac1{d}\Tr(Q_0^{-1}S)\right).$$
Note that this definition does not depend on the chosen representatives, and that it satisfies the cocycle condition $\mathfrak{s}_{o}(x,h)=\mathfrak{s}_{o'}(x,h)+\mathfrak{s}_{o}(o',h)$ for $o'\in \X$.

\medskip

The main advantage of the Selberg invariant over the invariant Riemannian metric of $\mathcal{X}$ is that the bisectors (resp.\ half-spaces) of $\mathfrak{s}$ are projective hyperplanes (resp.\ half-spaces) intersected with $\mathcal{X}$.

For $x_1\neq x_2$ in $\mathcal{X}$, we let $\mathcal{H}(x_1,x_2)$ denote the \emph{closed half-space} in $\overline{\mathcal{X}}$, defined as the set of $y\in \overline{\mathcal{X}}$ such that:
$$\Tr\left((X_1^{-1}-X_2^{-1})Y\right)\geq 0 $$
where $X_1$ and $X_2$ are positive definite representatives of $x_1$ and $x_2$ respectively such that $\det(X_1X_2^{-1})=1$ and $Y$ is a positive semidefinite representative of $y$. 
It is easy to check that the half-space is also given by 
$$ \mathcal{H}(x_1,x_2) = \lbrace y\in \overline{\mathcal{X}}  \mid \mathfrak{s}_o(x_1,y)\leq \mathfrak{s}_o(x_2,y)\rbrace, $$
and note that this is independent of the basepoint $o \in \mathcal{X}$. 
For $x_1 \ne x_2$ in $\mathcal{X}$, the \emph{Selberg bisector} is the set
$$ \Bis(x_1,x_2) \coloneqq \{y \in \overline{\mathcal{X}} \mid \mathfrak{s}_o(x_1,y) = \mathfrak{s}_o(x_2,y) \}, $$
and can also be written as the subset consisting of all $y \in \overline{\mathcal{X}}$ satisfying 
$$ \Tr(X_1^{-1}Y)= \Tr(X_2^{-1}Y) $$
where $X_1,X_2,Y$ are representatives of $x_1,x_2,y$ respectively with $X_1.X_2$ positive definite and $\det(X_1^{-1}X_2)=1$. 

\medskip

Given a discrete subgroup $\Gamma$ of $\PSL(V)$ we may consider a variation of the Dirichlet domain associated to the Selberg invariant. 
Let $o\in \mathcal{X}$, and define the \emph{Dirichlet-Selberg domain} based at $o$ by:
$$\mathcal{DS}_\Gamma(o) \coloneqq \lbrace x \in \overline{\mathcal{X}}\mid \forall \gamma \in \Gamma, \, \mathfrak{s}_o(o,x)\leq\mathfrak{s}_o(\gamma\cdot o,x)\rbrace= \bigcap_{\gamma\in \Gamma\setminus \Gamma_o} \mathcal{H}(o, \gamma\cdot o). $$ 
This domain is in general a compact convex subset of $\overline{\mathcal{X}} \subset \mathbb{P}(S^2V)$. 

\medskip

\begin{definition}
	\label{defn:PropSided}
	We call a Dirichlet-Selberg domain $\DS_\Gamma(o)$ \emph{properly finite-sided} if there exists a neighborhood $U$ of $\mathcal{DS}_\Gamma(o)$ in $\overline{\mathcal{X}}$ and a finite set $F\subset \Gamma$ such that for all $\gamma \in \Gamma\setminus F$, $U \subset \mathcal{H}(o,\gamma \cdot  o).$
\end{definition}

In particular if a Dirichlet-Selberg domain is properly finite-sided then there exists a finite set $F\subset \Gamma$ such that $$\mathcal{DS}_\Gamma(o) = \bigcap_{\gamma\in F}\mathcal{H}(o, \gamma\cdot o).$$

\medskip

These definitions can be related to purely geometric notions of convex subsets of $\mathcal{X}$.
We adapt the definitions of Ratcliffe \cite{Rat19} for real hyperbolic space to $\mathcal{X}$.
A \emph{side} of a convex subset $C$ of $\mathcal{X}$ is a nonempty maximal convex subset of $\partial C \subset \mathcal{X}$. 
A \emph{convex polyhedron in $\mathcal{X}$} is a nonempty closed convex subset of $\mathcal{X}$ such that the collection of its sides is locally finite in $\mathcal{X}$. 

\begin{proposition}[\cite{Kap23}]\label{prop:ds domains are convex polyhedra}
	For any discrete $\Gamma < \PSL(V)$, the Dirichlet-Selberg domain $\mathcal{DS}_\Gamma(o) \cap \mathcal{X}$ is a convex polyhedron in $\mathcal{X}$.
\end{proposition}

The proof follows from the fact that the Selberg invariant is comparable with the Riemannian metric or any $G$-invariant Finsler metric on the symmetric space $\mathcal{X}$, see Lemma \ref{lem:comparing Finsler and Selberg}.
Proposition \ref{prop:ds domains are convex polyhedra} may also be deduced from a result of Jaejeong Lee \cite{Lee08} which applies to more general properly convex domains, see also \cite[Section 2.5]{Mar09}.

\begin{proposition}\label{prop:equivalent definitions of finite-sided}
	The domain $\mathcal{DS}_\Gamma(o) \cap \mathcal{X}$ has finitely many sides if and only if there exists a finite subset $F\subset \Gamma$ such that:
	$$\mathcal{DS}_\Gamma(o) = \bigcap_{\gamma\in F}\mathcal{H}(o, \gamma\cdot o).$$
\end{proposition}

\begin{proof}
	A convex polyhedron in $\mathcal{X}$ has finitely many sides if and only if it is the intersection of $\mathcal{X}$ with finitely many closed half-spaces.
	This is proved by Ratcliffe for hyperbolic space, see \cite[Theorems 6.3.2]{Rat19}, but the proof goes through for properly convex domains in general.
	In particular if a Dirichlet-Selberg domain may be represented as a finite intersection of closed half-spaces, then its intersection with $\mathcal{X}$ has finitely many sides.
	
	\medskip
	
	On the other hand, suppose that $\mathcal{DS}_\Gamma(o) \cap \mathcal{X}$ has finitely many sides. 
	Each of its sides spans a bisector $\Bis(o,\gamma \cdot o)$, by the same proof as \cite[Theorem 6.7.4(1)]{Rat19}.
	Moreover, distinct sides span distinct bisectors by convexity. 
	It follows that there is a finite set $F \subset \Gamma$ such that 
	\begin{equation}\label{eqn:ds domain intersect X}
		\mathcal{DS}_\Gamma(o) \cap \mathcal{X} = \bigcap_{\gamma\in F}\mathcal{H}(o, \gamma\cdot o) \cap \mathcal{X}.
	\end{equation}
	In general, if $Y$ is a closed convex subset of $\overline{\mathcal{X}}$ and $Y$ contains a point in $\mathcal{X}$, then $Y = \overline{Y \cap \mathcal{X}}$.
	So by taking closures in (\ref{eqn:ds domain intersect X}) we can conclude the proof.
\end{proof}

\subsection{Infinitely-sided Dirichlet Selberg domains}

In this subsection we study Dirichlet-Selberg domains for lattices $\Gamma$ in the subgroup $\SO(1,n)<\SL(n+1,\R)$ of elements that preserve a symmetric bilinear form $\langle \cdot,\cdot \rangle$ of signature $(1,n)$. 
We show that the Dirichlet-Selberg domain is  infinitely-sided for some specific basepoints $o \in \mathcal{X}$.

\medskip

Let $q$ be the symmetric bilinear form of signature $(1,n)$ preserved by $\SO(1,n)$.
The subgroup $\SO(1,n)\subset \SL(n+1,\R)$ preserves a totally geodesic copy $\mathcal{H}\subset \mathcal{X}_{n+1} = \mathcal{X}(\mathbb{R}^{n+1})$ of the hyperbolic space $\mathbb{H}^n$. 
Indeed $\mathbb{H}^n$ can be seen as the space of lines on which the symmetric  bilinear form $q$ is positive. 
To such a line $\ell$ whose orthogonal for $q$ is $\ell^\perp$, we associate the inverse $x_\ell=q_\ell^{-1}:V^\ast \to V$ of the symmetric bilinear form $q_\ell:V\to V^\ast$ for which $\ell$ and $\ell^\perp$ are orthogonal, and such that $q_\ell=q$ on $\ell$ and $q_\ell=-q$ on $\ell^\perp$. 

\begin{theorem}
	\label{thm:InfiniteSided}
	Let $\Gamma$ be a lattice in $\SO(1,n)$ and let $o \in \mathcal{H}$.
	The Dirichlet-Selberg domain $\mathcal{DS}_\Gamma(o)\cap \mathcal{X}$ has infinitely many sides.
\end{theorem}

In particular, uniform lattices in $\SO(1,n)$ are projective Anosov subgroups of $\SL(n+1,\R)$, and these admit Dirichlet-Selberg domains in $\mathcal{X}_{n+1} = \SL(n+1,\R)/\SO(n+1)$ with infinitely many sides.  
When $n=2$, a uniform lattice in $\SO(1,2)$ includes as a Borel Anosov subgroup of $\SL(3,\R)$.

\medskip

The main ingredient of the proof is to understand the intersection of $\mathbb{RP}^{n}$ with the Selberg bisector between $x_\ell$ and $x_{\ell'}$ for $\ell\neq \ell'\in \mathbb{H}^n$.  Recall that the subset consisting of rank one symmetric tensors in $\overline{\mathcal{X}_{n+1}}$ is in one-to-one correspondence with $\mathbb{RP}^n$. We identify these two spaces in all this subsection. 
When $x,y \in S^2(\mathbb{R}^{n+1})$ are positive of the same determinant, the intersection of the Selberg bisector $\Bis([x],[y])$ with $\mathbb{RP}^n$ is the zero set of the quadratic form $x^{-1}-y^{-1}$.

\medskip

We can make the following two observations:

\begin{lemma}
	\label{lem:InfiniteSidedBissectors}
	For every $\ell\neq \ell'\in \mathbb{H}^n$, the intersection of the half-space $\mathcal{H}(x_{\ell'},x_{\ell})$ and $\mathbb{P}(\ell^\perp)\subset \mathbb{RP}^n$ is the hyperplane $\mathbb{P}\left(\ell^\perp\cap (\ell')^\perp\right)$. 
	
	If $\ell$ is fixed and $\ell'$ converges to $u\in \partial \mathbb{H}^n\subset \mathbb{RP}^n$, then the bisectors $\Bis(x_\ell,x_{\ell'}) \cap \mathbb{RP}^n$ converge for the Hausdorff topology to $\mathbb{P}(u^\perp)$. 
\end{lemma}

\begin{proof}
	In the present case, the intersection of $\Bis(q_\ell,q_{\ell'})$ and $\mathbb{RP}^n$ is equal to the quadratic form $q_\ell-q_{\ell'}$. 
	This bilinear form vanishes on $\ell^\perp\cap (\ell')^\perp$ but restricts to a form of signature $(1,1)$ on $\ell \oplus \ell'$.
	Hence this symmetric bilinear form has signature $(1,1,n-1)$, and the zero locus  of the corresponding quadric is the intersection of two distinct projective hyperplanes whose intersection is $\ell^\perp\cap (\ell')^\perp$. 
	Moreover neither $\ell$ nor $\ell'$ belongs to the intersection of $\Bis(q_\ell,q_{\ell'})$ and $\mathbb{RP}^n$.
	
	\medskip
	
	The intersection of the half-space $\mathcal{H}(x_{\ell'},x_{\ell})$ and $\mathbb{P}(\ell^\perp)\subset \mathbb{RP}^n$ is  therefore the set of lines on which $q_{\ell'} \ge q_{\ell}$, but $q_\ell= -q$ on $\ell^\perp$. Since $q_{\ell'}\geq-q$, it means that $\mathbb{P}(\ell^\perp)$ only intersects this half-space for lines on which $q_{\ell'}=-q$, i.e. on $\mathbb{P}\left(\ell^\perp\cap(\ell')^\perp\right)$. 
	
	\medskip
	
	The two hyperplanes that form $\Bis(q_\ell,q_{\ell'})$ are the hyperplanes $H^+$ and $H^-$ which are generated by $\ell^\perp\cap \ell'^\perp$ and respectively $\mathrm{v}+\mathrm{w}$ and $\mathrm{v}-\mathrm{w}$ where $\mathrm{v}\in \ell$ and $\mathrm{w}\in \ell'$ satisfy $q(\mathrm{v},\mathrm{v})=q(\mathrm{w},\mathrm{w})=1$. When $\ell'$ converges to $u\in \partial \mathbb{H}^n$, the intersection $\ell^\perp\cap \ell'^\perp$ converges to $\ell^\perp \cap u^\perp$ and the lines generated by $\mathrm{v}+\mathrm{w}$ and $\mathrm{v}-\mathrm{w}$ both converge to $u\subset u^\perp\setminus \ell^\perp$. Hence $H^+$ and $H^-$ both converge to $u^\perp$, so $\Bis(q_\ell,q_{\ell'})$ converges for the Hausdorff topology to $\mathbb{P}(u^\perp)$.
	
\end{proof}
\begin{figure}
	\begin{center}
		\begin{tikzpicture}
			
			\draw (0,0) circle (2);
			\draw[dotted] (-3,0) -- (6,0);
			
			\draw[blue,thick] (0.84040820577  ,-3) -- +(0,6);
			\draw[blue,thick] (4.75959179423,-3) -- +(0,6);
			
			\fill(0,0) circle (1pt);
			\node[above] () at (0, 0) {$\ell$};
			
			\fill(1.42857142857 ,0) circle (1pt);
			\node[above] () at (1.42857142857 , 0) {$\ell'$};
			
		\end{tikzpicture}
	\end{center}
	\caption{Illustration of the intersection of a Selberg bisector and $\mathbb{RP}^2$.}
	\label{fig:BissectorsSelberg}
\end{figure}

Figure \ref{fig:BissectorsSelberg} illustrates the intersection of $\Bis(q_\ell,q_{\ell'})$ and $\mathbb{RP}^n$ for $n=2$. 
The circle represents the isotropic lines for $q$. 
The two projective lines represent the intersection with $\Bis(q_\ell,q_{\ell'})$. They intersect the line at infinity $\ell^\perp$ at $\ell^\perp\cap (\ell')^\perp$. If $\ell'$ converges to $u$ on the circle, then Lemma \ref{lem:InfiniteSidedBissectors} tells us that the two blue lines converge to the tangent of the circle at $u$. 

\medskip

We can now prove Theorem \ref{thm:InfiniteSided}. 

\begin{proof}
	Suppose for the sake of contradiction that the Dirichlet-Selberg domain has finitely many sides.
	By Proposition \ref{prop:equivalent definitions of finite-sided} there is a finite set $F \subset \Gamma$ such that
	$$\mathcal{DS}_\Gamma(o)=\bigcap_{\gamma\in F}\mathcal{H}(o, \gamma\cdot o).$$
	Since $o$ is in $\mathcal{H}$, there is a unique point $\ell \in \HH^n$ so that $o = x_\ell$. 
	The interior of  $\mathcal{DS}_\Gamma(o)$ intersected with $\mathbb{P}(\ell^\perp)\subset \mathbb{RP}^n$ is the complement in $\mathbb{P}(\ell^\perp)$ of the union of $\mathcal{H}(x_{\gamma\cdot \ell},x_\ell)$ for all $\gamma\in F$. This is a finite union of hyperplanes by Lemma \ref{lem:InfiniteSidedBissectors}, hence this intersection is non-empty.
	
	\medskip
	
	Choose  a line $w \subset\ell^\perp$ in this intersection. 
	Let  $u \in \partial \HH^n$ be an isotropic line such that $w\subset u^\perp$.	
	Since the limit set of $\Gamma$ is all of $\partial \mathbb{H}^n$, there exist a sequence $(\gamma_n)$ in $\Gamma$ such that $(\gamma_n \cdot \ell)$ converges to $u$. 
	The bisectors $\Bis(o,\gamma_n o)$ converge to $u^\perp$ by Lemma \ref{lem:InfiniteSidedBissectors}. 
	In particular, these bisectors eventually meet the interior of the Dirichlet-Selberg domain. 
	Then, since the bisectors have empty interior in $\mathbb{RP}^n$, the complements of $\mathcal{H}(o,\gamma_n \cdot o)$ eventually intersect the Dirichlet-Selberg domain, yielding a contradiction. 
\end{proof}

\begin{remark}
	It is not clear if such groups admit finite-sided Dirichlet-Selberg domains for other basepoints $o\in \mathcal{X}$, so this example does not resolve Question \ref{q:Do Anosov subgroups admit finite-sided domains}. 
\end{remark}

\subsection{Properly finite-sided Dirichlet-Selberg domains}

When $G=\SL(d,\R)$, the Selberg invariant is close to the Riemannian metric for points that are close, but it also always stays at bounded distance from the Finsler distance $d_{\omega_1}$.
The identification $\X \simeq \mathcal{X}$ induces a continuous identification $\X\cup\partial_{\rm horo}^{\omega_1}X \simeq \mathcal{X}\cup \partial \mathcal{X}$, see Theorem \ref{thm:Satake=Horofunction}.

\begin{lemma}
	\label{lem:comparing Finsler and Selberg}
	Let $x_1,x_2\in \X \simeq \mathcal{X}$:
	$$d_{\omega_1}(x_1,x_2)-\log(d) \leq\mathfrak{s}(x_1,x_2)\leq d_{\omega_1}(x_1,x_2).$$
	
	Moreover for $[h]\in \X\cup\partial_{\rm horo}^{\omega_1}X \simeq \mathcal{X}\cup \partial \mathcal{X}$ and $o,x\in \X \simeq \mathcal{X}$ :
	$$h(x)-h(o)-\log(d) \leq \mathfrak{s}_o(x,[h])\leq h(x)-h(o)+\log(d).$$
\end{lemma}

\begin{proof}
	
	Let $\lambda_1,\cdots, \lambda_d$ be the eigenvalues of $x_1^{-1}x_2$. 
	Applying Proposition \ref{prop:SelbergInv} we see that:
	$$d_{\omega_1}(x_1,x_2)= \log\left(\max_{1\leq i\leq d} |\lambda_i|\right)\leq  \log\left(\sum_{1\leq i\leq d} |\lambda_i|\right) = \mathfrak{s}(x_1,x_2)+\log(d).$$
	$$d_{\omega_1}(x_1,x_2)= \log\left(\max_{1\leq i\leq d} |\lambda_i|\right)\geq  \log\left(\frac{1}{d}\sum_{1\leq i\leq d} |\lambda_i|\right) = \mathfrak{s}(x_1,x_2).$$
	
	In particular one gets for $x,y,o\in \mathcal{X}\simeq\X$:
	$$d_{\omega_1}(x,y)-d_{\omega_1}(o,y)-\log(d) \leq \mathfrak{s}_o(x,y)\leq d_{\omega_1}(x,y)-d_{\omega_1}(o,y)+\log(d).$$
	By passing to the limit as $y$ goes to $[h]$, we get the desired result.
\end{proof}

We may now deduce Theorem \ref{thm:omega1-URU implies finitely-sided} from Theorem \ref{thm:finiteSidedURU}.

\begin{corollary}\label{cor:omega1-URU implies properly finite sided}
	Let $\Gamma < \SL(2n,\R)$ be an $\omega_1$-undistorted subgroup. 
	The Dirichlet-Selberg domain $\mathcal{DS}_\Gamma(o)$ is properly finite-sided.
\end{corollary}

\begin{proof}
	Let $o\in \X$ be any basepoint. 
	Let $U$ be a neighborhood of $\mathcal{D}^\omega_\Gamma(o)$ in $\X\cup \partial_{\rm horo}^{\omega_1}X$ provided by Theorem \ref{thm:finiteSidedURU} for $A=\log(d)$. 
	There exists a finite set $F\subset \Gamma$ such that for all $[h]\in U$ and $\gamma\in \Gamma\setminus F$, $h(\gamma\cdot o)-h(o)> A$. 
	Hence $\mathfrak{s}_o(o, \gamma\cdot o)> 0$ by Lemma \ref{lem:comparing Finsler and Selberg}, so $U$ is contained in each of the projective half-spaces $\mathcal{H}(o, \gamma\cdot o)$ with $\gamma \in \Gamma \setminus F$.
	Therefore $\mathcal{DS}_\Gamma(o)$ is properly finite-sided. 
\end{proof}

We will give a second proof of Theorem \ref{thm:omega1-URU implies finitely-sided} using Theorem \ref{thm:Properly finite sided domains in Hull(F)} in Section \ref{sec:restricting ds domains}.

\begin{remark}
Du proves in \cite{du2024geometry} that the group generated by
	$$ \gamma = \begin{pmatrix}
		\lambda&&\\
		& 1 &\\
		&& \lambda^{-1} 
	\end{pmatrix}$$
	has a finite-sided Dirichlet-Selberg domain at $o \in \mathcal{X}_3$ if and only if $o$ is in the axis of $\gamma$, i.e.\ is diagonal.
	But even for such points $o\in \mathcal{X}$, the domain is not properly finite-sided. 
	One can deduce from this that an elementary subgroup of $\SL(3,\R)$ admits a properly finite-sided Dirichlet-Selberg domain if and only if it is $\omega_1$-undistorted, in which case every Dirichlet-Selberg domain is properly finite-sided.
\end{remark}

\section{Locally symmetric spaces}\label{sec:locally symmetric spaces}

In this section we consider any locally symmetric space $\mathbb{X} /\Gamma$ where $\Gamma$ is an $\omega$-undistorted subgroup of $G$. 
The Finsler distance $d_\omega$ on $\X$ descends to a natural metric on the quotient. 
We show that the horofunction compactification agrees with the quotient $(\X \cup \partialomX )/ \Gamma$.
As a consequence, we recover that such locally symmetric spaces are topologically tame, a special case of \cite[Theorem 1.4]{GKW} and \cite[Theorem 1.5.(ii)]{KL18a}. 

\medskip

Recall that the symmetric space of $\SL(V)$ embeds in $\mathbb{P}(S^2V)$ as the space $\mathcal{X}$ of positive tensors. 
The embedding $G\subset \SL(V)$ induces a totally geodesic embedding between the corresponding symmetric spaces, $\X \subset \mathcal{X}\subset \mathbb{P}(S^2V)$.

Given any $x,y\in \X\subset \mathcal{X}$ one can define the \emph{restricted Selberg invariant}. 
For this we choose positive definite
 representatives $X,Y:V^*\to V$ of $x,y$ such that $\det(X^{-1}Y)=1$, and we set:
$$\mathfrak{s}^V(x,y)=\log\Tr(X^{-1}Y) .$$ 

This restricted Selberg invariant has the following formula:

\begin{proposition}
Let $x,y\in \X$, the restricted Selberg invariant is equal to:
$$\mathfrak{s}^V(x,y)=\log\sum_{\alpha\in \Xi}n(\alpha)e^{\alpha(\vec{d}(x,y))}.$$
In this expression, $\Xi\subset \mathfrak{a}^*$ is the weight system associated to the representation $V$, and for $\alpha\in \Xi$, $n(\alpha)$ is the dimension of the associated weight space.
\end{proposition}

\begin{proof}
The restricted Selberg invariant $\mathfrak{s}^V$ is the restriction to the totally geodesic symmetric space $\X$ of the Selberg invariant defined for the symmetric space of $\SL(V)$. 
The eigenvalues of the element of $\SL(V)$ corresponding to $\exp(\mathrm{v})$ for $\mathrm{v}\in \mathfrak{a}$ are equal to $e^{\alpha(\mathrm{v})}$ for $\alpha\in \Xi$. 
Hence this formula follows from \ref{prop:SelbergInv}.
\end{proof}

This definition can be extended for any semi-positive $y\in \mathbb{P}(S^2V)$. If $y\in \mathbb{P}(S^2V)$ is semi-positive and if $o\in \X$ one can extend the Selberg invariant by taking representatives $O,X,Y$ of $o,x,y$ such that $\det(X^{-1}O)=1$:
$$\mathfrak{s}^V_o(x,y)=\log \Tr(X^{-1}Y)-\log\Tr(O^{-1}Y).$$

\begin{proposition}
The embedding $y\in \X\mapsto [\mathfrak{s}^V(\cdot,y)]\in \mathcal{Y}(\X)$ induces a horofunction compactification of $\X$, which is naturally identified with the generalized Satake compactification $\overline{\X}$ of $\X\subset \mathcal{X}$.
\end{proposition}

\begin{proof}
Indeed the map $y\in \overline{\mathbb{X}} \mapsto [\mathfrak{s}_o^V(\cdot,y)]\in \mathcal{Y}(\X)$ is an embedding from a compact space with dense image, hence it is a homeomorphism onto the horofunction compactification of $\X$.
\end{proof}

If $\Gamma\subset G$ is a discrete subgroup, one can define the Selberg invariant on $\X/\Gamma$ to be:
$$\mathfrak{s}^V(\Gamma\cdot x,\Gamma\cdot y)=\min_{\gamma\in \Gamma}\mathfrak{s}^V(x,\gamma\cdot y).$$ 

Note that this minimum is reached because the action of $\Gamma$ on $\X$ is proper.
This map also defines an embedding $y\in \X/\Gamma\mapsto \mathfrak{s}^V(\cdot,y)\in \mathcal{Y}(\X/\Gamma)$, from which one can define a horofunction compactification of $\X/\Gamma$.

\begin{theorem}
\label{thm:Comparison compactifications of locally sym space}
Let $V$ be an irreducible representation of $G$ with highest restricted weight $\omega$.
Let $\Gamma$ be a torsion-free $\omega$-undistorted subgroup of $G$. 
The horofunction compactification $\overline{\X/\Gamma}$ of $\X/\Gamma$ for the restricted Selberg invariant is naturally identified with $\left(\X\cup \Omega^\omega_{\rm horo}\right)/\Gamma$.
\end{theorem}

In other words the compactification of $\X/\Gamma$ is equal to the quotient by $\Gamma$ of a domain of discontinuity in the compactification of $\X$.

\begin{proof}

 Let $\widetilde{\phi}:\X\cup \Omega^\omega_\text{horo} \to \mathcal{Y}(\overline{\X/\Gamma})$ be the map that associates to a class of functions $[h:\X\to \R]$ the class of functions $[\min_{\gamma\in \Gamma}\gamma \cdot h]$, where $\gamma\cdot h(x)=h(\gamma^{-1}\cdot x)$. Since every horofunction in $\X\cup \Omega^\omega_\text{horo}$ is proper and bounded from below on one and hence any $\Gamma$-orbit, $\widetilde{\phi}$ is well defined. Moreover on every open set $U\subset \X\cup \Omega^\omega_\text{horo} $, there exist a finite set  $S\subset \Gamma$ such that on $U$ one has $\min_{\gamma\in \Gamma}\gamma \cdot h=\min_{\gamma\in S}\gamma \cdot h$, see Lemma \ref{lem:LocalUnif}. Hence $\widetilde{\phi}$ is continuous. Moreover the image of $\X$ by $\tilde{\phi}$ lies in $\X/\Gamma$, so the image of $\tilde{\phi}$ lies in $\overline{\X/\Gamma}$.

\medskip

The map $\widetilde{\phi}$ is also $\Gamma$-invariant by definition, so it induces a map:
$$\phi:\left(\X\cup \Omega_\text{horo}^\omega\right)/\Gamma\to \overline{\X/\Gamma}.$$
The restriction of this map to $\X/\Gamma$ is the identity: indeed the restricted Selberg invariant between $\Gamma\cdot x$ and $\Gamma\cdot y$ is equal to $\min_{\gamma\in \Gamma} \mathfrak{s}(x,\gamma\cdot y)$. Since $\widetilde{\phi}$ is continuous, so is $\phi$ and because $\left(\X\cup \Omega\right)/\Gamma$ is compact, its image is compact. Hence $\phi$ is surjective.

\medskip

It remains to show that  $\phi$ is injective.
Let $[h_1],[h_2]\in \X\cup \Omega_\text{horo}^\omega$ be such that $\widetilde{\phi}([h_1])=\widetilde{\phi}([h_2])$.
Since $h_1$ and $h_2$ are proper and bounded from below, given any compact set $K$ there exist a finite set $S_K\subset \Gamma$ such that for $i=1,2$, $\min_{\gamma\in \Gamma}\gamma\cdot h_i=\min_{\gamma\in S_K} \gamma\cdot h_i$ on $K$.

\medskip

This implies that $K$ is covered by the closed sets $K_{\gamma_1,\gamma_2}=\lbrace x\in K \mid \gamma_1\cdot h_1(x)=\gamma_2\cdot h_2(x)\rbrace$ for $\gamma_1,\gamma_2\in S_K$. 
If we take a compact set $K$ with non-empty interior, then one of these sets $K_{\gamma_1,\gamma_2}$ must have non-empty interior for some $\gamma_1,\gamma_2\in S_K$. 
Hence $\gamma_1 \cdot h_1=\gamma _2\cdot h_2$ on an open set on $\X$.
Note that $h_1$ and $h_2$ are analytic as the restriction of the $\log$ of a linear map to an analytic submanifold. 
This implies that $[h_1]=[\gamma_1^{-1}\gamma_2\cdot h_2]$ and hence the two points must correspond to the same element in $\left(\X\cup \Omega^\omega_{\rm horo}\right)/\Gamma$.

\medskip

In conclusion $\phi$ is injective, so we have proven that it induces an homeomorphism.
\end{proof}

\begin{remark}
If we try to apply the same argument for the horofunction compatification of $\X/\Gamma$ using the Finsler distance, the proof of the injectivity does not immediately apply since horofunctions are not analytic. However the map $\phi$ is still well-defined and surjective.
\end{remark}

A consequence of this result is that the locally symmetric space $\X/\Gamma$ is topologically tame. A manifold is \emph{topologically tame} if it is the interior of a compact manifold with boundary.

\begin{proposition}[{\cite[Proposition 6.1]{GGKW17b}}]
\label{prop:Tameness}
Let $X$ be a real semi-algebraic set and $\Gamma$ a torsion-free discrete group acting on $X$ by real algebraic homeomorphisms. 
Suppose $\Gamma$ acts properly discontinuously and cocompactly on some open subset $\Omega$ of $X$. Let $U$ be a $\Gamma$-invariant real semi-algebraic subset of X contained in $\Omega$. 
If $U$ is a manifold and $\overline{U}\subset \Omega$, then $U/\Gamma$ is topologically tame.
\end{proposition}

We can apply this to semi-algebraic compactifications of $\X$.
Note that the Tarski principle implies that semi-algebraic set are closed under projection. 
Hence any orbit of the algebraic action of an algebraic group on a finite dimensional vector space is semi-algebraic. 
For instance, given an irreducible representation $V$ of the semi-simple group $G$ with highest restricted weight $\omega$, the totally geodesic embedding $\X\subset \mathcal{X}$ is algebraic. 
Note also that the closure of a semi-algebraic set is semi-algebraic.

\medskip

In particular we get the following result for every $\omega$ that is the highest restricted  weight of a representation:
\begin{corollary}
\label{cor:Locally symmetric space is topologically tame}
Let $\Gamma$ be a torsion-free $\omega$-undistorted subgroup of $G$. 
The locally symmetric space $\X/\Gamma$ is topologically tame. 
\end{corollary}

This recovers a particular case of \cite[Theorem 1.4]{GKW} and \cite[Theorem 1.5.(ii)]{KL18a} since $\omega$-undistorted representations are Anosov. Note that the set of $\omega\in \mathfrak{a}$, considered up to positive scaling and the action of the Weyl group, that are the highest restricted weight of a representation is dense. Therefore any $\omega$-undistorted representation for $\omega\in \mathfrak{a}$ is $\omega'$-undistorted for such an $\omega'$, and hence the Theorem applies. 

This corollary is weaker than the result for Anosov subgroups for the following reason. 

\begin{proposition}
	Let $G$ be simple of rank at least $3$. 
	Then there exists $\alpha \in \Delta$ and a Zariski dense free subgroup $\Gamma < G$ which is $\{\alpha\}$-Anosov but not $\omega$-undistorted for any $\omega \in \mathfrak{a}^\ast$.
\end{proposition}

\begin{proof}
There exists $\alpha,\beta \in \Delta$ and $w \in W$ so that $w \alpha = \beta$ but $\beta$ is not equal to $\alpha$ nor $\iota \alpha$. 
To find this, look for an $A_3$ subdiagram of the Dynkin diagram for the restricted root system and let $\alpha$ be the middle root and let $\beta$ be either of the others.
If there is no $A_3$ subdiagram, the restricted root system has type $B_3$, $C_3$, or $F_4$, and the opposition involution is trivial, so we can take $\alpha$ and $\beta$ to be two simple roots of equal length.

Now we find a convex set $\mathcal{C}\subset \overline{\mathfrak{a}^+}$ which is invariant by the opposition involution and disjoint from $\ker \alpha$ whose orbit $W \cdot \mathcal{C}$ intersects every hyperplane in $\mathfrak{a}$. 
Fix $\epsilon >0$ and let $\mathcal{C}_\epsilon$ contain the set of unit vectors in $\mathfrak{a}^+$ at distance at least $\epsilon$ from $\Ker(\alpha)\cup \iota\Ker(\alpha)$. 
We claim that for $\epsilon$ small enough, we may take $\mathcal{C} = \mathcal{C}_\epsilon$. 
Otherwise, there is a sequence of linear forms $\omega_n$ in $\mathfrak{a}^\ast$ so that $(W \cdot \ker \omega_n) \cap \overline{\mathfrak{a}^+}$ is in the $1/n$-neighborhood of $\ker \alpha \cup \ker \iota \alpha$, and then up to passing to a subsequence we may assume that $\omega_n$ limits to a positive multiple of $\alpha$ or $\iota \alpha$. 
But then $w \cdot \omega_n$ converges to $\beta$ or $\iota \beta$ and for $n$ large enough $(W \cdot \ker \omega_n) \cap \overline{\mathfrak{a}^+}$ becomes close to $\ker \beta \cup \ker \iota \beta$.
This leads to a contradiction since $\{\alpha, \iota \alpha\}$ is disjoint from $\{\beta, \iota \beta\}$.

We can then use the construction of Benoist \cite{Benoist} as in Proposition \ref{prop:existence of omega-undistorted subgroups} to construct an undistorted free group with limit cone contained in $\mathcal{C}$ and disjoint from $\ker \alpha$.
Such a subgroup $\Gamma$ is $\alpha$-Anosov but for every $\omega \in \mathfrak{a}^\ast$, the subgroup $\Gamma$ cannot be $\omega$-undistorted.
\end{proof}

\section{Restriction of Selberg's construction}\label{sec:restricting ds domains}

In this section we consider Anosov subgroups $\Gamma\subset \SL(V)$ which are not necessarily $\omega_1$-undistorted, and we try to find a smaller domain of $\overline{\mathcal{X}}$ on which the Dirichlet-Selberg domain is nevertheless properly finite-sided.

\subsection{The general statement} 

We first introduce a general but technical statement that we will then apply to more specific situations.

\medskip

Let $G$ be a semisimple real Lie group and $V$ a finite-dimensional linear real representation of $G$.
Recall from Section \ref{sec:representation} the associated restricted weight space decomposition:
$$ V=\bigoplus_{\lambda\in \Phi} V_\lambda $$
with $\Phi$ denoting the set of restricted weights.

\medskip

More precisely for every pair of transverse full flags $f,g\in \mathcal{F}_\Delta$ for $G$, one has an identification of the model maximal abelian subspace $\mathfrak{a}\subset \mathfrak{p}$ with another $\mathfrak{a}_{f,g}\subset \mathfrak{g}$.
For each such choice we get a restricted weight decomposition:
$$V=\bigoplus_{\lambda\in \Phi} V^{f,g}_\lambda.$$

\medskip

Given a closed subset $\mathcal{C}\subset \sigma_\text{mod}=\mathbb{S} \mathfrak{a}^+$, and a subset $\Theta\subset\Delta$ of simple roots of $G$, we define the following subsets of the space of restricted weights:
$$\Phi^+_{\mathcal{C},\Theta}=\lbrace \lambda\in \Phi\mid \lambda>0 \text{ on } W_\Theta\cdot\mathcal{C}\rbrace, $$
$$\Phi^-_{\mathcal{C},\Theta}=\lbrace \lambda\in \Phi\mid \lambda<0 \text{ on } W_\Theta\cdot\mathcal{C}\rbrace, $$
$$\Phi^0_{\mathcal{C},\Theta}=\Phi\setminus \left(\Phi^+_{\mathcal{C},\Theta}\cup\Phi^-_{\mathcal{C},\Theta}\right).$$

Here $W_\Theta$ is the subgroup of the Weyl group generated by the involutions $s_\alpha$ associated to the simple roots $\alpha\in \Delta\setminus \Theta$ and $W_\Theta\cdot \mathcal{C}\subset \mathbb{S}\mathfrak{a}$.

\begin{lemma}
\label{lem: V_+^f is well defined since it is an ideal.}
Let us fix some $\mathcal{C} \subset \sigma_{mod}$ and $\Theta \subset \Delta$, and let $\xi \in \mathcal{F}_\Theta$. 
The following subspaces are independent of the choice of transverse flags $f,g\in \mathcal{F}_\Delta$ such that the simplex corresponding to $\xi$ is included in the simplex corresponding to $f$:
$$ V^\xi_+ \coloneqq \bigoplus_{\lambda\in \Phi^+_{\mathcal{C},\Theta}}V^{f,g}_\lambda, \;\; V^\xi_{\geq} \coloneqq \bigoplus_{\lambda\in \Phi^+_{\mathcal{C},\Theta}\cup \Phi^0_{\mathcal{C},\Theta}}V^{f,g}_\lambda.
$$
\end{lemma}

\begin{proof}
Note that $A=\Phi^+_{\mathcal{C},\Theta}$ or $\Phi^+_{\mathcal{C},\Theta}\cup \Phi^0_{\mathcal{C},\Theta}$ are ideals for the order relation of $\Phi$: for all $\lambda\in A$, if $\lambda'\in \Phi$ satisfies $\lambda'-\lambda\geq 0$ on $\mathfrak{a}^+$ then $\lambda'\in A$.

\medskip

If $g'$ is an other full flag transverse to $f$, then $g'=u\cdot g$ for some unipotent element of $G$ fixing $f$, which in turn is the exponential of an element $\mathrm{u}\in \mathfrak{g}$ that belongs to the sum $\mathfrak{u}_\Delta$ of all positive root spaces $\mathfrak{g}^{f,g}_\alpha$ for $\alpha$ a positive root.
However if $\mathrm{u}\in \mathfrak{g}^{f,g}_\alpha$ and $\mathrm{v}\in V^{f,g}_\lambda$ in the root space associated to the root $\alpha$ satisfies $[\mathrm{u},\mathrm{v}]\in V^{f,g}_{\lambda+\alpha}$. 
Since $A$ is an ideal, $\lambda\in A$ implies that $\lambda+\alpha\in A$.

\medskip

The fact that this is independent of the choice of $f$ is due to the fact that $A$ is $W_\Theta$-invariant.

\end{proof}

\medskip

Let $\mathcal{F}\subset \mathbb{P}(V)$ be a compact $\Gamma$-invariant subset. 
We define $S^2\mathcal{F}\subset \mathbb{P}(S^2V)$ to be the corresponding set of rank one tensors. 
We call $\Hull(S^2\mathcal{F})\subset \mathcal{X}(V)= \mathbb{P}(S^2V^{\geq 0})$ the convex hull of these points. 
We denote by $\Hull(S^2\mathcal{F})^*$ the open dual convex domain, i.e. the set of linear form that do not vanish on $S^2\mathcal{F}$.
This space contains the space $\mathcal{X}(V)^*$ of projectivizations of positive definite bilinear forms on $V$.
We note that there is a natural identification of $\mathcal{X}$ with $\mathcal{X}^\ast$ given by $[X \colon V^\ast \to V] \mapsto [X^{-1} \colon V \to V^\ast].$

\medskip

Given $[O]=o\in \Hull(S^2\mathcal{F})^*$, we define $\mathcal{DS}^\mathcal{F}_\Gamma(o)$ to be the set of elements $[X]\in\Hull(S^2\mathcal{F})$ such that for all $\gamma\in \Gamma$:
$$\text{Tr}\left(X (O-\gamma\cdot O)\right)\geq 0.$$

Here we chose the signs of the representatives of $X$ and $O$ so that $\text{Tr}(XO)>0$. 
When $o\in \mathcal{X}^*\simeq \mathcal{X}$, this domain coincides with $\mathcal{DS}_\Gamma(o^{-1})\cap \Hull(S^2\mathcal{F})$.
\medskip

\begin{definition}
	Let $[O]=o \in \Hull(S^2\mathcal{F})^*$.
	We say that $\mathcal{DS}^\mathcal{F}_\Gamma(o)$ is \emph{properly finite-sided} in $\Hull(S^2\mathcal{F})$ if there exists a neighborhood $U$ of $\mathcal{DS}^\mathcal{F}_\Gamma(o)$ in $\Hull(S^2\mathcal{F})$ and a finite set $F \subset \Gamma$ such that for all $\gamma$ not in $F$, $U$ is contained in the set 
	$$ \mathcal{H}_\mathcal{F}(o,\gamma \cdot o) \coloneqq \{ [X] \in \Hull(S^2\mathcal{F}) \mid \text{Tr}\left(X (O-\gamma\cdot O)\right)\geq 0 \},$$
	with representatives of $X$ and $O$ chosen so that $\text{Tr}(XO)>0$.
\end{definition}

\begin{theorem}
\label{thm:Properly finite sided domains in Hull(F)}
Let $\Gamma$ be a $\Theta$-Anosov subgroup of $G$, and let $\mathcal{C}=\mathcal{C}_\Gamma$. 
Let $V$ be a representation of $G$, and let $\mathcal{F} \subset \mathbb{P}(V)$ be a $\Gamma$-invariant compact subset.
Suppose that $\mathcal{F}$ is disjoint from $V^\xi_{\geq}\setminus V^\xi_+$ for all $\xi=\xi_\Theta(x)$ for $x\in \partial \Gamma$ (see Lemma \ref{lem: V_+^f is well defined since it is an ideal.}).
Then $\mathcal{DS}^\mathcal{F}_\Gamma(o)$ is properly finite-sided in $\Hull(S^2\mathcal{F})$ for all $o\in \Hull(S^2\mathcal{F})^*$.
\end{theorem}

\begin{remark}
One can consider $\mathcal{F}_\text{max}$ the largest subset of $\mathbb{P}(V)$ that avoids $V^\xi_{\geq}\setminus V^\xi_+$ for all $\xi=\xi_\Theta(x)$ for $x\in \partial \Gamma$. 
This subset is not always closed, so one can only apply our result to compact $\Gamma$-invariant subsets $\mathcal{F}\subset \mathcal{F}_\text{max}$.
\end{remark}

The proof of Theorem \ref{thm:Properly finite sided domains in Hull(F)} will be done in Section \ref{sec:Subsection where we prove the generalisation to Hull(F)}.

\subsection{The general argument} 
\label{sec:Subsection where we prove the generalisation to Hull(F)}

Throughout Section \ref{sec:Subsection where we prove the generalisation to Hull(F)} we assume that $\Gamma$ and $\mathcal{F}$ satisfy the assumptions of Theorem \ref{thm:Properly finite sided domains in Hull(F)}. 
Namely, we assume that $\Gamma$ is $\Theta$-Anosov and take $\mathcal{F}$ to be a compact $\Gamma$-invariant subset of $\mathbb{P}(V)$ disjoint from $V^\xi_{\geq}\setminus V^\xi_+$ for all $\xi=\xi_\Theta(x)$ for $x\in \partial \Gamma$ with respect to $\mathcal{C}_\Gamma$. 

For every closed subset $\mathcal{C}\subset \sigma_\text{mod}=\mathbb{S}\mathfrak{a}^+$ we define $C_{\mathcal{C},\Phi}$ to be the infimum of $\frac{\abs{\lambda(v)}}{\lVert v \rVert}$ for $v\in \mathfrak{a}$ such that $[v]\in \mathcal{C}$ and $\lambda\in \Phi^+_{\mathcal{C},\Theta}\cup \Phi^-_{\mathcal{C},\Theta}$.
We will consider sufficiently small neighborhoods $\mathcal{C}$ of $\mathcal{C}_\Gamma$ so that the sets of weights $\Phi^+_{\mathcal{C},\Theta}$ etc.\ are unchanged. 

\medskip

We will first focus on the case when the basepoint $o^{-1}$ belongs to a totally geodesic $\X\subset \mathcal{X}$ corresponding to the symmetric space of $G$, and then we will see that the result still holds for other basepoints.

\medskip

Each line $\ell \in \overline{\mathcal{X}(V)}$ defines a function $h_\ell \colon \mathcal{X} \to \R$, up to an additive constant, by setting:
$$ h_{[L]}([X]) \coloneqq \log \left( \frac{1}{d} \Tr(X^{-1}L) \right),$$ 
for representatives $X$ satisfying $\det(X^{-1}O)=1$ where $O$ is a positive definite representative of a basepoint $o \in \mathcal{X}$.  
Here $d$ is the dimension of $V$.
Each such line is a convex combination of some rank $1$ lines in $S^2V$; i.e.\ for each $L$ there exists $v_i \in V, i \in I$ such that $L = \sum v_i \otimes v_i$. 
The corresponding functions are then related by 
\begin{equation}
	h_{[L]}(x)= \log( \frac{1}{d} \sum_{i\in I} e^{h_{v_i}(x)-h_{v_i}(o)}) .
\end{equation}
At rank $1$ points, these functions are exactly the Busemann functions on $\mathcal{X}$ centered at the minimal flag manifold $\mathbb{P}(V)$.
The rest of the projective boundary $\partial \mathcal{X}$ can be interpreted as a sort of horoboundary with respect to the Selberg invariant. 
Note that these horofunctions are hence equal to, for some functions $f_i$ that are convex and $1$-Lipshitz with respect to the Riemannian metric on $\mathcal{X}$:
 $$h=\log\left( \frac{1}{d} \sum_{i=1}^d e^{f_i}\right)$$
Therefore these functions are also $1$-Lipschitz and convex. 

\medskip

If we fix $\xi\in \mathcal{F}_\Theta$, and we take a line $\ell\in \mathbb{P}(S^2V)$ exactly one of the three possibilities occur:

\begin{itemize}
\item[(a)] $\ell \in \Hull\left(S^2\mathbb{P}(V_+^\xi)\right)$
\item[(b)] $\ell \in \Hull\left(S^2\mathbb{P}(V_\geq^\xi)\right)$ and $\ell \notin \Hull\left(S^2\mathbb{P}(V_+^\xi)\right)$,
\item[(c)] $\ell \notin \Hull\left(S^2\mathbb{P}(V_\geq^\xi)\right)$.
\end{itemize}

\smallskip

The hypothesis that we put on $\mathcal{F}$ in the statement of Theorem \ref{thm:Properly finite sided domains in Hull(F)} implies that case (b) never occurs for $[\mathrm{v}]\in S^2\mathcal{F}$. 
The subset that will play the role of the thickening here will be 
$$ \text{Th}(\xi) \coloneqq \Hull\left(S^2\mathbb{P}(V_+^\xi)\right) \cap\Hull(S^2\mathcal{F}). $$
We obtain the following dichotomy:

\begin{lemma}
	\label{lem:Dichotomy2}
	Let $\ell\in \Hull(S^2\mathcal{F})\subset \mathbb{P}(S^2V)$, with $\mathcal{F}$ as in Theorem \ref{thm:Properly finite sided domains in Hull(F)}. 
	Let $o\in \mathbb{X}$ be a basepoint and let $\tau \in \mathcal{F}_\Theta$. 
	Exactly one of the following holds:
	\begin{itemize}
		\item[(i)] $\ell\in \text{Th}(\tau)$ and for every $\eta \in \st_\mathcal{C}(\tau)$ the geodesic ray $c_{o,\eta}$ satisfies 
		$$ h_\ell(c_{o, \eta}(t))-h(o)\leq -C_{\mathcal{C},\omega}t .$$ 
		\item[(ii)] $\ell\notin \text{Th}(\tau)$ and for all $\epsilon>0$ there exist $A>0$ such that for every $\eta \in \st_\mathcal{C}(\tau)$ the geodesic ray $c_{o,\eta}$ satisfies 
		$$ h_\ell(c_{o,\eta}(t))-h(o)\geq (C_{\mathcal{C},\omega}-\epsilon)t-A .$$
	\end{itemize}
\end{lemma}

\begin{proof}
We first consider the case when $\ell\in S^2\mathcal{F}$. 
We want to compute the asymptotic slope for $h_\ell$ and $\eta \in \st_\mathcal{C}(\tau)$, as defined in Section \ref{sec:Slopes and thickenings}. 
Let us fix two opposite full flags $f,g\in \mathcal{F}_\Delta$ such that such that $\eta$ belongs to the ideal Weyl chamber associated to $f$; note that $\tau$ belongs to this chamber as well. 
Let $\zeta$ be the projection of $\eta$ to $\sigma_\text{mod}=\mathbb{S}\mathfrak{a}^+$. Let $\mathrm{v}\otimes \mathrm{v}\in \ell$ be non-zero, we can decompose $\mathrm{v}$ for some $\mathrm{v}_\lambda\in V^{f,g}_\lambda$ as:
$$\mathrm{v}=\sum_{\lambda\in \Phi}\mathrm{v}_\lambda.$$

The basepoint $o\in \X\subset \mathcal{X}$ determines a norm $\lVert \cdot \rVert$ on $\R^n$, and one has:
$$ h_\ell(c_{o, \eta}(t))-h(o)=\frac{1}{d}\log\left(\left\lVert \sum_{\lambda\in \Phi}e^{\lambda(\zeta)t}\mathrm{v}_\lambda \right\rVert\right)-\frac{1}{d}\log\left(\lVert \mathrm{v}\rVert\right).$$

Since $\zeta\in \mathcal{C}$, the behavior of this quantity depends on the same case distinction as before:

\begin{itemize}
\item[(a)] if $\mathrm{v}\in \Hull\left(S^2\mathbb{P}(V_+^\xi)\right)$, for all $t\ge 0$, $h_\ell(c_{o, \eta}(t))-h(o)\leq -C_{\mathcal{C},\omega}t $ ,
\item[(b)] if $\mathrm{v}\in \Hull\left(S^2\mathbb{P}(V_\geq^\xi)\right)$ but $\mathrm{v}\notin \Hull\left(S^2\mathbb{P}(V_+^\xi)\right)$, the situation is unclear,
\item[(c)] if $\mathrm{v}\notin \Hull\left(S^2\mathbb{P}(V_\geq^\xi)\right)$, for all $\epsilon>0$ there exist $A>0$ such that for all $t>0$
 $$ h_\ell(c_{o,\eta}(t))-h(o)\geq (C_{\mathcal{C},\omega}-\epsilon)t-A .$$
\end{itemize}

Here case (b) cannot occur by the hypothesis that was put on $\mathcal{F}$. 
Note that case (a) means exactly that $\ell\in \Th(\tau)$. 
Hence we got the desired result for $\ell\in \Hull(S^2\mathcal{F})$.

\medskip

We now consider an arbitrary $\ell=[v]\in \Hull(S^2\mathcal{F})$.
Then $v$ can be written as a convex combination of extremal points of $\Hull(S^2\mathcal{F})$. 
The way we defined the associated function on $\X$ was taking the log of a linear expression, so the associated function can be written for some $p_i\in \mathcal{F}$, $\lambda_i>0$ and $o_i\in \X$ for $i\in I $ as:
$$h_\ell(x)= \frac{1}{d}\log(\sum_{i\in I}\lambda_i e^{h_{p_i}(x)-h_{p_i}(o)})=\frac{1}{d}\log(\sum_{i\in I} e^{h_{p_i}(x)-h_{p_i}(o_i)}).$$

\medskip

If $\ell\in \text{Th}(\tau)$, one can choose $(p_i)_{i\in I}$ such that for all $i\in I$ one has $p_i\in V^\tau_+$. 
Therefore for all $t\geq 0$, $b_{p_i,o}(c_{o,\eta}(t))-b_{p_i,o}(o)\leq -C_{\mathcal{C},\Phi}t$. Hence $h_\ell(c_{o,\eta}(t))-h(o)\leq -C_{\mathcal{C},\Phi}t .$

\medskip

Suppose now that $\ell\notin \text{Th}(\tau)$, then one can choose $(p_i)_{i\in I}$ such that $p_{i_0}\notin V_+^\tau$ for some $i_0\in I$, and hence $p_{i_0}\notin V_\ge^\tau$.
Therefore there exists $A>0$ such that the geodesic ray $c_{o,\eta}$ satisfies $h_{p_{i_0}}(c_{o,\eta}(t))-h_{p_{i_0}}(o)\geq (C_{\mathcal{C},\Phi}-\epsilon)t-A$.
Hence:
 $$h_\ell(c_{o,\eta}(t))-h(o)\geq (C_{\mathcal{C},\Phi}-\epsilon)t -A+\log(\lambda_{i_0}).$$
\end{proof}

\begin{remark}
A consequence of this argument is that the thickening can also be described as follows: let $p\in \Hull(S^2\mathcal{F})$, we consider the exposed face $F_p$ of the compact $\Hull(S^2\mathcal{F})$ containing $p$, i.e.\ the intersection of $\Hull(S^2\mathcal{F})$ with all support hyperplanes passing through $p$.
The point $p$ belongs to $\text{Th}(\tau)$ if and only if the extremal points of $F_p$ are all in $V^\tau_+$.
\end{remark}

\begin{definition}
Let us define the following domain:
	$$\Omega\coloneqq \Hull(S^2\mathcal{F})\setminus\bigcup_{x\in \partial\Gamma} \text{Th}(\xi_\Theta(x)).$$
\end{definition}

The following statements are the analog in this setting of Proposition \ref{prop:characterization of domains}, Lemma \ref{lem:LocalUnif}, Proposition \ref{prop:CoarseFibration} and Theorem \ref{thm:finiteSidedURU} respectively. The exact same proofs apply, by replacing $\X\cup \partial^\omega_\text{horo}\X$ by $\Hull(S^2\mathcal{F})$, $\text{Th}_\text{horo}^\omega$ by $\text{Th}$, $\mathcal{D}_\Gamma^\omega(o)$ by $\mathcal{DS}_\Gamma^\mathcal{F}(o)$ and $\Omega^\omega_\text{horo}\cup \X$ by $\Omega$. 

\begin{proposition}[{Analog of Proposition \ref{prop:characterization of domains}}]
\label{prop:uniformlyProper2}
An element $\ell\in\Hull(S^2\mathcal{F})$ belongs to $\Omega$ if and only if $h_\ell$ restricted to the $\Gamma$-orbit of $o\in \X$ is bounded from below.
In this case, $h_\ell$ is proper on any $\Gamma$-orbit. 
In particular $\mathcal{DS}^\mathcal{F}_\Gamma(o)\subset\Omega$ for all $o\in \X$.
\end{proposition}

\begin{lemma}[{Analog of Lemma \ref{lem:LocalUnif}}]
\label{lem:LocalUnif2}
Let $\ell_0\in \Omega$, and let $o\in \X$. 
There exists a neighborhood $U\subset \Omega$ of $\ell_0$ and a constant $A>0$ such that for $\ell\in U$ and $\gamma\in \Gamma$:
$$h_\ell(\gamma\cdot o)-h_\ell(o)\geq C_{\mathcal{C},\Phi}d(o, \gamma\cdot o)-A.$$  
\end{lemma}

\begin{theorem}[{Analog of Theorem \ref{thm:finiteSidedURU}}]
\label{thm:finiteSidedURU2}
For all $o\in \X$ and for any $A>0$ one can find a finite set $S\subset \Gamma$ and a neighborhood $U$ of $\mathcal{DS}_\Gamma^\mathcal{F}(o)$ such that for all $\ell\in U$ and $\gamma\in \Gamma\setminus S$, $h_\ell(\gamma\cdot o)> h_\ell(o)+A$. 
\end{theorem}

Note that in this theorem we use the fact that $\mathcal{DS}_\Gamma^\mathcal{F}(o)$ is compact as the intersection of closed spaces in a compact space. 
Hence we need here to have $\mathcal{F}$ closed.

We can now prove Theorem \ref{thm:Properly finite sided domains in Hull(F)}. 

\begin{proof}[{Proof of Theorem \ref{thm:Properly finite sided domains in Hull(F)}}]
Let $o\in \X$ and $o'\in \Hull(\mathcal{F})^*$, and fix two representatives $O^{-1}$ and $O'\colon V\to V^\ast$ of $o^{-1}$ and $o'$.
The following quantity is well-defined and continuous in $x=[X]$ on the compact set $\Hull(\mathcal{F})\subset \mathbb{P}(S^2V^{\geq 0})$, since $o^{-1},o'\in \Hull(\mathcal{F})^*$:
$$\abs{\log \abs{ \Tr (XO^{-1})}- \log \abs{\Tr (XO')}}.$$
We denote by $B$ the supremum of this quantity.

\medskip

Now we apply Theorem \ref{thm:finiteSidedURU2} for $o\in \X$ and $A=2B$, and we get that there exist a finite set $S\subset \Gamma$ and a neighborhood $U$ of $\mathcal{DS}_\Gamma^\mathcal{F}(o)$ such that for all $\ell\in U$ and $\gamma\in \Gamma\setminus S$, $h_\ell(\gamma\cdot o)> h_\ell(o)+2B$. 
This implies that the half-space $\mathcal{H}_\mathcal{F}(o', \gamma\cdot o')$ contains $U$ for all but finitely many $\gamma\in \Gamma$, so $\mathcal{DS}_\Gamma^\mathcal{F}(o')$ is properly finite-sided. 
\end{proof}

\subsection{Dirichlet-Selberg domains}
In the remainder of Section \ref{sec:restricting ds domains} we present applications of Theorem \ref{thm:Properly finite sided domains in Hull(F)}.
We first deduce Theorem \ref{thm:omega1-URU implies finitely-sided} for a second time.

\begin{corollary}\label{cor:ds domain is finite sided 2}
	Let $\Gamma$ be a subgroup of $\SL(2n,\R)$ which is $\omega_1$-undistorted. 
	Then for any $o \in \mathcal{X}$, the Dirichlet-Selberg domain $\DS_\Gamma(o)$ is properly finite-sided.
\end{corollary}

\begin{proof}
	As observed in Proposition \ref{prop:omega1ImpliesNAnosov}, if $\Gamma$ is not virtually cyclic it is  $n$-Anosov. 
	We set $\mathcal{F} = \mathbb{P}(V)$, which is clearly compact and $\Gamma$-invariant.
	The condition of being $\omega_1$-undistorted guarantees that $\Phi^0$ is empty.
	Therefore we may apply Theorem \ref{thm:Properly finite sided domains in Hull(F)}.
	Since $\overline{\mathcal{X}}=\Hull(S^2\mathcal{F})$,
	$\DS_\Gamma(o) = \DS_\Gamma^\mathcal{F}(o^{-1})$. 
	The case when $\Gamma$ is elementary is similar.
\end{proof}

\subsection{Projective Anosov subgroups}
We give two applications of the previous theorem for projective Anosov representations. 
In these examples the set $\mathcal{F}$ will depend on the representation.

\begin{theorem}
\label{thm:Projective Anosov finite sided}
Let $\Gamma$ be a projective Anosov subgroup of $\SL(d,\R)$.
Let $\Lambda$ be the projective limit set, i.e., $\Lambda=\lbrace \xi^1_\Gamma(x)|x\in \partial\Gamma\rbrace \subset \mathbb{P}(\R^d)$. 
The domain $\mathcal{DS}_\Gamma^{\Lambda}(o)$ is properly finite-sided in $\Hull(S^2\Lambda)\subset \mathbb{P}(S^2\mathbb{R}^d)$ for all $o\in \Hull(S^2\Lambda)^*$.
\end{theorem}

\begin{proof}
We apply Theorem \ref{thm:Properly finite sided domains in Hull(F)}.
Given a partial flag $\xi=(\xi^1,\xi^{n-1})$ in $\R^n$ consisting of a line and a hyperplane, the corresponding set $V_+^\xi$ is equal to $\xi^1$ and $V_\geq^\xi$ is equal to $\xi^{n-1}$.
The transversality of the boundary map $\xi_\Gamma$ implies that $\Lambda$ satisfies the hypothesis of Theorem \ref{thm:Properly finite sided domains in Hull(F)}.
\end{proof}

If $\Gamma$ is moreover convex-cocompact in the sense of \cite{DGK19} or \cite{Zim21}, we can choose $\mathcal{F}$ to be larger.

\begin{theorem}
\label{thm:Projective Anosov convex cocompact finite sided}
Let $\Gamma\subset \SL(d,\R)$ be a projective Anosov subgroup that is convex cocompact, i.e. that preserves a properly convex domain $\Omega$ and acts cocompactly on a convex set $\mathcal{C}\subset \Omega$. 
Let $\Lambda=\lbrace \xi^1_\Gamma(x)|x\in \partial\Gamma\rbrace$. 
The domain $\mathcal{DS}_\Gamma^{\mathcal{C} \cup \Lambda}(o)$ is properly finite-sided in $\Hull\left(S^2\left( \mathcal{C}\cup \Lambda\right)\right)\subset \mathbb{P}(S^2\mathbb{R}^d)$  for all $o\in \Hull(S^2\Lambda)^*$.
\end{theorem}

Note that following from \cite{DGK19,Zim21}, every projective Anosov subgroup that preserves a convex domain admits such a non-empty convex set $\mathcal{C}$.

\begin{proof} 
We can again apply Theorem \ref{thm:Properly finite sided domains in Hull(F)}. 
As previously it suffices to verify that if $x\in \partial \Gamma$, then $\mathcal{C}\cup \Lambda$ does not intersect $\xi^{n-1}_\Gamma(x)\setminus \xi^1_\Gamma(x)$. 
Since $\Omega$ is a proper convex domain preserved by the projective Anosov subgroup $\Gamma$, the hyperplane $\xi^{n-1}_\Gamma(x)$ is disjoint from $\Omega$ for all $x\in \partial \Gamma$. 
Moreover since $\Gamma$ acts cocompactly on $\mathcal{C}\subset\Omega$, one has $\overline{\mathcal{C}}=\mathcal{C}\cup \Lambda$. 
Complete details for the proofs of the previous two sentences can be found in \cite[Section 8]{DGK19}.
\end{proof}

\subsection{\texorpdfstring{$\omega$}{omega}-undistorted subgroups through a representation}\label{sec:URU subgroups through a representation}
We consider a semisimple Lie group $G$ and an irreducible finite dimensional representation $V$. 
We construct a subset $\mathcal{F}$ that satisfies the hypothesis of Theorem \ref{thm:Properly finite sided domains in Hull(F)} for $\omega$-undistorted subgroups.

\medskip

Let $I\subset \Phi$ be an ideal and write $V^f_I = \bigoplus_{\lambda \in I} V_\lambda^f$ for $f \in \mathcal{F}_\Delta$. 
Set 
$$ \mathcal{F}_I \coloneqq \left\{ [v] \mid v \in V^f_I \setminus \{0\}, f \in \mathcal{F}_\Delta \right\} \subset \mathbb{P}(V) .$$

\begin{theorem}
\label{thm:restiction applied for URU representations in G}
Let $\Gamma\subset G$ be an $\omega$-undistorted subgroup for all $\omega\in I$. 
The domain $\mathcal{DS}_\Gamma^{\mathcal{F}_I}(o)$ is properly finite-sided in $\Hull(S^2\mathcal{F}_I)$ for all $o\in \Hull(S^2\mathcal{F}_I)^*$.
\end{theorem}

An example of an ideal can always be obtained by taking the highest restricted weight $\lbrace \omega\rbrace \subset \Phi$. In this case, for $G=\SL(d,\R)$ with the standard representation on $\R^d$, we recover Theorem \ref{thm:omega1-URU implies finitely-sided}. 

\begin{proof}
Note that $\mathcal{F}$ is closed. For each $\lambda\in I$, we get a set $\Theta_\lambda \subset \Delta$ such that $\Gamma$ is $\Theta$-Anosov. We fix $\Theta$ to be the union of all these sets.

\medskip

Let $[v]\in \mathcal{F}$ and let $x\in \partial \Gamma$. For some full flag $f\in \mathcal{F}_\Delta$, one has $v\in V^f_I$. One can find an opposite full flag $g\in \mathcal{F}_\Delta$ such that the flat determined by $f,g$ contains $\xi_\Gamma(x)$. For all $w\in W$, either $w\cdot \omega>0$ on $W_\Theta\cdot \mathcal{C}$ or $w\cdot \omega<0$ on $W_\Theta\cdot \mathcal{C}$. Hence $v$ does not belong to $V^\xi_{\geq}\setminus V^\xi_+$, so $\mathcal{F}$ satisfies the hypothesis of Theorem \ref{thm:Properly finite sided domains in Hull(F)}.
\end{proof}

One can apply Proposition \ref{prop:Tameness} to the quotient of this convex hull.

\begin{corollary}
Let $\Gamma$ be a torsion-free subgroup of $G$ that is $\omega$-undistorted for all $\omega\in I$. 
The quotient by $\Gamma$ of the relative interior of $\Hull(S^2\mathcal{F}_I)$ is topologically tame. 
\end{corollary}

Indeed $S^2\mathcal{F}_I$ is the orbit of an algebraic set by an algebraic group and hence it is semi-algebraic. 
Moreover the convex hull in a finite dimensional vector space of a semi-algebraic set is also semi-algebraic: given a semi-algebraic set $A\subset \R^n$ we consider the subset $B\subset (\R^n)^{n+1}\times \R^{n+1}\times \R^n$ of elements $((x_0,\dots,x_n),(\lambda_0,\dots,\lambda_n),x)$ such that $x_0,\dots, x_n\in A$, $\lambda_0,\dots,\lambda_n \ge0$, $\lambda_0 + \cdots \lambda_n=1$ and $x=\lambda_0x_0+\cdots +\lambda_n x_n$. 
By Carath\'{e}odory's theorem, projecting $B$ to the last copy of $\R^n$ yields the convex hull.

\medskip

Let $\Gamma$ now be an $\omega_\Delta$-undistorted subgroup of $G=\SL(d,\R)$, equivalently a $\Delta$-Anosov subgroup, or sometimes called a Borel Anosov subgroup. 
Recall that $\omega_\Delta$ is the highest root, i.e.\ the highest weight of the adjoint representation $V = \mathfrak{sl}(d,\R)$.  
The highest weight space is given by the span of the unit matrix $E_{1,d}$. 
The set $\mathcal{F}_{\lbrace \omega_\Delta\rbrace}$ can be identified with the flag manifold $\mathcal{F}_{\omega_\Delta}$ which is the partial flag manifold of lines in hyperplanes in $\R^d$. 
The identification is given by $([v],[\alpha]) \mapsto [v \otimes \alpha] \in \mathbb{P}(\mathfrak{sl}(d,\R))$.

Therefore Theorem \ref{thm:restiction applied for URU representations in G} implies that one can obtain a properly finite sided domain in the convex hull of $S^2\mathcal{F}_{\lbrace \omega_\Delta\rbrace}\subset \mathbb{P}(S^2\mathfrak{sl}(n,\R))$ for every Borel Anosov subgroup of $\SL(d,\R)$.

\medskip
	
However one can apply this theorem to a larger ideal, and for any semi-simple Lie group $G$ with its adjoint representation $V=\mathfrak{g}$. We consider $I=\Sigma^+$ the set of positive roots.
The nilpotent cone $\mathcal{N}=\mathcal{F}_{I}\subset \mathbb{P}(\mathfrak{g})$ is the closed cone of nilpotent elements in the Lie algebra $\mathfrak{g}$ of $G$. 

\begin{corollary}
\label{cor:RestrictionBorel Anosov}
If $\Gamma\subset G$ is $\Delta$-Anosov, the Dirichlet-Selberg domain $\mathcal{DS}^\mathcal{N}_\Gamma(o)$ is properly finite-sided in $\Hull(S^2\mathcal{N})\subset \mathbb{P}(S^2\mathfrak{g})$ for all $o\in \Hull(S^2\mathcal{N})^*$. 
\end{corollary}

\newpage

\bibliographystyle{hamsalpha}
\bibliography{biblio}
	
\end{document}